\documentclass{article}
\usepackage[english]{babel}
\usepackage{geometry,amsmath,amssymb,enumerate,bbm,latexsym,xcolor,theorem}
\usepackage[square,numbers]{natbib}
\newcommand{\Iota}{\mathrm{I}}
\newcommand{\assign}{:=}
\newcommand{\backassign}{=:}
\newcommand{\cdummy}{\cdot}
\newcommand{\mathd}{\mathrm{d}}
\newcommand{\nobracket}{}

\newcommand{\tmop}[1]{\ensuremath{\operatorname{#1}}}
\newcommand{\tmstrong}[1]{\textbf{#1}}
\newcommand{\tmtextbf}[1]{\text{{\bfseries{#1}}}}
\newcommand{\tmtextit}[1]{\text{{\itshape{#1}}}}
\newenvironment{enumeratealpha}{\begin{enumerate}[a{\textup{)}}] }{\end{enumerate}}
\newenvironment{enumerateroman}{\begin{enumerate}[i.] }{\end{enumerate}}
\newenvironment{proof}{\noindent\textbf{Proof\ }}{\hspace*{\fill}$\Box$\medskip}
\newenvironment{proof*}[1]{\noindent\textbf{#1\ }}{\hspace*{\fill}$\Box$\medskip}
\newcounter{nnacknowledgments}

{\theorembodyfont{\rmfamily}\newtheorem{acknowledgments*}[nnacknowledgments]{Acknowledgments}}
\newtheorem{corollary}{Corollary}
\newtheorem{definition}{Definition}
\newtheorem{lemma}{Lemma}
\newtheorem{notation}{Notation}
\newtheorem{proposition}{Proposition}
{\theorembodyfont{\rmfamily}\newtheorem{remark}{Remark}}
\newtheorem{theorem}{Theorem}
\newcommand{\tmkeywords}{\textbf{Keywords:} }
\newcommand{\tmmsc}{\textbf{A.M.S. subject classification:} }

\begin{document}

\title{Invariant Gibbs measure for Anderson nonlinear wave equation}

\author{
  Nikolay Barashkov \thanks{Department of
  	Mathematics and Statistics, University of Helsinki, email:\tmtextit{nikolay.barashkov@helsinki.fi}} \and Francesco C. De Vecchi\thanks{Dipartimento di Matematica,
  	Universit{\`a} degli Studi di Pavia, email:\tmtextit{francescocarlo.devecchi@unipv.it}} \and Immanuel Zachhuber\thanks{Institut f{\"u}r Mathematik, FU
  	Berlin, email:\tmtextit{immanuel.zachhuber@fu-berlin.de}}
  }
  \date{}

\maketitle

\begin{abstract}
  We study the Gaussian measure whose covariance is related to the Anderson
  Hamiltonian operator, proving that it admits a regular coupling to the
  (standard) Gaussian free field exploiting the stochastic optimal control
  formulation of Gibbs measures. Using this coupling, we define the
  renormalized powers of the Anderson free field and we prove that the
  associated quartic Gibbs measure is invariant under the flow of a nonlinear
  wave equation with renormalized cubic nonlinearity.
  
  \ 
\end{abstract}

\tmmsc{60H30, 35L71 (Primary); 28C20, 60L40}

\tmkeywords{Anderson Hamiltonian, singular stochastic PDEs, nonlinear wave
equation, Gaussian measures, paracontrolled distributions}

{\tableofcontents}

\section{Introduction}

In this paper we want to study the invariant measure of the nonlinear Anderson
wave equation, namely we consider the hyperbolic PDE
\begin{equation}
  \partial_t^2 u (t, x) + \text{``$H^{\omega} u (t, x)$''} = - \text{``$u^3
  (t, x)$''}, \quad t \in \mathbb{R}, x \in \mathbb{T}^2, \omega \in \Omega
  \label{eq:main1}
\end{equation}
where $H^{\omega}$ is the operator usually called \tmtextit{Anderson
Hamiltonian} (formally) defined as a singular Schr{\"o}dinger operator
\[ H^{\omega} u (t, x) \text{``$=$''} - \Delta u (t, x) + \xi
   (\omega, x) u (x, t) \]
where $\xi (\cdot, x) : \Omega \rightarrow \mathcal{S}' (\mathbb{T}^2)$ is a
Gaussian white noise defined on some probability space $(\Omega, \mathcal{F},
\mathbb{P})$, and we aim to prove that the (formal) measure
\begin{equation}
  \exp \left( - \frac{1}{2} (\varphi, H^{\omega} (\varphi)) \mathd x -
  \frac{1}{2} \int_{\mathbb{T}^2} | \partial_t \varphi |^2 \mathd x - \frac{1}{4}
  \int_{\mathbb{T}^2} \varphi^4 \mathd x \right) \mathcal{D} (\varphi,
  \partial_t \varphi) \label{eq:heuristicmeasure},
\end{equation}
where $\mathcal{D} (\varphi, \partial_t \varphi)$ is some (non-existent)
Lebesgue measure on a function space, is invariant for the flow of equation
\eqref{eq:main1}. The above problem combines two different aspects of
hyperbolic (stochastic) PDEs: the invariant measure of infinite dimensional
dynamical systems, and the dispersive PDEs with singular stochastic
potentials. Note that in the absence of the noise
$\xi,$\eqref{eq:heuristicmeasure} would amount to the well-studied $\Phi_2^4$
measure {\cite{BG2020,DaP2003,GlimmJaffe,webermou2d,Simonbookphi}} in the
$\varphi$ integration variable and the white noise Gaussian measure in the
$\partial_t \varphi$ variable.

The $\Phi_2^4$ measure whose quadratic part is given by
the Anderson Hamiltonian may have applications in the study of suitable
continuum limits of Ising models with random bonds (see e.g.
{\cite{Chen2014,CometsGiacominLebowitz1999,CometsGuerraToninelli2005,Shalaev1994}},
see also
{\cite{GlimmJaffe,grazieschi2023dynamical,MourratWeber2017,Simonbookphi}} for
the relation between Ising models with $\Phi^4$ measures).

The study of invariant measures of Hamiltonian PDEs is at this point a
classical problem in dispersive PDEs, having been studied intensively after
the seminal papers by Bourgain from the 90's on nonlinear Schr{\"o}dinger and
KdV equations {\cite{bourgain1994periodic,bourgain1996invariant}} which itself
came after {\cite{LeRoSp}} and
{\cite{albeverio1990global}}. The analogue of the current result was shown in
{\cite{OhLa2020}}, i.e. the invariance of the Gibbs measure under the flow of
the Wick ordered cubic wave equation. Later this methodology was applied to
closely related topics like \tmtextit{almost sure well-posedness}
{\cite{BZ08I,BZ08II}} and \tmtextit{quasi-invariance} in settings when there
is no invariant measure {\cite{OhTzquasi,GOTW2022}} to name just a few. \

\

Due to the low regularity of the support of the (formal) invariant measure,
it is necessary to \tmtextit{Wick order} the cubic nonlinearity in
\eqref{eq:main1} which naively means that we have to subtract a diverging
linear counter term to make it well-defined (we can thus think of it as being
akin to \tmtextit{renormalization}), hence writing the cube is just formal.
This notion is similar to that appearing for example in
{\cite{bourgain1996invariant}} and {\cite{OhLa2020}}, where the Wick ordering
was done with respect to the Gaussian free field, as the invariant measure is
mutually absolutely continuous with respect to it. Our case is somewhat
analogous, except that our reference Gaussian measure is not the free field
but the Gaussian measure with covariance (formally) given by $(H^{\omega})^{-
1}$.

\

The second renormalization needed in equation \eqref{eq:main1} is  in the
rigorous definition of the Anderson Hamiltonian operator $H^{\omega}$. The
issue is that the very low regularity of the noise $\xi$ makes it necessary to
subtract an infinite correction term to make sense of the product between $u$
and $\xi .$ \

The parabolic equation whose linear part is given by the Anderson
Hamiltonian, called \tmtextit{parabolic Anderson model or PAM,} has many
applications, see e.g. {\cite{Kbook2016}} and the references therein. It was
studied both in the linear case (see, e.g.,
{\cite{CGP2017,HL2015,hairerlabbe3d,MP2019}}) and in its nonlinear
generalization (see, e.g., the gPAM equation
{\cite{CFG2017,CF2018,GIP2015,H2014}} or the rough super-Brownian motion
equation {\cite{PR2021}}) in two and three dimensions. More recently wave and
Schr{\"o}dinger equations with Anderson Hamiltonian linear part were
considered (see, e.g., {\cite{allez_continuous_2015}} for the linear case and
{\cite{DeWe2018,GUZ,mouzardstrichartz,TV2022(2),TV2022(1),Ugu2022,zachhuber2019strichartz,Z2021}}
for the nonlinear one, see also
{\cite{bailleul2022analysis,CvZ,KPvZ2022,labbe,mouzard2020weyl,matsuda2022integrated,matsuda2022anderson}}
for problems relating to the spectral properties of the Anderson Hamiltonian).

\

The first paper to construct the Anderson Hamiltonian rigorously as a
semi-bounded self-adjoint operator on $L^2 (\mathbb{T}^2)$ together with an
explicit description of its domain and a bound on the lowest eigenvalue was
{\cite{allez_continuous_2015}}. In particular, therein it was showed that one
can rigorously define it as a limit (taken in the norm resolvent sense) of
regularized operators via
\begin{equation}
  \mathbb{H}^{\omega} \assign \lim_{\varepsilon \rightarrow 0} (- \Delta +
  a_{\varepsilon} \ast \xi (\omega) - c_{\varepsilon}) \label{eq:AndersonH1},
\end{equation}
where $a_{\varepsilon} (x) \assign \varepsilon^{- 2} a (\varepsilon^{- 1} x)$
is a smooth mollifier, and $c_{\varepsilon} \in \mathbb{R}_+$ is a suitable
diverging sequence of real numbers (not depending on $\omega$). This approach
was reformulated and extended in {\cite{GUZ}} to 3 dimensions, where also some
associated semilinear SPDEs were solved. This formulation is in fact the
approach we follow in this work. Let us also mention the work {\cite{labbe}}
which deals with the construction of the operator with both periodic and
Dirichlet boundary conditions and {\cite{mouzard2020weyl}} where the operator
is defined on compact surfaces and also a Weyl law is proved for its
eigenvalues.

\

Let us emphasize the fact that for almost every realisation $\omega$ one can
define $\mathbb{H}^{\omega}$ as a self-adjoint operator on $L^2$ which is
bounded from below by a constant which of course also depends on $\omega$,
i.e.
\[ (u, \mathbb{H}^{\omega} u) \geqslant - K (\omega) \| u \|^2_{L^2}, \]
and one can quantify the (measurable) dependence, indeed it can be chosen as a
suitable norm of the enhanced noise, see Proposition 2.9 in
{\cite{mouzard2020weyl}} for an explicit expression of the constant. Often we
will shift the operator by the constant $K (\omega) + 1$ and define
\begin{equation}
  \mathbb{H}^{\omega, K} \assign \mathbb{H}^{\omega} + K (\omega) + 1
  \label{eqn:introHK}
\end{equation}
so as to make it uniformly positive and we can take fractional and negative
powers with impunity.

\

We consider the regularized equation, the motivation for which will be given
in Section \ref{sec:global}.
\begin{eqnarray}
  \partial_t^2 u_{\varepsilon} (t, x)&=&-\mathbb{H}^{\omega, K}
  u_{\varepsilon} (t, x) - \rho_{\varepsilon} \ast (\rho_{\varepsilon} \ast
  u_{\varepsilon})^3 (t, x) + (3 a_{\varepsilon} + 1 + K (\omega))
  \rho_{\varepsilon}^{\ast 2} \ast u_{\varepsilon} (t, x), \quad 
  \label{eq:main2}\\
  (u_{\varepsilon}, \partial_t u_{\varepsilon}) |_{t = 0}&=&(u_0, u_1)
  \nonumber
\end{eqnarray}
where $t \in \mathbb{R}$, $x \in \mathbb{T}^2$, $\rho_{\varepsilon} :
\mathbb{T}^2 \rightarrow \mathbb{R}_+$ is a symmetric mollifier and \
$a_{\varepsilon} = \sum_{k \in \mathbb{Z}^2} \frac{| \hat{\rho}_{\varepsilon}
(k) |^2}{| k |^2 + 1} = \tmop{Tr}_{\mathbb{T}^2} (\rho_{\varepsilon} \ast (-
\Delta + 1)^{- 1} \ast \rho_{\varepsilon})$, where the operator
$\mathbb{H}^{\omega, K}$ (and the corresponding linear evolution on $L^2
(\mathbb{T}^2)$) is defined as in equations \eqref{eq:AndersonH1} and
\eqref{eqn:introHK}.

\

First we prove a rigorous version of the heuristically defined measure
\eqref{eq:heuristicmeasure}.

\begin{proposition}
  \label{proposition:existencemeasure}Let $0 \leqslant K (\omega) < + \infty$
  be such that $\mathbb{H}^{\omega, K} =\mathbb{H}^{\omega} + K (\omega) + 1$
  is a (strictly) positive operator and let $\mu^{\mathbb{H}^{\omega, K}}$ be
  the law of the Gaussian field with covariance $(\mathbb{H}^{\omega, K})^{-
  1}$. Then, for almost every $\omega \in \Omega$, the (weak) limit of probability
  measures
  \begin{eqnarray*}
    \nu^{\omega} (\mathd \varphi) & \assign& \lim_{\varepsilon \rightarrow 0}
    \nu^{\omega, \varepsilon} (\mathd \varphi)\\
    & \assign& \lim_{\varepsilon \rightarrow 0} \frac{e^{- \frac{1}{4}
    \int_{\mathbb{T}^2} | \rho_{\varepsilon} \ast \varphi (x) |^4 \mathd x +
    \frac{(3 a_{\varepsilon} + K (\omega) + 1)}{2} \int_{\mathbb{T}^2} |
    \rho_{\varepsilon} \ast \varphi (x) |^2 \mathd
    x}}{\mathcal{Z}_{\varepsilon}} \mu^{\mathbb{H}^{\omega, K}} (\mathd
    \varphi) \otimes \mu^{\mathbb{I}_{L^2}} (\mathd \partial_t \varphi),
  \end{eqnarray*}
  where $\mathcal{Z}_{\varepsilon}$ is the normalization constant, exists
  in $\mathcal{P} (H^{- \kappa} (\mathbb{T}^2) \times H^{- 1 -
  \kappa} (\mathbb{T}^2))$ for $\kappa > 0$ and $\mu^{\mathbb{I}_{L^2}}$ is
  the white noise measure.
\end{proposition}

Having the candidate invariant measure $\nu^{\omega},$ we prove that it is
invariant with respect to the limit of the flow defined by \eqref{eq:main2}.
More precisely:

\begin{theorem}
  Using the notation of Proposition \ref{proposition:existencemeasure}, for
  almost all $\omega \in \Omega$, there is a set of full measure
  $\mathcal{A}^{\omega} \subset H^{- \kappa} (\mathbb{T}^2) \times H^{- 1 -
  \kappa} (\mathbb{T}^2)$ with respect to $\mu^{\mathbb{H}^{\omega, K}}
  \otimes \mu^{\mathbb{I}_{L^2}}$, for which equation \eqref{eq:main2} has a
  unique global in time solution, leaving the measure $\nu^{\omega,
  \varepsilon}$ invariant. Furthermore, for any $(u_0, \partial_t u_0) \in
  \mathcal{A}^{\omega}$, the solution $u_{\varepsilon}$ \eqref{eq:main2}
  admits a unique limit defining a one parameter group of maps on
  $\mathcal{A}^{\omega}$. The measure $\mu^{\omega}$ (defined as in
  Proposition \ref{proposition:existencemeasure}) is an invariant measure with
  respect to that one parameter group of maps.
\end{theorem}

It is important to note that the counter term $a_{\varepsilon}$ used in the
renormalization of the measure $\nu^{\omega, \varepsilon}$ and of the equation
\eqref{eq:main2} is the same one appearing in the renormalization of classical
Wick ordered wave equation (see {\cite{OhLa2020}}), and thus $a_{\varepsilon}$
is independent of $\omega \in \Omega$. In other words we do not use the Wick
renormalization with respect to the Gaussian measure $\mu^{\mathbb{H}^{\omega,
K}}$ (as one might expect), whose counter term is
\[ \tilde{a}_{\varepsilon} (\omega, x) \assign \int(\rho_{\varepsilon} \ast
   \varphi (x))^2 \mu^{\mathbb{H}^{\omega, K}} (\mathd \varphi), \]
which depends on both $x \in \mathbb{T}^2$ and $\omega \in \Omega$. The
possibility of choosing $a_{\varepsilon}$ constant is nontrivial since we are
practically saying that we can use the (standard) free field Wick cube in
order to define the nonlinearity ``$u^3$'' in equation \eqref{eq:main1}, while
the Wick power with respect to the free field is not a priori well defined on
a set of measure one with respect to $\mu^{\mathbb{H}^{\omega, K}}$ since
$\mu^{\mathbb{H}^{\omega, K}}$ are mutually singular for almost all $\omega$.

\

The reason why we can nonetheless choose $a_{\varepsilon}$ constant is due to
the fact that, despite their mutual singularity, we can obtain the Gaussian
field with law $\mu^{\mathbb{H}^{\omega, K}}$ as a Gaussian free field
translated by a shift which has regularity just below that of the Cameron
Martin space. More precisely,if we denote by $\mu^{- \Delta}$ the measure on
$H^{- \varepsilon} (\mathbb{T}^2)$ of the Gaussian free field with covariance
$(- \Delta + 1)^{- 1}$, we obtain the following result.

\begin{theorem}
  \label{theorem:shift}For almost all $\omega \in \Omega$, there is a coupling
  $\eta^{\omega} \in \mathcal{P} (H^{- \varepsilon} (\mathbb{T}^2) \times H^{-
  \varepsilon} (\mathbb{T}^2))$ between the Gaussian measures
  $\mu^{\mathbb{H}^{\omega, K}} \in \mathcal{P} (H^{- \varepsilon}
  (\mathbb{T}^2))$ and $\mu^{- \Delta} \in \mathcal{P} (H^{- \varepsilon}
  (\mathbb{T}^2))$ such that if $(\varphi^{\mathbb{H}^{\omega, K}}, \varphi^{-
  \Delta}) \sim \eta^{\omega}$ then $\varphi^{\mathbb{H}^{\omega}_A} -
  \varphi^{- \Delta} \in H^{1 - \delta} (\mathbb{T}^2)$ almost surely for any
  $\delta > 0$.
\end{theorem}

The proof of Theorem \ref{theorem:shift} is based on the variational
formulation of Gibbs measure first proved by Bou{\'e} and Dupuis
{\cite{BD1998}} (see also {\cite{U2014,Z2009}}) and then applied first to
quantum field theory in {\cite{BG2020}} and used in
{\cite{BG2022,Barashkovwholespace,BG2021,BaGuHo2022,BauHof2022,CGW2022,OhSeTo2020}}
for studying measures related to quantum field theory. This technique was
already used in the context of invariant measures of PDEs in
{\cite{B2022,BDNY2022,OOT2020}} where the invariant measure considered was
singular with respect to the Gaussian free field. Thanks to the decomposition
proved in Theorem \ref{theorem:shift}, by writing
$\varphi^{\mathbb{H}^{\omega, K}} - \varphi^{- \Delta} = h \in H^{1 -
\delta}$, we are able to prove that
\[ (\varphi^{\mathbb{H}^{\omega, K}})^{\circ 3} \assign \lim_{\varepsilon
   \rightarrow 0} (\rho_{\varepsilon} \ast (\varphi^{\mathbb{H}^{\omega,
   K}}))^3 - 3 \alpha_{\varepsilon} \rho_{\varepsilon} \ast
   \varphi^{\mathbb{H}^{\omega .K}} = : (\varphi^{- \Delta})^3 : + 3 :
   (\varphi^{- \Delta})^2 : \cdot h + 3 \varphi^{- \Delta} \cdot h^2 + h^3 \]
holds, where $: (\varphi^{- \Delta})^3 :$ and $: (\varphi^{- \Delta})^2 :$ are
the standard Wick powers of Gaussian free field $\varphi^{- \Delta}$, and the
products $: (\varphi^{- \Delta})^2 : \cdot h$ and $\varphi^{- \Delta} \cdot
h^2$ are well-defined and lie in $H^{- \varepsilon}$. This allows us to write
\begin{eqnarray}
  \partial_t^2 u  (t, x) &= & -\mathbb{H}^{\omega, K} u (t, x) - u ^{\circ 3}
  (t, x) \label{eqn:intro:NLW} \\
  (u , \partial_t u) |_{t = 0} &= & (u_0, u_1) \nonumber
\end{eqnarray}
as a rigorous version of \eqref{eq:main1} and the limit of \eqref{eq:main2}.
We think that Theorem \ref{theorem:shift} can be of independent interest,
since it is useful to understand better the Gaussian measure with covariance
$(\mathbb{H}^{\omega, K})^{- 1}$ and the related Wick product (see also
{\cite{bailleul2022analysis}} where the Wick square of the Anderson Gaussian
free field appears in the description of the polymer measure).
As a byproduct of the proof of Theorem
\ref{theorem:shift}, we also obtain a different proof of the existence of the
Anderson Hamiltonian operator using only the variational formulation of
{\cite{BG2020}}, see Remark \ref{rem:variationalAH}.

\begin{remark}
  \label{rem:scaletoscale}We expect the coupling of Theorem
  \ref{theorem:shift} to have the ``scale to scale property'' which means that
  on large scales the coupling is independent from the underlying GFF on small
  scales. In particular $\Delta_i h$ and $\Delta_j \varphi^{- \Delta}$ are
  independent if $j > i$. In {\cite{BauHof2022,BaGuHo2022}} such a
  coupling was used, to establish that the recentred maximum of the associated
  log-correlated field (in this case $\mu^{\mathbb{H}^{w, K}}$) behaves
  similarly to the Gaussian free field {\cite{Ding_2017}}. 
\end{remark}

The paper is organized as follows: In Section \ref{sec:AHamilton} we recall
the rigorous construction of the operator $\mathbb{H}^{\omega}$ generally
following {\cite{GUZ}} and collect some salient results about it. Next we
provide some results on Gaussian measures on function spaces in Section
\ref{sec:Wick}. Section \ref{section:coupling} is dedicated to the
construction of the coupling between $\mu^{- \Delta}$ and
$\mu^{\mathbb{H}^{\omega, K}}$, namely the proof of Theorem
\ref{theorem:shift} which allows us to define the Wick powers. Section
\ref{sec:localsol} details the local-in-time well-posedness theory of the Wick
ordered Anderson wave equation \eqref{eqn:intro:NLW} as well as the
convergence of approximations. Finally Section \ref{sec:global} combines the
local well-posedness and the Hamiltonian structure of the equation to prove
invariance of the Gibbs measure via a Bourgain-type argument.
Finally, in Appendix \ref{app:AGFFdiff} we give an
alternative proof of the coupling following ideas from
{\cite{bailleul2022analysis}} which works well for Gaussian measures (which is
sufficient for our current setting) but is less general than the method from
Section \ref{section:coupling} which does not assume Gaussianity.

\

{\tmstrong{Notation}}: We frequently use function spaces which are
either Lebesgue spaces denoted, as per usual, by $L^p$ i.e. $u \in L^p$ if $\|
u \|^p_{L^p} = \int_{\mathbb{T}^2 } | u |^p < \infty$ with the usual modification for $p =
\infty$ or Sobolev spaces $H^{\sigma}$ whose norm is given by $\| v
\|_{H^{\sigma}} = \left\| (1 - \Delta)^{\frac{\sigma}{2}} v \right\|_{L^2} .$
Moreover, we frequently employ the Besov spaces $B_{p, q}^s$ and the
H{\"o}lder-Besov spaces $\mathcal{C}^s = B_{\infty \infty}^s$ whose definition
together with related concepts is recalled for the reader's convenience in
Appendix \ref{app:Besov}. As we are exclusively working on the two-dimensional
torus $\mathbb{T}^2 = (\mathbb{R}/\mathbb{Z})^2$ we sometimes write
$H^{\sigma} = H^{\sigma} (\mathbb{T}^2)$ etc and occasionally, when we want to
differentiate between space and time, write $L^p ([0, T], H^{\sigma})$ for the
space of space-time functions $u (t, x)$ with finite norm $\| u \|_{L^p ([0,
T], H^{\sigma})} \assign \left( \int^T_0 \left\| u (t)
\right\|^p_{H^{\sigma}}  \right)^{\frac{1}{p}} < \infty$ and sometimes we
use short-hand notations such as $L_t^p H_x^{\sigma}$ to differentiate between
the variables.\\

Furthermore, we frequently use the notation $\lesssim$ to mean a bound up to
an implicit constant that may change from line to line, relatively we use
$\lesssim_{\rho}$ to mean a bound that may depend on $\rho$ explicitly.
Similarly $C, c > 0$ may frequently denote implicit constants which we allow
to change from line to line and $C (\alpha)$ may denote a changing constant
with a dependence on the quantity $\alpha$.\\

As a general rule, we write $\mathbb{E}$ for an expectation and $\mathbb{P}$
for a probability. More concretely we have two different probability spaces,
which we denote by $(\Omega, \mathcal{F}, \mathbb{P})$ for the definition of
the Anderson Hamiltonian (which we construct for almost every $\omega \in
\Omega$) and $(\Omega', \mathcal{F}', \mathbb{P}')$ for other randomnesses
appearing after Section \ref{sec:prel}. On occasion we will use notations like
$\mathbb{E}^{\omega}, \mathbb{E}^{\omega'}$for expectations w.r.t. those
probabilities and $\mathbb{E}_{\mu}$ as the expectation w.r.t. a probability
measure $\mu$ etc.

\begin{acknowledgments*}
  The first author gratefully acknowleges support by the European Research
  Council Grant 741487 ``Quantum Fields and Probability''. The second author
  was partially supported by INdAM (Istituto Nazionale di Alta Matematica,
  Gruppo Nazionale per l'Analisi Matematica, la Probabilit{\`a} e le loro
  Applicazioni and Gruppo Nazionale per la Fisica Matematica), Italy. The last
  author would like to thank Ismael Bailleul and Antoine Mouzard for some
  helpful discussions that led to the alternative proof of coupling of the
  Gaussian measures in Appendix \ref{app:AGFFdiff}.
\end{acknowledgments*}

\section{Preliminaries}\label{sec:prel}

\subsection{Anderson Hamiltonian and Paracontrolled
Calculus}\label{sec:AHamilton}

In this section we briefly recall the salient properties of the
\tmtextit{Anderson Hamiltonian} on $\mathbb{T}^2$ which will be required in
the following sections. We will largely follow {\cite{GUZ}} where the
interested reader can also find the details omitted here.

Formally we can see the (continuum) Anderson Hamiltonian on $\mathbb{T}^2$ as
a Schr{\"o}dinger operator with a spatial white noise potential, i.e.
\begin{equation}
  H^{\omega}  \text{``$=$''} - \Delta + \xi (\omega, \cdummy), \label{Aformal}
\end{equation}
where $\xi : \Omega \rightarrow \mathcal{S}' (\mathbb{T}^2)$ (where $(\Omega,
\mathcal{F}, \mathbb{P})$ is a probability space and $\mathcal{S}'
(\mathbb{T}^2)$ is the space of distributions on $\mathbb{T}^2$) is a random
distribution satisfying the formal property $\mathbb{E} [\xi (\cdot, x) \xi
(\cdot, y)] = \delta (x - y)$, see {\cite{GUZ}} for a rigorous definition. In
particular, the spatial white noise can be written as the following random
Fourier series (or Karhunen-Lo{\`e}ve expansion)
\begin{equation}
  \xi (\omega, x) = \underset{k \in \mathbb{Z}^2}{\sum} e_k (x) \xi_k
  (\omega), \label{xisum}
\end{equation}
where $e_k (x) = e^{2 \pi i k \cdummy x}$ and the $\xi_k =
\overline{\hat{\xi}_{- k}} : \Omega \rightarrow \mathbb{R}$ are i.i.d.
standard complex Gaussians. Note that the sum in \eqref{xisum} converges at
best a.s. in $\mathcal{C}^{- 1 - \varepsilon} (\mathbb{T}^2)$ for any
$\varepsilon > 0$ which means that $\xi (\omega, \cdot) \in \mathcal{C}^{- 1 -
\varepsilon} (\mathbb{T}^2)$ for almost all $\omega \in \Omega$ (sometimes
this is suggestively written as $\xi (\omega) \in \mathcal{C}^{- 1 -}$), see
\eqref{def:Besov} in the appendix for the definition of the H{\"o}lder-Besov
spaces $\mathcal{C}^{\alpha} (\mathbb{T}^d)$. Hereafter we sometimes use the
notation $\xi (\omega) = \xi (\omega, \cdot) \in \mathcal{S}' (\mathbb{T}^2)$
or $\xi (x) = \xi (\cdot, x)$ when we need to stress the dependence of $\xi$
on $\omega \in \Omega$ or $x \in \mathbb{T}^2$.

Due to the low regularity of $\xi$, one can not classically make sense of
$H^{\omega}$ as an (unbounded) operator on $L^2$. However, for almost all
$\omega \in \Omega$, it is possible to rigorously construct a
\tmtextit{renormalized} version $\mathbb{H}^{\omega}$ of the formal operator
$H^{\omega}$ as a self-adjoint, unbounded operator on $L^2 (\mathbb{T}^2)$
which is bounded from below by a constant $- K (\omega)$ (which can be chosen
$\mathcal{F}$-measurably as a random variable $K : \Omega \rightarrow
\mathbb{R}_+$). In addition, one can give a domain and a form domain for this
renormalized operator. Due to these properties, one can define a functional
calculus for the positive self-adjoint operator $\mathbb{H}^{\omega, K}
=\mathbb{H}^{\omega} + (K (\omega) + 1) \mathbb{I}_{L^2}$, which allows us to
define operators like
\begin{equation}
  e^{i t\mathbb{H}^{\omega, K}}, \sin \left( t \sqrt{\mathbb{H}^{\omega, K}}
  \right), \cos \left( t \sqrt{\mathbb{H}^{\omega, K}} \right)  \text{for } t
  \in \mathbb{R}, \label{func:cal}
\end{equation}
as bounded operators on $L^2 (\mathbb{T}^2)$ which are strongly continuous in
time, see Section 3 of {\cite{GUZ}}.

The functional calculus \eqref{func:cal} allows us to solve linear wave-/ and
Schr{\"o}dinger equations whose linear part is given by $\mathbb{H}^{\omega}$
as was done in {\cite{GUZ}} which corresponds to solving the SPDE with a white
noise potential (sometimes called \tmtextit{multiplicative stochastic
wave/Schr{\"o}dinger equations}).

\

In order to make rigorous sense of \eqref{Aformal} in $L^2 (\mathbb{T}^2)$,
we introduce the final definition of the operator and the noise space and then
motivate this by a formal derivation. We begin by recalling the correct notion
of \tmtextit{noise space} which contains all the needed ``higher-order''
information on the noise term $\xi (\omega)$.

\begin{definition}[and Lemma]
  \label{def:noisespace}For $\alpha = 1 + \varepsilon$ for very small
  $\varepsilon > 0$ we define the noise space
  \[ \mathcal{X}^{\alpha} \assign \overline{\{ (\psi, (1 - \Delta)^{- 1} \psi
     \circ \psi - a), \psi \in \mathcal{S} (\mathbb{T}^2), a \in \mathbb{R}
     \}} |_{\mathcal{C}^{- \alpha} \times \mathcal{C}^{2 - 2 \alpha}} \]
  i.e. the closure of tuples of the form $(\psi, (1 - \Delta)^{- 1} \psi \circ
  \psi - a)$ w.r.t. the $\mathcal{C}^{- \alpha} \times \mathcal{C}^{2 - 2
  \alpha}$-norm for smooth functions $\psi \in \mathcal{S} (\mathbb{T}^2)$ and
  constants $a \in \mathbb{R}$. See also equation \eqref{eq:resonantproduct}
  for the definition of the resonant product $\circ$.
  
  For $\xi (\omega)$ the spatial noise as introduced in \eqref{xisum} and
  $\xi_{\varepsilon} (x, \omega) = \eta_{\varepsilon} \ast \xi (x, \omega)$ a
  smooth regularization one has that
  \begin{equation}
    \Xi_{\varepsilon} = (\xi_{\varepsilon} (\omega), (1 - \Delta)^{- 1}
    \xi_{\varepsilon} (\omega) \circ \xi_{\varepsilon} (\omega) -
    c_{\varepsilon}) \rightarrow \Xi = (\xi (\omega), \Xi^2 (\omega)) 
    \text{a.s. in } \mathcal{X}^{\alpha}  \text{as } \varepsilon \rightarrow 0
    \label{def:Xi}
  \end{equation}
  where
  \[ c_{\varepsilon} \assign \mathbb{E} (\xi_{\varepsilon} (\omega) \circ (1 -
     \Delta)^{- 1} \xi_{\varepsilon} (\omega)) \sim \log \left(
     \frac{1}{\varepsilon} \right) \]
  is a diverging sequence.
  
  Often we drop the $\omega$ dependence for brevity when there is no confusion
  i.e. $\Xi^2 = \Xi^2 (\omega)$ etc. It will be important in later sections to
  track this dependence but it is not pertinent to the discussion in this
  section.
\end{definition}

\begin{proof}
  See Theorem 5.1 in {\cite{allez_continuous_2015}}.
\end{proof}

Next we define the space of functions paracontrolled by the enhanced noise
$\Xi \in \mathcal{X}^{\alpha}$ which is the limit from \eqref{def:Xi}.

\begin{definition}
  \label{def:Ham}Let $\alpha$ as above, then we define the paracontrolled (by
  $\Xi$) space $\mathcal{D}_{\Xi}^{\alpha} \subset L^2 (\mathbb{T}^2)$ as the
  space of functions $u$ of the form
  \begin{equation}
    u - (1 - \Delta)^{- 1} ((\xi + \Xi^2) \succ u + \xi \prec u) \backassign
    u^{\sharp} \in H^2 . \label{eqn:Ansatzusharp}
  \end{equation}
  On such functions we define the operator, called (renormalized) Anderson
  Hamiltonian \
  \begin{equation}
    \mathbb{H}^{\omega} (u) \assign (1 - \Delta) u^{\sharp} + u^{\sharp} \circ
    \xi + B (u, \Xi),
  \end{equation}
  where
  \[ B (u, \Xi) \assign \Xi^2 \prec u + \Xi^2 \circ u + C (\xi, (1 -
     \Delta)^{- 1} \xi, u) + ((1 - \Delta)^{- 1} (\Xi^2 \succ u + \xi \prec
     u)) \circ \xi, \]
  see Proposition \ref{prop:commu} for the trilinear commutator $C$. Similarly
  we set
  \begin{equation}
    \mathbb{H}_{\varepsilon}^{\omega} (u_{\varepsilon}) \assign (1 - \Delta)
    u_{\varepsilon}^{\sharp} + u_{\varepsilon}^{\sharp} \circ
    \xi_{\varepsilon} + B (u_{\varepsilon}, \Xi_{\varepsilon}),
  \end{equation}
  where the noise $\Xi_{\varepsilon}$ is as in \eqref{def:Xi} and
  $u_{\varepsilon}$ and $u_{\varepsilon}^{\sharp}$ are as in
  \eqref{eqn:Ansatzusharp} with $\Xi_{\varepsilon}$ instead of $\Xi .$
\end{definition}

Now we give a formal derivation of the form of the paracontrolled space
$\mathcal{D}_{\Xi}^{\alpha}$ and the operator $\mathbb{H}^{\omega} .$

\

We formally start by decomposing the product $u \cdummy \xi$ into para- and
resonant products, see equation \eqref{eq:resonantproduct} and Lemma
\ref{lem:para} from the appendix. For brevity we will write things like ``$f
\in H^{s -}$'' meaning $f \in H^{s - \varepsilon}$ for any $\varepsilon > 0$
etc.

The aim is to construct a space of functions $u \in L^2 (\mathbb{T}^2)$ s.t.
\[ (1 - \Delta) u + \xi u \in L^2 (\mathbb{T}^2), \]
where $\xi \in \mathcal{C}^{- 1 -}$ a.s. This ansatz tells us that for this to
be possible, one would need $u \in H^{1 -}$ but not better. In order to
proceed, we decompose the product via paraproducts, see equation
\eqref{eq:resonantproduct} and Lemma \ref{lem:para},
\begin{eqnarray*}
  (1 - \Delta) u + \xi \cdummy u & = & (1 - \Delta) u + \xi \prec u + \xi
  \circ u + \xi \succ u,\\
  &  & \text{where one expects}\\
  \xi \succ u & \in & H^{- 1 -} (\mathbb{T}^2)\\
  \xi \circ u & \in & H^{0 -} (\mathbb{T}^2)  \text{ but it is not defined!}\\
  \xi \prec u & \in & H^{0 -} (\mathbb{T}^2)\\
  (1 - \Delta) u & \in & H^{- 1 -} (\mathbb{T}^2) .
\end{eqnarray*}
Now the point is to consider the functions $u$ for which $(1 - \Delta) u$
cancels the parts of the product $\xi \cdummy u$ which are worse than $L^2 .$
In addition, we have to ensure that the resonant product $\xi \circ u$ can be
defined somehow; this will actually lead to the necessity to renormalize.

This is where the theory of Paracontrolled Distributions (originally
introduced in {\cite{GIP2015}}) enters, which in this context was first used
by Allez and Chouk in {\cite{allez_continuous_2015}}. The idea is (we will
have to refine this slightly) to consider functions $u \in L^2 (\mathbb{T}^2)$
for which
\[ u + (1 - \Delta)^{- 1} (\xi \succ u + \xi \prec u) \backassign u^{\sharp}
   \in H^2 . \]
For such functions we have
\begin{eqnarray*}
  (1 - \Delta) u + \xi \cdummy u & = & (1 - \Delta) u + \xi \prec u + \xi
  \circ u + \xi \succ u\\
  & = & (1 - \Delta) u^{\sharp} + \xi \circ u\\
  & = & (1 - \Delta) u^{\sharp} + \xi \circ u^{\sharp} - \xi \circ (1 -
  \Delta)^{- 1} (\xi \succ u + \xi \prec u),
\end{eqnarray*}
and we see that the situation is much improved since one has the regularities
\begin{eqnarray*}
  (1 - \Delta) u^{\sharp} & \in & L^2\\
  \xi \circ u^{\sharp} & \in & H^{1 -}\\
  \xi \circ (1 - \Delta)^{- 1} (\xi \prec u) & \in & H^{1 -},
\end{eqnarray*}
with only the term $\xi \circ (1 - \Delta)^{- 1} (u \prec \xi)$ being
problematic. Thanks to the commutator lemma from {\cite{GIP2015}}, see
Proposition \ref{prop:commu}, one can however transform this term as follows
\begin{equation}
  \xi \circ (1 - \Delta)^{- 1} (\xi \prec u) = C (\xi, (1 - \Delta)^{- 1} \xi,
  u) + u (\xi \circ (1 - \Delta)^{- 1} \xi), \label{Acommu}
\end{equation}
where $C (\xi, (1 - \Delta)^{- 1} \xi, u) \in H^{1 -}$.

\

The second term in \eqref{Acommu} is now a classically defined product
provided we can make sense of the purely stochastic term
\[ \xi \circ (1 - \Delta)^{- 1} \xi \overset{!}{\in} \mathcal{C}^{0 -} . \]
The issue here is that this object does not exist in any reasonable sense,
unless we \tmtextit{renormalize} it, meaning --naively-- that we replace it by
an almost surely well-defined object
\begin{equation}
  \xi \diamondsuit (1 - \Delta)^{- 1} \xi \assign \underset{\varepsilon
  \rightarrow 0}{\lim} \xi_{\varepsilon} \circ (1 - \Delta)^{- 1}
  \xi_{\varepsilon} - c_{\varepsilon} = \Xi^2 \in \mathcal{C}^{0 -}
  \label{def:diam}
\end{equation}

which is precisely the second component of the noise term $\Xi$ from
Definition \ref{def:noisespace} where the $\xi_{\varepsilon}$ is the noise
mollified by a standard test function and the constants $c_{\varepsilon}$
satisfy
\begin{equation}
  c_{\varepsilon} \assign \mathbb{E} (\xi_{\varepsilon} \circ (1 - \Delta)^{-
  1} \xi_{\varepsilon}) \sim \log \left( \frac{1}{\varepsilon} \right)
  \label{def:ceps}
\end{equation}
i.e. they diverge logarithmically. This is intimately related to Wick ordering
which will also appear in a different context later on in Section
\ref{sec:Wick}.

\

Thanks to \eqref{def:diam}, we can rigorously repeat the above computation
with the regularized noise for smooth functions $u$ setting
\begin{equation}
  u - (1 - \Delta)^{- 1} ((\xi_{\varepsilon} + \Xi^2_{\varepsilon}) \succ u +
  \xi_{\varepsilon} \prec u) \backassign u^{\sharp} \in H^2 . \label{def:ueps}
\end{equation}
This yields
\begin{eqnarray*}
  (1 - \Delta) u + \xi_{\varepsilon} \cdummy u & = & (1 - \Delta) u^{\sharp} +
  \Xi^2_{\varepsilon} \succ u + u \circ \xi_{\varepsilon}\\
  & = & (1 - \Delta) u^{\sharp} + \Xi^2_{\varepsilon} \succ u + u^{\sharp}
  \circ \xi_{\varepsilon} + ((1 - \Delta)^{- 1} \xi_{\varepsilon} \succ u)
  \circ \xi_{\varepsilon} + \tilde{B} (u, \Xi_{\varepsilon}^2)\\
  & = & (1 - \Delta) u^{\sharp} + u^{\sharp} \circ \xi_{\varepsilon} +
  \Xi^2_{\varepsilon} u - u ((1 - \Delta)^{- 1} \xi_{\varepsilon} \circ
  \xi_{\varepsilon}) + B (u, \Xi_{\varepsilon}^2)\\
  & = & (1 - \Delta) u^{\sharp} + u^{\sharp} \circ \xi_{\varepsilon} -
  c_{\varepsilon} u + B (u, \Xi_{\varepsilon}^2)
\end{eqnarray*}
which is rearranged to
\begin{equation} \mathbb{H}^{\omega}_{\varepsilon} u \assign (1 - \Delta) u +
   \xi_{\varepsilon} u + c_{\varepsilon} u = \Delta u^{\sharp} + u^{\sharp}
   \circ \xi_{\varepsilon} + B (u, \Xi_{\varepsilon}^2), 
   \label{def:Heps} \end{equation}
where $B, \tilde{B}$, given explicitly below, are bounded bilinear maps from
$H^{1 -} \times \mathcal{X}^{\alpha} \rightarrow H^{1 -}$.
\begin{eqnarray*}
  \tilde{B} (u, \Xi_{\varepsilon}) & \assign & ((1 - \Delta)^{- 1}
  (\Xi^2_{\varepsilon} \succ u + \xi_{\varepsilon} \prec u)) \circ
  \xi_{\varepsilon}\\
  B (u, \Xi_{\varepsilon}) & \assign & \Xi_{\varepsilon}^2 \prec u +
  \Xi_{\varepsilon}^2 \circ u + C (\xi_{\varepsilon}, (1 - \Delta)^{- 1}
  \xi_{\varepsilon}, u) + \tilde{B} (u, \Xi_{\varepsilon}) .
\end{eqnarray*}
These maps satisfy the following continuity property.

\begin{lemma}
  The bilinear maps
  \begin{eqnarray*}
    \tilde{B} (u, \Xi) & : & H^{\sigma} \times \mathcal{X}^{\alpha}
    \rightarrow H^{2 - 2 \alpha + \sigma}\\
    B (u, \Xi) & : & H^{\sigma} \times \mathcal{X}^{\alpha} \rightarrow H^{2 -
    2 \alpha + \sigma},
  \end{eqnarray*}
  are bounded for $2 \alpha - 2 < \sigma < 1,$ in particular this implies
  \[ B (u_{\varepsilon}, \Xi_{\varepsilon}) \rightarrow B (u, \Xi)  \text{in
     } H^{2 - 2 \alpha + \sigma}  \text{as } \varepsilon \rightarrow 0 \text{
     for } u_{\varepsilon} \rightarrow u \text{ in } H^{\sigma} . \]
\end{lemma}

\begin{proof}
  This follows from the bounds on the paraproducts and the commutator from
  Lemmas \ref{lem:para} and Proposition \ref{prop:commu}.
\end{proof}

The point of this computation is that the right-hand side of \eqref{def:Heps}
is now continuous w.r.t. $(\xi_{\varepsilon}, \Xi^2_{\varepsilon})$ in the
noise space $\mathcal{X}^{\alpha}$ so it allows us to pass to the limit in
some sense. For now, we can rigorously define the operator
\begin{eqnarray}
  \mathbb{H}^{\omega} u &\assign & (1 - \Delta) u^{\sharp} + u^{\sharp} \circ
  \xi + B (u, \Xi)  \label{def:Homega}\\
  && \text{for} \\
  u \in L^2 (\mathbb{T}^2) && \text{s.t.} \nonumber\\
  u - (1 - \Delta)^{- 1} ((\xi + \Xi^2) \succ u + \xi \prec u) & \backassign&
  u^{\sharp} \in H^2 (\mathbb{T}^2)  \label{eqn:ansatzlimit}
\end{eqnarray}
The ansatz \eqref{eqn:ansatzlimit} is of course the limit of the ansatz
\eqref{def:ueps} and one has the following rigorous result.

\begin{theorem}[Self-adjointness and (Form-)Domain of the Anderson
Hamiltonian]
  \label{thm:anderson-K}The operator $(\mathbb{H}^{\omega},
  \mathcal{D}_{\Xi}^{\alpha})$ is an unbounded self-adjoint semi-bounded
  operator on $L^2 (\mathbb{T}^2)$. One has the norm equivalence
  \[ \| \mathbb{H}^{\omega} u \|_{L^2} + \| u \|_{L^2} \approx \| u^{\sharp}
     \|_{H^2} . \]
  Moreover, one has that if the remainder $u^{\sharp}$ in the paracontrolled
  ansatz is only in $H^1$, i.e. it satisfies the paracontrolled ansatz
  \[ u - (1 - \Delta)^{- 1} ((\xi + \Xi^2) \succ u + \xi \prec u) \backassign
     u^{\sharp} \in H^1 (\mathbb{T}^2), \]
  such a paracontrolled function $u$ is in the form domain of
  $\mathbb{H}^{\omega}$ meaning $| (u, \mathbb{H}^{\omega} u) | < \infty$. In
  fact one has the norm equivalence
  \[ | (u, \mathbb{H}^{\omega} u) | + \| u \|^2_{L^2} \approx \| u^{\sharp}
     \|^2_{H^1} . \]
  The operator $\mathbb{H}^{\omega}$ is bounded from below, meaning there
  exists a constant $K (\omega) > 0$ depending polynomially on the
  $\mathcal{X}$-norm of the enhanced noise $\Xi$ s.t.
  \[ (\mathbb{H}^{\omega} u, u) \geqslant - K (\omega) \| u \|^2_{L^2} 
     \text{for all } u \in \mathcal{D}_{\Xi}^{\alpha} \]
  and we define the shifted operators
  \begin{eqnarray}
    \mathbb{H}^{\omega, K} \assign & \mathbb{H}^{\omega} + K (\omega) + 1 
    \label{def:HK}\\
    \mathbb{H}_{\varepsilon}^{\omega, K} \assign &
    \mathbb{H}_{\varepsilon}^{\omega} + K (\omega) + 1  \label{def:HKeps}
  \end{eqnarray}
  which is now uniformly positive and self-adjoint so one can define its
  square root and other fractional powers without issues.
\end{theorem}

\begin{proof}
  See Section 2.1 in {\cite{GUZ}}.
\end{proof}

Moreover we can quantify exactly in which way the regularized operators
$\mathbb{H}_{\varepsilon}^{\omega}$ converge to $\mathbb{H}^{\omega} .$

\begin{proposition}[Norm resolvent convergence of approximate operators]
  \label{prop:normconv}For the operators \ $\mathbb{H}_{\varepsilon}^{\omega}$
  and $\mathbb{H}^{\omega}$ as above we have that there exists a constant $K
  (\omega)$, which is a polynomial in the $\mathcal{X}^{\alpha}$ norm of
  $\Xi$, for which
  \[ (K (\omega) +\mathbb{H}_{\varepsilon}^{\omega})^{- 1} \rightarrow (K
     (\omega) +\mathbb{H}^{\omega})^{- 1}  \text{as $\varepsilon \rightarrow
     0$ in } \mathcal{L} (L^2 (\mathbb{T}^2) ; L^2 (\mathbb{T}^2)) . \]
  We may choose the constants $K (\omega) $as in the previous theorem.
\end{proposition}

\begin{proof}
  See Proposition 2.23 in {\cite{GUZ}}.
\end{proof}

For reasons which will become apparent later, we introduce a frequency
cut-off,
\[ \Pi_{> N} \assign \mathcal{F}^{- 1} \mathbb{I}_{| \cdummy | > 2^N}
   \mathcal{F} \quad \text{for } N \in \mathbb{N}, \]
as in {\cite{GUZ}} namely we define
\begin{equation}
  \Phi_N (u) \assign u - \Pi_{> N} ((1 - \Delta)^{- 1} ((\xi + \Xi^2) \succ u
  + \xi \prec u)) \label{def:Phi}
\end{equation}
as a bounded operator on $L^2 (\mathbb{T}^2)$ which admits an inverse for $N$
large enough depending on the $\mathcal{X}-$norm of $(\xi, \Xi^2)$ which we
denote by
\begin{equation}
  \Gamma v = v + \Pi_{> N} ((1 - \Delta)^{- 1} ((\xi + \Xi^2) \succ \Gamma v +
  \xi \prec \Gamma v)) \label{def:Gamma}
\end{equation}
having omitted the $N$ in the notation as was done in {\cite{GUZ}}. In
precisely the same way, we can define $\Phi^{\varepsilon}$ and
$\Gamma^{\varepsilon}$ analogously to \eqref{def:Phi} and \eqref{def:Gamma}
respectively by replacing $(\xi, \Xi^2)$ by $(\xi_{\varepsilon},
\Xi^2_{\varepsilon})$. As was remarked in {\cite{GUZ}}, one may choose the
same $N$ independently of $\varepsilon$ (but of course depending measurably on
$\omega$).

\

We think of the $\Gamma$ map in the following way:
\[ \Gamma : \text{``Remainder''} \mapsto \text{``Paracontrolled function with
   that remainder''} \]
and it exactly parameterizes a paracontrolled space like the one in Definition
\ref{def:Ham}, concretely one has
\[ \mathcal{D}_{\Xi}^{\alpha} = \Gamma H^2 . \]
With this notation in place, we collect some results on the maps $\Gamma$ and
$\Gamma^{\varepsilon}$ as well as their convergence properties.

\begin{proposition}
  There is a choice of $N \in 2^{\mathbb{N}}$for which the maps $\Gamma,
  \Gamma^{\varepsilon} : L^2 (\mathbb{T}^2) \rightarrow L^2 (\mathbb{T}^2)$ in
  \eqref{def:Gamma} exists, i.e. as the inverse of the $\Phi_N$ defined in
  \eqref{def:Phi}. One has the properties:
  
  Let $s \in [0, 1)$, then $\Gamma$ is a homeomorphism on the following spaces
  \begin{eqnarray*}
    \Gamma : H^s & \rightarrow & H^s\\
    \Gamma : H^1 & \rightarrow & \mathcal{D} \left( \sqrt{\mathbb{H}^{\omega}}
    \right)\\
    \Gamma : H^2 & \rightarrow & \mathcal{D} (\mathbb{H}^{\omega})\\
    \Gamma  : \mathcal{C}^s & \rightarrow & \mathcal{C}^s .
  \end{eqnarray*}
  $\Gamma^{\varepsilon} $is also a homeomorphism on $H^s $and $\mathcal{C}^s
  .$
  
  Furthermore we have $\Gamma^{\varepsilon} \rightarrow \Gamma$ for
  $\varepsilon \rightarrow 0$ as bounded operators on $H^s$or $\mathcal{C}^s$.
\end{proposition}

\begin{proof}
  See {\cite{GUZ}} but this follows from the paraproduct estimates and the
  fact that $\Gamma$ was already defined as an inverse.
\end{proof}

Using these maps, we actually have the stronger convergence
\[ \mathbb{H}^{\omega}_{\varepsilon} \Gamma_{\varepsilon} \rightarrow
   \mathbb{H}^{\omega} \Gamma \text{in } \mathcal{L} (H^2 ; L^2) . \]
which implies Proposition \ref{prop:normconv}, see Proposition 2.19 in
{\cite{GUZ}}.

We finish off the section by collecting some other properties of the Anderson
Hamiltonian which we will need in the remainder of the paper.

\begin{theorem}[Properties of the Anderson Hamiltonian]
  \label{thm:AHprop}With the notations as above, we have
  \begin{enumerate}
    \item {\tmstrong{Embeddings:}} For $\mathcal{D} (\mathbb{H}^{\omega}) =
    \Gamma H^2$, the domain of $\mathbb{H}^{\omega}$, we have
    \[ \mathcal{D} (\mathbb{H}^{\omega}) \cap H^2 = \{ 0 \}, \text{but }
       \mathcal{D} (\mathbb{H}^{\omega}) \hookrightarrow H^{1 - \varepsilon} 
       \text{and $\mathcal{D} (\mathbb{H}^{\omega}) \hookrightarrow
       \mathcal{C}^{1 - \varepsilon}$ for any } \varepsilon > 0. \]
    For $\mathcal{D} \left( \sqrt{\mathbb{H}^{\omega}} \right) = \Gamma H^1$,
    the form domain of $\mathbb{H}^{\omega},$we have
    \[ \mathcal{D} \left( \sqrt{\mathbb{H}^{\omega}} \right) \cap H^1 = \{ 0
       \}, \text{but } \mathcal{D} \left( \sqrt{\mathbb{H}^{\omega}} \right)
       \hookrightarrow H^{1 - \varepsilon}  \text{for any } \varepsilon > 0.
    \]
    \item \tmtextbf{Functional calculus:} For any bounded continuous function
    \[ g : \mathbb{R}_+ \rightarrow \mathbb{R}, \]
    one has (using the shifted operators from \eqref{def:HK},
    \eqref{def:HKeps})
    \[ g (\mathbb{H}^{\omega, K}_{\varepsilon}) \rightarrow g
       (\mathbb{H}^{\omega, K})  \text{in } \mathcal{L} (L^2 ; L^2) . \]
    In particular one has for all times $t \in \mathbb{R}$
    \begin{eqnarray}
      \sin \left( t \sqrt{\mathbb{H}^{\omega, K}} \right), \cos \left( t
      \sqrt{\mathbb{H}^{\omega, K}} \right) & \in & \mathcal{L} (L^2 ; L^2) 
      \label{eqn:cosL2}\\
      \frac{\sin \left( t \sqrt{\mathbb{H}^{\omega, K}}
      \right)}{\sqrt{\mathbb{H}^{\omega, K}}} & \in & \mathcal{L} \left( L^2 ;
      \mathcal{D} \left( \sqrt{\mathbb{H}^{\omega, K}} \right) \right), \\
      & \text{and} &  \nonumber\\
      \cos \left( t \sqrt{\mathbb{H}^{\omega, K}_{\varepsilon}} \right)
      \rightarrow \cos \left( t \sqrt{\mathbb{H}^{\omega, K}} \right) &
      \text{in} & \mathcal{L} (L^2 ; L^2) \\
      \frac{\sin \left( t \sqrt{\mathbb{H}^{\omega, K}_{\varepsilon}}
      \right)}{\sqrt{\mathbb{H}^{\omega, K}_{\varepsilon}}} \rightarrow
      \frac{\sin \left( t \sqrt{\mathbb{H}^{\omega, K}}
      \right)}{\sqrt{\mathbb{H}^{\omega, K}}} & \text{in} & \mathcal{L} (L^2 ;
      H^{1 - \varepsilon})  \text{for any } \varepsilon > 0. 
    \end{eqnarray}
    Moreover these operators are strongly continuous in $t.$
    
    \item \tmtextbf{Eigenvalues and Weyl Law:} $\mathbb{H}^{\omega, K}$ has
    discrete spectrum and it has eigenvalues
    \[ 0 < \lambda_1 (\omega) \leqslant \lambda_2 (\omega) \leqslant \ldots
       \lambda_n (\omega) \leqslant \cdots \rightarrow + \infty \text{as } n
       \rightarrow \infty \]
    and $L^2 $normalized eigenfunctions $f_n \in \mathcal{D}
    (\mathbb{H}^{\omega})$
    \begin{equation}
      \mathbb{H}^{\omega, K} f_n = \lambda_n (\omega) f_n . \label{eqn:Heigen}
    \end{equation}
    Moreover, $\mathbb{H}^{w, K}$(and thus $\mathbb{H}^{\omega}$) satisfies a
    \tmtextit{Weyl law}, meaning almost surely
    \[ \underset{n \rightarrow \infty}{\lim} \frac{\lambda_n (\omega)}{n} = C
       (\omega), \]
    i.e. the eigenvalues grow like the eigenvalues of the Laplacian.
    
    \item \tmtextbf{Equivalence of fractional norms:} For $s \in (- 1, 1)$ we
    have the norm equivalence
    \begin{equation}
      \| u \|_{H^s} \approx \left\| \left( {\mathbb{H}^{w, K}} 
      \right)^{\frac{s}{2}} u \right\|_{L^2} . \label{eqn:normequ}
    \end{equation}
  \end{enumerate}
\end{theorem}

\begin{proof}
  The first two points are found in {\cite{GUZ}} the third point is from
  {\cite{mouzard2020weyl}} and the last point was proved in Proposition 1.14
  in {\cite{mouzardstrichartz}}.
\end{proof}

\begin{remark}
  \label{remark:N-delta}A consequence of Theorem \ref{thm:AHprop} Point 3 and
  Point 4 is that for almost every $\omega \in \Omega$, writing $P_{\leqslant
  N} : H^s (\mathbb{T}^2) \rightarrow H^s (\mathbb{T}^2)$, with $N \in
  \mathbb{N}$ and $s, s' \in (- 1, 1)$, $s \leqslant s'$, the orthogonal
  projection on the first $N$ eigenvectors of $\mathbb{H}^{\omega}$, we have
  \begin{equation}
    \| (I - P_{\leqslant N}) f \|_{H^s} \lesssim
     N^{s - s'} \| f
    \|_{H^{s'}},
  \end{equation}
  and similarly also
  \begin{equation}
    \| P_{\leqslant N} f \|_{H^{s'}} \lesssim N^{s' - s} \| f \|_{H^s}
  \end{equation}
  see also Lemma 1.3 of {\cite{mouzardstrichartz}}.
\end{remark}

\subsection{Gaussian measures and Wick powers}\label{sec:Wick}

A Gaussian measure $\mu$ on the space of tempered distribution $\mathcal{S}'
(\mathbb{T}^d)$ on the $d \in \mathbb{N}$ dimensional torus, is a Radon
measure $\mu$ on $\mathcal{S}' (\mathbb{T}^d)$ (with respect its strong
topology) such that for any smooth function $f \in C^{\infty} (\mathbb{T}^d)
\backassign \mathcal{S} (\mathbb{T}^d)$ the (real valued) random variable $x
\mapsto \langle x, f \rangle_{\mathcal{S}', \mathcal{S}}$ (where $x \in
\mathcal{S}' (\mathbb{T}^d)$) is a Gaussian random variable.

A Gaussian measure is completely characterized by its mean $m \in
\mathcal{S}' (\mathbb{T}^d)$ and its covariance operator $\Sigma : \mathcal{S}
(\mathbb{T}^d) \rightarrow \mathcal{S}' (\mathbb{T}^d)$ (which is a linear
positive operator), which are the two unique objects appearing in the
characteristic function $\hat{\mu} : \mathcal{S} (\mathbb{T}^d) \rightarrow
\mathbb{C}$ of $\mu$, which takes the form
\[ \hat{\mu} (f) \assign \int_{\mathcal{S}' (\mathbb{T}^d)} e^{i \langle x, f
   \rangle_{\mathcal{S}', \mathcal{S}}} \mu (\mathd x) = \exp \left( i \langle
   m, f \rangle_{\mathcal{S}', \mathcal{S}} - \frac{1}{2} \langle \Sigma (f),
   f \rangle \right) . \]
A consequence of Minlos-Sazonov theorem (see, e.g., Theorem 20.1 in
{\cite{Yam1985}}) is that for any $m \in \mathcal{S}' (\mathbb{T}^d)$ and any
\tmtextit{continuous, linear, positive and symmetric operator} $\Sigma$ there
is a unique Gaussian measure $\mu$ on $\mathcal{S}' (\mathbb{T}^d)$ with mean
$m$ and covariance $\Sigma$. Hereafter we mainly focus on the case where $m =
0$ and we write $\mu^{\Sigma^{- 1}}$ for the Gaussian measure on $\mathcal{S}'
(\mathbb{T}^d)$ with variance $\Sigma$ (the reason of the presence of $- 1$
will become apparent later). We also often identify $\mathcal{S}
(\mathbb{T}^d)$ with a subset of $\mathcal{S}' (\mathbb{T}^d)$ thanks to the
$L^2 (\mathbb{T}^d)$ scalar product. In this paper we are mainly interested in
three cases of Gaussian measures:
\begin{enumerate}
  \item when $\langle \Sigma (f), g \rangle \assign \int_{\mathbb{T}^d} f (x)
  g (x) \mathd x$, namely $\Sigma = I_{L^2}$, which is the
  \tmtextit{(Gaussian) white noise measure} on $\mathbb{T}^d$;
  
  \item when $\langle \Sigma (f), g \rangle \assign \int_{\mathbb{T}^d} (-
  \Delta + K)^{- 1} (f) (x) g (x) \mathd x$, where $K > 0$, namely $\Sigma =
  (- \Delta + K)^{- 1}$, which is usually called \tmtextit{Gaussian free
  field} on $\mathbb{T}^d$ with mass $K$;
  
  \item when $\langle \Sigma (f), g \rangle \assign \int_{\mathbb{T}^d}
  (\mathbb{H}^{\omega} + K)^{- 1} (f) (x) g (x) \mathd x$, for $K > K
  (\omega)$ (see Section \ref{sec:AHamilton}), namely $\Sigma =
  (\mathbb{H}^{\omega} + K)^{- 1}$, which henceforth we will call
  \tmtextit{Anderson Gaussian free field} with mass $K$.
\end{enumerate}
We also recall two convenient facts about Gaussian
measures. Firstly the fact that convergence of the covariance operators
implies weak convergence of the Gaussian measures and secondly that one has
precise knowledge of whether two Gaussian measures are mutually singular or
absolutely continuous.

\begin{lemma} \label{lem:abscont}
  \begin{enumeratealpha}
    \item If the operators
    \[ \Sigma^{\varepsilon} \rightarrow \Sigma \text{ in } \mathcal{L} (L^2
       (\mathbb{T}^d) ; L^2 (\mathbb{T}^d)) \]
    then we have
    \[ \mu^{(\Sigma^{\varepsilon})^{- 1}} \rightharpoonup \mu^{\Sigma^{- 1}}
    \]
    weakly in the sense of measures on $\mathcal{S}' (\mathbb{T}^d)$.
    
    Consequently we have
    \begin{equation}
      \mu^{\mathbb{H}_{\varepsilon}^{\omega, K}} \rightharpoonup
      \mu^{\mathbb{H}^{\omega, K}}  \text{weakly in the sense of measures on
      $\mathcal{S}' (\mathbb{T}^d)$ as } \varepsilon \rightarrow 0,
      \label{lem:Hepsweak}
    \end{equation}
    see Proposition \ref{prop:normconv}.
    
    \
    
    \item If $\sqrt{\Sigma}^{- 1} (L^2) \neq \sqrt{\Lambda}^{- 1} (L^2)$ then
    the Gaussian measures $\mu^{\Sigma^{- 1}}$ and $\mu^{\Lambda^{- 1}}$ are
    mutually singular. In particular this implies that the Gaussian measure
    $\mu^{\mathbb{H}^{\omega, K}}$ is mutually singular with respect to
    $\mu^{- \Delta + K}$ and also $\mu^{\mathbb{H}_{\varepsilon}^{\omega, K}}
    .$
  \end{enumeratealpha}
\end{lemma}

\begin{proof}
  \begin{enumeratealpha}
    \item The general statement is in Section 5 of {\cite{BDW18}}. The
    statement \eqref{lem:Hepsweak} then follows from Proposition
    \ref{prop:normconv}.
    
    \item  The main statement is a consequence of the Feldman-Hajek theorem,
    see e.g. Theorem 2.23 in {\cite{DaZa}}, the other statement follows from the
    fact that
    \[ \mathcal{D} \left( \sqrt{\mathbb{H}^{\omega, K}} \right) \neq H^1
       =\mathcal{D} \left( \sqrt{\mathbb{H}_{\varepsilon}^{\omega, K}} \right)
       =\mathcal{D} \left( \sqrt{- \Delta + K} \right), \]
    which is contained in Theorem \ref{thm:AHprop}.
  \end{enumeratealpha}
\end{proof}

Usually the support (not understood in a topological sense but more
generically as a subset of full measure) of a Gaussian measure $\mu^{\Sigma^{-
1}}$ on $\mathcal{S}' (\mathbb{T}^d)$ is not the whole space $\mathcal{S}'
(\mathbb{T}^d)$ but there is a proper (Banach) subspace $W \subset
\mathcal{S}' (\mathbb{T}^d)$ supporting $\mu^{\Sigma^{- 1}}$. In the case of
the white noise and Gaussian free field the support is well known.

\begin{proposition}
  For any $\delta > 0$ and $1 \leqslant p \leqslant + \infty$ we have
  \[ \mu^{- \Delta + K} \left( B^{1 - \frac{d}{2} - \delta}_{p, p}
     (\mathbb{T}^d) \right) = \mu^{I_{L^2}} \left( B^{- \frac{d}{2} -
     \delta}_{p, p} (\mathbb{T}^d) \right) = 1. \]
\end{proposition}

\begin{proof}
  See, e.g., Lemma 3.2 in {\cite{DaP2003}}.
\end{proof}

From now on if $(\Omega, \mathcal{F}, \mathbb{P})$ is a probability space and
$\psi : \Omega \rightarrow W \subset \mathcal{S}' (\mathbb{T}^d)$ (where $W$
is a Banach space) is a measurable function, we say that $\psi$ is a
\tmtextit{Gaussian random field with covariance $\Sigma$} if the law of $\psi$
is $\mu^{\Sigma^{- 1}}$, written in symbols $\tmop{Law} (\psi) =
\mu^{\Sigma^{- 1}}$. \

In the following we will use the notion of Wick products of Gaussian random
variables: let $X_1, \ldots, X_n$, $n \in \mathbb{N}$, be a set of (jointly)
Gaussian random variables with zero mean. We call the \tmtextit{Wick products}
of $X_1, \ldots, X_n$ the random variable
\begin{equation}
  : X_1 \cdots X_n : \assign \left( \frac{\partial^n}{\partial t_1 \cdots
  \partial t_n} \left( \frac{\exp \left( \sum_{i = 1}^n t_i X_i
  \right)}{\mathbb{E} \left[ \exp \left( \sum_{i = 1}^n t_i X_i \right)
  \right]} \right) \right)_{t_1 = \cdots = t_n = 0} \label{eq:Wickproduct}
\end{equation}
In the case where $X_1 = \cdots = X_n = Z$ then
\[ : Z^n \assign (\tmop{var} (Z))^{\frac{n}{2}} H_n \left(
   \frac{Z}{\sqrt{\tmop{var} (Z)}} \right), \]
where $H_n$ is the $n$-th Hermite polynomial.

\begin{proposition}
  \label{prop:Wickexp}Let $X_1, \ldots, X_n$, $Y_1, \ldots, Y_m$ be some
  jointly Gaussian random variables with zero mean then
  \[ \mathbb{E} [: X_1 \cdots X_n : : Y_1 \cdots Y_m :] = \delta_{n, m}
     \sum_{\sigma \in \mathcal{S}_n} \prod_{j = 1}^m \mathbb{E} [X_j Y_{\sigma
     (j)}] \]
  where $\mathcal{S}_n$ is the set of permutations of $\{ 1, \ldots, n \}$.
\end{proposition}

\begin{proof}
  See Theorem 3.9 in {\cite{Jan1997}}.
\end{proof}

Let $L^2 (W, \mu^{\Sigma^{- 1}})$ (where $W \subset \mathcal{S}'
(\mathbb{T}^d)$ is a support of $\mu^{\Sigma^{- 1}}$) be the set of $L^2$
random variables with defined on $W$ with respect the Gaussian measure
$\mu^{\Sigma^{- 1}}$. For every $n \in \mathbb{N}$, We define
\[ \Gamma_n = \overline{\tmop{span} \{ : \langle w, f_1 \rangle \cdots
   \langle w, f_n \rangle : |f_1, \ldots, f_n \in \mathcal{S} (\mathbb{T}^d)
   \}} \subset L^2 (W, \mu^{\Sigma^{- 1}}) \]
where the closure is taken with respect the natural topology of $L^2 (W,
\mu^{\Sigma^{- 1}})$. We write $\Gamma_0 =\mathbb{R}$. With these definition
we have:

\begin{proposition}[Chaos decomposition and hypercontractivity]
  We have that
  \[ L^2 (W, \mu^{\Sigma^{- 1}}) = \bigoplus_{n = 0}^{\infty} \Gamma_n . \]
  Furthermore for every $G \in \Gamma_n$ and $p \geqslant 2$ we have
  \[ \mathbb{E} [| G |^p] \leqslant (p - 1)^{\frac{n p}{2}} (\mathbb{E} [| G
     |^2])^{\frac{p}{2}} . \]
\end{proposition}

\begin{proof}
  For the first statement see, e.g., Theorem 4.1 in {\cite{Jan1997}}. For the
  second statement see, e.g., Theorem 5.1 and Remark 5.11 in {\cite{Jan1997}}.
\end{proof}

Let $\varphi : \Omega \rightarrow W \subset \mathcal{S}' (\mathbb{T}^d)$ be a
Gaussian free field of mass $m$ (namely $\tmop{Law} (\varphi) = \mu^{- \Delta
+ m^2}$) and consider $a_{\varepsilon} (x)$ a smooth mollifier, where $a :
\mathbb{T}^d \rightarrow \mathbb{R}_+$ is a smooth function with compact
support such that $\int_{\mathbb{T}^d } a (x) \mathd x = 1$. We defined $\varphi_{\varepsilon}
= a_{\varepsilon} \ast \varphi$, which is Gaussian random field taking values
in $C^{\infty} (\mathbb{T}^d)$. Since $\varphi_{\varepsilon}$ takes values in
a space of functions, we can define $: \varphi_{\varepsilon}^n :$ as a smooth
random function given by
\[ : \varphi^n_{\varepsilon} : (x) = : (\varphi_{\varepsilon} (x))^n :
   =\mathbb{E} [| \varphi_{\varepsilon} (x) |^2]^{\frac{n}{2}} H_n \left(
   \frac{\varphi_{\varepsilon} (x)}{\sqrt{\mathbb{E} [| \varphi_{\varepsilon}
   (x) |^2]}} \right) . \]
\begin{theorem}
  \label{theorem:Wickproduct}In the above setting, taking $d = 2$, for any $1
  \leqslant p < + \infty$ and $\delta > 0$ the sequence of random functions $:
  \varphi^n_{\varepsilon} :$ converges in
  \[ L^p ((\Omega, \mathcal{F}, \mathbb{P}), B^{- \delta}_{p, p}
     (\mathbb{T}^2)) \]
  to some random distribution $: \varphi^n :$ defined on $(\Omega,
  \mathcal{F}, \mathbb{P})$ and taking values in $B^{- \delta}_{p, p}
  (\mathbb{T}^2)$. Furthermore the random distribution $: \varphi^n :$ does
  not depend on the mollifier $a_{\varepsilon}$.
\end{theorem}

\begin{proof}
  See, e.g., Lemma 3.2 in {\cite{DaP2003}}.
\end{proof}

\begin{remark}
  By Besov embedding, see Lemma \ref{lem:besovem}, this implies that the Wick
  powers lie even in $L^p ((\Omega, \mathcal{F}, \mathbb{P}), C^{- \delta}
  (\mathbb{T}^2))$ for any $p \geqslant 2$ and $\delta > 0.$
\end{remark}

\section{A (regular) coupling between GFF and AGFF}\label{section:coupling}

From Lemma \ref{lem:abscont} we know that the Gaussian measures $\mu^{- \Delta
+ K}$ and $\mu^{\mathbb{H}^{\omega} + K}$ are mutually singular for almost
every $\omega$ and for any $K \geqslant K (\omega)$ (where $K (\omega)$ is the
almost surely positive random variable in Section \ref{sec:AHamilton} which
ensures that $\mathbb{H}^{\omega} + K$ is invertible and thus
$\mu^{\mathbb{H}^{\omega} + K}$ is well-defined). In this section we prove a
weaker but interesting result, namely the existence of a regular coupling
between the measure $\mu^{- \Delta + K}$ and $\mu^{\mathbb{H}^{\omega} + K}$
for almost every $\omega \in \Omega$ and for any $K \geqslant K (\omega)$.
This coupling permits us to extend the definition of the Wick product of the
Gaussian free field to the support of the Anderson free field, see Section
\ref{sec:AndersonWick}.

\subsection{Variational formulation of the coupling problem}

First we want to give a variational representation of the Gaussian measure
$\mu^{\varepsilon} = \mu^{\mathbb{H}_{\varepsilon}^{\omega, K}}$ (or more
generally $\mu^{\varepsilon}$'s Laplace transform) in a way similar to what
was done in {\cite{BG2022,BG2020,Barashkovwholespace,BG2021}}. The first step
is to write the Radon-Nikodym derivatives of $\mu^{\varepsilon}$ with respect
to $\mu^{- \Delta + K}$ and to prove that it has some good properties
permitting the variational representation.

\

So, consider a Gaussian white noise $\xi : \Omega \rightarrow \mathcal{S}'
(\mathbb{T}^2)$ defined on the probability space $(\Omega, \mathcal{F},
\mathbb{P})$. We define the measure
\begin{equation}
  \frac{\mathd \mu^{\varepsilon} (\omega)}{\mathd \mu^{- \Delta + K}}
  (\varphi) = \frac{1}{\mathcal{Z}^{\varepsilon} (\omega)} \exp \left( -
  \int_{\mathbb{T}^2} (\xi_{\varepsilon} (x, \omega) - (\gamma^{(1)} (\omega)
  - \gamma_{\varepsilon}^{(2)})) : \varphi^2 : (x) \mathd x \right)
  \label{eq:nuepsilon}
\end{equation}
where
\[ \mathcal{Z}^{\varepsilon} (\omega) = \int \exp \left( -
   \int_{\mathbb{T}^2} (\xi_{\varepsilon} (x, \omega) - (\gamma^{(1)} (\omega)
   - \gamma_{\varepsilon}^{(2)})) : \varphi^2 : (x) \mathd x \right) \mathd
   \mu^{- \Delta + K} (\varphi) \]
and where $\gamma^{(1)} : \Omega \rightarrow \mathbb{R}_+$ is a positive
random variable (not depending on $\varepsilon$ which
plays a role analogous to the $K (\omega)$ from Proposition
\ref{prop:normconv}) and $\gamma_{\varepsilon}^{(2)}$ is a real number
(depending on $\varepsilon > 0$ but not depending on $\omega$), and $:
\varphi^2 :$ denotes the Wick product of the random field $\varphi$ with
respect to the Gaussian measure $\mu^{- \Delta + K}$. As in prior sections, \
$\xi : \Omega \rightarrow \mathcal{S}' (\mathbb{T}^2)$ is a Gaussian white
noise and $\xi_{\varepsilon}$ is a smooth approximation of $\xi$ (for example
$\xi_{\varepsilon} = \rho_{\varepsilon} \ast \xi$ where $\rho_{\varepsilon} =
\varepsilon^{- 2} \rho (\cdot / \varepsilon)$ and $\rho \in C^{\infty}
(\mathbb{T}^2)$ $\int_{\mathbb{T}^2 } \rho (x) \mathd x = 1$).

\

The constant $\gamma_{\varepsilon}^{(2)} \rightarrow + \infty$ as
$\varepsilon \rightarrow 0$ such that $- \Delta + \xi_{\varepsilon} -
\gamma^{(2)}_{\varepsilon} \rightarrow \mathbb{H}^{\omega}$ as $\varepsilon
\rightarrow 0$ in the norm resolvent sense. The constant $\gamma^{(1)}
(\omega)$ is chosen in such a way, that there is $\gamma^{(1)} (\omega)
\geqslant K (\omega)$ (where $K (\omega) > 0$ is the random variable defined
in Theorem \ref{thm:anderson-K}), for which, for every $\omega \in \Omega$, we
have
\[ - \Delta + \xi_{\varepsilon} (\omega) + \gamma^{(1)} (\omega) \geqslant
   \mathbb{I}_{L^2} \]
as self-adjoint operators on the domain of $- \Delta$ (i.e. we suppose that
the lowest eigenvalue of $- \Delta + \xi_{\varepsilon} (\omega) + \gamma^{(1)}
(\omega)$ is greater than $1$ almost surely). We suppose also that for any $p
\in \mathbb{R}_+$
\[ \mathbb{E} [| \gamma^{(1)} |^p] < + \infty . \]
Under these notations and setting we prove the following lemma.

\begin{lemma}
  \label{lemma:absolutecontinuity}Let $g : \mathbb{T}^2 \rightarrow
  \mathbb{R}$ be a smooth function such that the operator $H_g = - \Delta + K
  + g (x)$ (on the domain of $- \Delta + K$) is positive, then if $\mu^{H_g}$
  is the Gaussian measure with covariance $(- \Delta + K + g (x))^{- 1}$ and
  $\mu^{- \Delta + K}$ is the measure of the Gaussian free field we have
  \begin{equation}
    \log \left( \frac{\mathd \mu^{H_g}}{\mathd \mu^{- \Delta + K}} (\varphi)
    \right) = - \int_{\mathbb{T}^2} g (x) : \varphi^2 : (x) \mathd x + \log
    (\mathcal{Z}_{H_g}) \label{eq:density2}
  \end{equation}
  where $: \varphi^2 :$ denotes the Wick product of the field $\varphi$ with
  respect to the Gaussian free field measure $\mu^{- \Delta + K}$, and
  \[ \mathcal{Z}_{H_g} = \int \exp \left( - \int_{\mathbb{T}^2} g (x) :
     \varphi^2 : (x) \mathd x \right) \mu^{- \Delta + K} (\mathd \varphi) . \]
\end{lemma}

\begin{proof}
  First we prove that $\mu^{H_g}$ is absolutely continuous with respect to
  $\mu^{- \Delta + K}$. The measure $\mu^{H_g}$ can be obtained as the
  push-forward of the measure $\mu^{- \Delta + K}$ through the transformation
  \[ T (\varphi) = (- \Delta + K + g (x))^{- 1 / 2} (- \Delta + K)^{1 / 2}
     (\varphi), \]
  with inverse $S (w) = (- \Delta + K)^{- 1 / 2} (- \Delta + K + g (x))^{1 /
  2} (\varphi)$. We have that
  \begin{eqnarray}
    v (\varphi) & = & (- \Delta + K)^{- 1 / 2} (- \Delta + K + g (x))^{1 / 2}
    (\varphi) - \varphi \nonumber\\
    & = & ((I + (- \Delta + K)^{- 1} g (x))^{1 / 2} - I) (\varphi) \nonumber
  \end{eqnarray}
  where $I$ is the identity operator on $\mathcal{C}^{- \delta}
  (\mathbb{T}^2)$. The (linear) operator $v$ restricted on $H^1$ (i.e. the
  Cameron-Martin space of the free field having law $\mu^{- \Delta + K}$) is a
  Hilbert-Schmidt operator. Indeed we have
  \begin{eqnarray}
    v (\varphi) & = & \left( (I + (- \Delta + K)^{- 1} g (x))^{1 / 2} - I -
    \frac{1}{2} (- \Delta + K)^{- 1} g (x) \right) (\varphi) + \frac{1}{2} (-
    \Delta + K)^{- 1} (g (x) \varphi) \nonumber\\
    & = & \frac{1}{4} \left( (I + (- \Delta + K)^{- 1} g (x))^{1 / 2} + I +
    \frac{1}{2} (- \Delta + K)^{- 1} g (x) \right)^{- 1} ((- \Delta + K)^{- 1}
    g (x))^2 (\varphi) + \nonumber\\
    &  & + \frac{1}{2} (- \Delta + K)^{- 1} (g (x) \varphi) \nonumber
  \end{eqnarray}
  which shows that $v (\varphi)$ is the sum of the Hilbert-Schmidt operator
  $\frac{1}{2} (- \Delta + m^2)^{- 1} g (x)$ and of a trace class operator
  (being the remainder the product of a bounded operator and the square of a
  Hilbert-Schmidt operator). Since $v$ is linear (and thus differentiable in
  $\mathcal{C}^{- \delta} (\mathbb{T}^2)$ with derivative equal to $v$), by
  Theorem 3.5.3 of {\cite{Ustunelbook}}, we get
  \begin{equation}
    \frac{\mathd \mu^{H_g}}{\mathd \mu^{- \Delta + K}} = \frac{\mathd T_{\ast}
    \mu^{- \Delta + K}}{\mathd \mu^{- \Delta + K}} = {\det}_{2} (I + v) \exp
    \left( - \delta (v (\varphi)) - \frac{1}{2} \left\| v (\varphi)
    \right\|_{H^1}^2 \right) \label{eq:density1}
  \end{equation}
  where $\delta$ is the Skorokhod integral with respect to the Gaussian
  measure $\mu^{- \Delta + K}$ on $\mathcal{C}^{- \delta} (\mathbb{T}^2)$, and
  $\det_2$ is the regularized determinat (see, e.g., Chapter 9 of
  {\cite{SimonIdeals}}). What remains to be shown is that the term in the
  exponential in \eqref{eq:density1} is (up to some finite additional
  constant) equal to $\int_{\mathbb{T}^2 } g (x) \ :\varphi^2 (x): d x$. In order to show this
  equality, we consider a (finite dimensional) approximation of $v$ given by
  \begin{eqnarray}
    v_N (\varphi) & = & ((I + (- \Delta + K)^{- 1} \Pi_N g (x) \Pi_N)^{1 / 2}
    - I) (\varphi) \nonumber
  \end{eqnarray}
  where $\Pi_N$ is the $L^2 (\mathbb{T}^2)$ projection onto the subspace of
  trigonometric polynomial of degree less or equal than $N$. We have that, by
  definition of the Skorokhod integral for functions having trace class
  derivatives in the Cameron-Martin space (see, e.g., Section B.4 of
  {\cite{Ustunelbook}}), i.e. the formula
  \[ \delta (u (\varphi)) = \langle i^{\ast} (u (\varphi)), \varphi
     \rangle_{H^1 (\mathbb{T}^2)} - \tmop{Tr}_{H^1 (\mathbb{T}^2)} (\nabla u
     (\varphi)), \]
  the following holds
  \[ \delta (v_N (\varphi)) = \int (- \Delta + K) (v_N (\varphi)) (x) \varphi
     (x) - \tmop{Tr}_{L^2} ((- \Delta + K)^{1 / 2} v_N (- \Delta + K)^{- 1 /
     2}) . \]
  We also obtain
  \begin{eqnarray}
    &  & \tmop{Tr}_{L^2} ((- \Delta + K)^{1 / 2} v_N (- \Delta + K)^{- 1 /
    2}) \nonumber\\
    & = & \tmop{Tr}_{L^2} (v_N) \nonumber\\
    & = & \frac{1}{4} \tmop{Tr}_{L^2} \left. \left( \left( (I + (- \Delta +
    K)^{- 1} \Pi_N g (x) \Pi_N)^{1 / 2} + I + \frac{1}{2} (- \Delta + K)^{- 1}
    \Pi_N g (x) \Pi_N \right)^{- 1} \times \right. \right. \nonumber\\
    &  & \left. \phantom{\frac{1}{2}} ((- \Delta + K)^{- 1} \Pi_N g (x)
    \Pi_N)^2 (\cdot) \right) + \frac{1}{2} \tmop{Tr}_{L^2} ((- \Delta + K)^{-
    1} (\Pi_N g (x) \Pi_N)) . \nonumber
  \end{eqnarray}
  The first term in the previous sum is uniformly bounded in $N$, for the
  second term we get
  \begin{eqnarray}
    \tmop{Tr}_{L^2} ((- \Delta + K)^{- 1} (\Pi_N f (x))) & = & \tmop{Tr}_{L^2}
    (\Pi_N ((- \Delta + K)^{- 1} (\Pi_N f (x)))) \nonumber\\
    & = & \int_{\mathbb{T}^2} \mathcal{G}_N (x - x) f (x) \mathd x =
    \mathcal{G}_N (0) \int_{\mathbb{T}^2} f (x) \mathd x \nonumber\\
    & = & \mathbb{E} [(\Pi_N \varphi (0))^2] \int_{\mathbb{T}^2} f (x) \mathd
    x \nonumber
  \end{eqnarray}
  where $\mathcal{G}_N$ is the integral kernel of the operator $\Pi_N (-
  \Delta + m^2)^{- 1} \Pi_N$ and $\varphi$ is a random distribution with law
  $\mu^{- \Delta + K}$. On the other hand, by an explicit computation, we get
  \[ - \int (- \Delta + K) (v_N (\varphi)) (x) \varphi (x) \mu^{-\Delta+K}(\mathd \varphi) - \frac{1}{2} \|
     v_N (\varphi) \|_{H^1}^2 = \int_{\mathbb{T}^2} g (x) (\Pi_N \varphi)^2
     (x) \mathd x. \]
  Putting it all together, we get
  \begin{eqnarray}
    - \delta (v_N (\varphi)) - \frac{1}{2} \| v_N (\varphi) \|_{H^1}^2 & = &
    \int_{\mathbb{T}^2} g (x) (\Pi_N \varphi)^2 (x) \mathd x -\mathbb{E}
    [(\Pi_N \varphi (0))^2] \int_{\mathbb{T}^2} g (x) \mathd x + C_N
    \nonumber\\
    & = & \int_{\mathbb{T}^2} g (x) : (\Pi_N \varphi)^2 : (x) \mathd x + C_N
    \nonumber
  \end{eqnarray}
  where $C_N$ is a suitable constant converging to some $C \in \mathbb{R}$ as
  $N \rightarrow + \infty .$ Taking the limit on both sides of the previous
  expression we get the thesis.
\end{proof}

Thus the previous lemma proves expression \eqref{eq:nuepsilon}. We now
introduce the following useful definition describing the key property for the
variational representation of an exponential functional.

\begin{definition}
  \label{definition:tame}Let $W \subset \mathcal{S}' (\mathbb{T}^2)$ be a
  Banach space supporting the law of Gaussian free field $\mu^{- \Delta + K}$.
  We say that a measurable function $G : W \rightarrow \mathbb{R}$ is tame
  (with respect to the law of the Gaussian free field $\mu^{- \Delta + K}$) if
  there are $p, q \geqslant 1$, $\frac{1}{p} + \frac{1}{q} = 1$, such that
  \[ \int \exp (p G (\varphi)) \mu^{- \Delta + K} (\mathd \varphi) + \int
     | G (\varphi) |^q \mu^{- \Delta + K} (\mathd \varphi) < + \infty . \]
\end{definition}

Under the previous hypotheses the Radon-Nikodym derivative $\frac{\mathd
\mu^{\varepsilon}}{\mathd \mu^{- \Delta + K}}$ is a tame function. Indeed, we
have the following result.

\begin{lemma}
  \label{lemma:tameg}Suppose that $g$ is smooth and that $(- \Delta + m^2 + g
  (x)) \geqslant R\mathbb{I}_{L^2}$ for some $R > 0,$then the functional
  $\varphi \rightarrow \int_{\mathbb{T}^2} g (x) : \varphi^2 (x) : \mathd x$
  is tame.
\end{lemma}

\begin{proof}
  If $(- \Delta + m^2 + g (x)) \geqslant R\mathbb{I}_{L^2}$ then there is $p >
  1$ such that $(- \Delta + m^2 + p g (x)) > 0$. Indeed we have that
  \[ (- \Delta + m^2 + p g (x)) \geqslant (- \Delta + m^2 + g (x)) - (p - 1)
     \| g \|_{L^{\infty}} \mathbb{I}_{L^2} \geqslant (R - (p - 1) \| g
     \|_{L^{\infty}}) \mathbb{I}_{L^2} \]
  which is strictly positive whenever $p - 1 < \frac{R}{\| g
  \|_{L^{\infty}}}$. This means that we can apply Lemma
  \ref{lemma:absolutecontinuity} to the operator $(- \Delta + m^2 + p g (x))$,
  obtaining that
  \[ \exp \left( p \int_{\mathbb{T}^2} g (x) : \varphi^2 (x) : \mathd x
     \right) \in L^1 (\mu) . \]
  Since, by hypercontractivity, $\int_{\mathbb{T}^2} g (x) : \varphi^2 (x) :
  \mathd x \in L^q (\mu)$ for any $1 \leqslant q < + \infty$ the thesis is
  proved. 
\end{proof}

\

If we consider $W = \mathcal{C}^{- \delta} (\mathbb{T}^2) = B^{-
\delta}_{\infty, \infty} (\mathbb{T}^2)$ (for some $\delta > 0$ small enough),
under the previous conditions on $\gamma^{(1)}$ we will show that the
functional
\[ G^{\varepsilon, f} (\varphi, \omega) = f (\varphi) + \int_{\mathbb{T}^2}
   (\xi_{\varepsilon} (x, \omega) - (\gamma^{(1)} (\omega) -
   \gamma_{\varepsilon}^{(2)})) : \varphi^2 (x) : \mathd x \]
is tame, whenever $f (\cdummy)$ is tame in the sense of Definition
\ref{definition:tame}.

\

\begin{notation}
  \label{notation:omegaprime}In order to provide the variational representation
  of the formula \eqref{eq:nuepsilon}, we need to introduce a Gaussian white
  noise $X_t : \Omega' \times \mathbb{R}_+ \rightarrow \mathcal{S}'
  (\mathbb{T}^2)$ defined on the probability space $(\Omega', \mathcal{F}_t',
  \mathbb{P}')$. We define also the product space $\bar{\Omega} = \Omega
  \times \Omega'$ with the product $\sigma$-algebra and equipped with the
  filtration $\overline{\mathcal{F}}_t = \mathcal{F} \vee \mathcal{F}_t$ and
  the product probability measure $\bar{\mathbb{P}} =\mathbb{P} \otimes
  \mathbb{P}'$ (under this measure the white noise $\xi$ and $X_t$ are
  independent). We consider the operator $J_s : \mathcal{S}' (\mathbb{T}^2)
  \rightarrow C^{\infty} (\mathbb{T}^2)$ given by the expression
  \[ J_s (f) = \mathcal{F}^{- 1} \left( \frac{\sigma_s (| k |^2)}{\sqrt{m^2 +
     | k |^2}} \mathcal{F} (f) (k) \right) \]
  where $\sigma_{\cdot} : \mathbb{R}_+ \times \mathbb{R}^2 \rightarrow
  \mathbb{R}_+$ is a smooth function with compact support such that $\int_0^t
  \sigma_s^2 (k) \mathd s = \rho_t (k)$ where $\rho_t$ is a smooth cut-off
  function of the ball of radius $t > 0$. In order to distinguish between the
  expectation with respect the $\omega \in \Omega$ variable and the
  probability law $\mathbb{P}$, with respect the $\omega' \in \Omega'$
  variable and the probability law $\mathbb{P}'$, we will write
  $\mathbb{E}^{\omega} [\cdot]$ and $\mathbb{E}^{\omega'} [\cdot]$
  respectively. We write also $\mathbb{E}$ for the expectation with respect to
  both the variables (namely on the probability space $(\bar{\Omega},
  \overline{\mathcal{F}}_t, \bar{\mathbb{P}})$ described above).
\end{notation}

\begin{notation}
  \label{notation:W}We use the following the notations
  \[ W_t = \int^t_0 J_s \mathd X_s \qquad Z_t (u) = \int^t_0 J_s u_s \mathd s
  \]
  where $u : \Omega' \times \mathbb{R}_+ \rightarrow L^2 (\mathbb{T}^2)$ is a
  measurable function adapted with respect to the filtration $\mathcal{F}_t'$
  (when it is clear from the context we drop the dependence on $u$ in the
  random process $Z_t$).
\end{notation}

The next theorem gives a variational representation of the Laplace transform
of $\nu^{\varepsilon}$.

\begin{theorem}
  \label{theorem:functionalF}Let $\delta > 0.$ For every tame function $f :
  \mathcal{C}^{- \delta} \rightarrow \mathbb{R}$ and for every $\omega \in
  \Omega$ we have
  \begin{equation}
    \begin{array}{ll}
      & - \log \int \exp (- f (\varphi)) \nu^{\varepsilon} (\omega, \mathd
      \varphi) + \log (\mathcal{Z}^{\varepsilon} (\omega))\\
      = & - \log \int \exp (- G^{\varepsilon, f} (\varphi, \omega)) \mu^{-
      \Delta + 1} (\mathd \varphi)\\
      = & \inf_{u \in \mathbb{H}_a} \mathbb{E}^{\omega'} \left[ f (W_{\infty}
      + Z_{\infty} (u)) + \int \xi_{\varepsilon} (\omega) W_{\infty}
      Z_{\infty} (u) \mathd x + \int \xi_{\varepsilon} (\omega) Z^2_{\infty}
      (u) \mathd x \right.\\
      & \left. \nobracket \nobracket + \int ((\gamma^{(1)} (\omega) +
      \gamma_{\varepsilon}^{(2)})) W_{\infty} Z_{\infty} (u) + \int \left(
      \left( \gamma^{(1)} (\omega) + \gamma_{\varepsilon}^{(2)} \right)
      \right) Z^2_{\infty} (u) \mathd x + \frac{1}{2} \int^{\infty}_0 \| u_s
      \|^2_{L^2} \mathd s \right]\\
      = & \inf_{u \in \mathbb{H}_a} F^{\varepsilon} (f, \omega) .
    \end{array}
  \end{equation}
  \label{eq:optimalFepsilon}
\end{theorem}

\begin{proof*}{Proof}
  Fix an $\varepsilon > 0$, then $G^{\varepsilon, f}$ is a tame functional
  since is the sum of $f (\varphi)$ (which is tame by hypothesis) and
  \[ V (\varphi) = \int (\xi_{\varepsilon} (\omega, x) + (\gamma^{(1)}
     (\omega) + \gamma_{\varepsilon}^{(2)})) : \varphi^2 (x) : \mathd x \]
  which is tame by Lemma \ref{lemma:tameg}. Thus we can apply the main result
  of {\cite{U2014}} and hence we have
  \begin{eqnarray}
    &  & - \log \int \exp (- G^{\varepsilon, f} (\varphi, \omega)) \mu^{-
    \Delta + 1} (\mathd \varphi) \nonumber\\
    & = & \inf_{u \in \mathbb{H}_a} \mathbb{E}^{\omega'} \left[ f (W_{\infty}
    + Z_{\infty} (u)) + \frac{1}{2} \int^{\infty}_0 \| u_s \|^2_{L^2} \mathd s
    + V (W_{\infty} + Z_{\infty} (u)) \right] \nonumber\\
    & = & \inf_{u \in \mathbb{H}_a} \mathbb{E}^{\omega'} \left[ f (W_{\infty}
    + Z_{\infty} (u)) + \int \xi_{\varepsilon} (\omega) W_{\infty} Z_{\infty}
    (u) \mathd x + \int \xi_{\varepsilon} (\omega) Z^2_{\infty} (u) \mathd x
    \right. \nonumber\\
    &  & \left. \nobracket \nobracket + \int ((\gamma^{(1)} (\omega) +
    \gamma_{\varepsilon}^{(2)})) W_{\infty} Z_{\infty} (u) + \int \left(
    \left( \gamma^{(1)} (\omega) + \gamma_{\varepsilon}^{(2)} \right) \right)
    Z^2_{\infty} (u) \mathd x + \frac{1}{2} \int^{\infty}_0 \| u_s \|^2_{L^2}
    \mathd s \right] \nonumber
  \end{eqnarray}
  where we use the fact that the expectation of $: W_{\infty}^2 : (x)$ is
  zero, giving us the desired result.
\end{proof*}

It is convenient to rewrite the functional $F^{\varepsilon} (\omega, u)$ in a
more useful form for what follows.

\begin{proposition}
  \label{proposition:reformulationF} Let $u : \bar{\Omega} \times \mathbb{R}_+
  \rightarrow L^2 (\mathbb{T}^2)$ be an adapted process.
  
  We have the following identity
  \begin{eqnarray}
    F^{\varepsilon} (u, \omega) & = & \mathbb{E}^{\omega'} \left[ \sum_{i =
    1}^7 \Gamma_i + \mathfrak{G} + \frac{1}{2} \int^{\infty}_0 \| l_t (u)
    \|^2_{L^2} \mathd s \right] + \frac{1}{2} \mathbb{E}^{\omega'} \left[
    \int^{\infty}_0 \left\| {J_s}  \xi_{\varepsilon} W_s \right\|^2_{L^2}
    \mathd s \right]  \label{eq:functional-renom}
  \end{eqnarray}
  where
  \[ l_s (u) {= J_s}  \xi_{\varepsilon} W_s + J_s (\xi_{\varepsilon} \succ
     Z_s) - u_s \]
  and
  \begin{eqnarray*}
    \Gamma_1 & = & 2 \int^{\infty}_0 \int J_t (\xi_{\varepsilon} \preccurlyeq
    Z_t) u_t \mathd t \mathd x\\
    \Gamma_2 & = & 2 \int^{\infty}_0 \int \left( {J_s}  \xi_{\varepsilon} W_s
    \circ J_s \xi_{\varepsilon} - \dot{\gamma}_{\varepsilon, s}^{(2)} W_s
    \right) Z_s \mathd x \mathd s\\
    \Gamma_3 & = & \int^{\infty}_0 \int ((J_s \xi_{\varepsilon} \circ J_s
    \xi_{\varepsilon}) - \dot{\gamma}_{\varepsilon, s}^{(2)}) Z^2_s \mathd x
    \mathd s\\
    \Gamma_4 & = & \int^{\infty}_0 \int (J_s (\xi_{\varepsilon} \succ Z_s))^2
    \mathd x \mathd s - \int^{\infty}_0 \int (J_s \xi_{\varepsilon} \circ J_s
    \xi_{\varepsilon}) Z^2_s \mathd x \mathd s\\
    \Gamma_5 & = & 2 \int^{\infty}_0 \int {J_s}  (\xi_{\varepsilon} W_s) J_s
    (\xi_{\varepsilon} \succ Z_s) \mathd x \mathd s - 2 \int^{\infty}_0 \int
    \left( {J_s}  \xi_{\varepsilon} W_s \circ J_s \xi_{\varepsilon} \right)
    Z_s \mathd x \mathd s\\
    \Gamma_6 & = & - 2 \int^{\infty}_0 \int \gamma_{\varepsilon, t}^{(2)} W_t
    J_t u_t \mathd x \mathd t - 2 \int^{\infty}_0 \int \gamma_{\varepsilon,
    t}^{(2)} Z_t J_t u_t \mathd x \mathd t\\
    \Gamma_7 & = & 2 \int_0^{\infty} \dot{\gamma}_t^{(1)} W_t Z_t \mathd t + 2
    \int^{\infty}_0 \int \gamma_t^{(1)} W_t J_t u_t \mathd x d t + 2
    \int^{\infty}_0 \int \gamma_t^{(1)} Z_t J_t u_t \mathd x \mathd t\\
    \mathfrak{G} & = & \int_0^{\infty} \int \dot{\gamma}^{(1)}_t (\omega)
    Z_t^2 \mathd x \mathd t
  \end{eqnarray*}
  where $\gamma_t^{(1)} : \mathbb{R}_+ \times \Omega \rightarrow \mathbb{R}$
  and $\gamma^{(2)}_{\varepsilon, t} : \mathbb{R}_+ \rightarrow \mathbb{R}$
  are $C^1$ functions such that $\lim_{t \rightarrow + \infty} \gamma_t^{(1)}
  (\omega) = \gamma^{(1)} (\omega)$ and $\lim_{t \rightarrow + \infty}
  \gamma^{(2)}_{\varepsilon, t} = \gamma^{(2)}_{\varepsilon}$.
\end{proposition}

\begin{proof}
  Observe that by Ito's formula
  \[ \mathbb{E} \left[ \int \xi_{\varepsilon} W_{\infty} Z_{\infty} \mathd x
     \right] =\mathbb{E} \left[ \int^{\infty}_0 \int \xi_{\varepsilon} W_t J_t
     u_t \mathd t \mathd x \right] =\mathbb{E} \left[ \int^{\infty}_0 J_t
     (\xi_{\varepsilon} W_t) u_t \mathd t \right] \]
  and
  \[ \int \xi_{\varepsilon} Z^2_{\infty} \mathd x = 2 \int^{\infty}_0 \int
     J_t (\xi_{\varepsilon} Z_t) u_t \mathd t \mathd x = 2 \int^{\infty}_0
     \int J_t (\xi_{\varepsilon} \succ Z_t) u_t \mathd t \mathd x + 2
     \int^{\infty}_0 \int J_t (\xi_{\varepsilon} \preccurlyeq Z_t) u_t \mathd
     t \mathd x \]
  now consider the ansatz
  \[ u_s {= - J_s}  \xi_{\varepsilon} W_s - J_s (\xi_{\varepsilon} \succ Z_s
     (u)) + l_s \]
  Then we compute
  \begin{eqnarray*}
    \frac{1}{2} \int^{\infty}_0 \| u_s \|^2_{L^2} \mathd s & = & \frac{1}{2}
    \int^{\infty}_0 \left\| {J_s}  \xi_{\varepsilon} W_s + J_s
    (\xi_{\varepsilon} \succ Z_s) \right\|^2_{L^2} \mathd s\\
    &  & - \int^{\infty}_0 \left( {J_s}  \xi_{\varepsilon} W_s + J_s
    (\xi_{\varepsilon} \succ Z_s) \right) l_s \mathd s\\
    &  & + \frac{1}{2} \int^{\infty}_0 \| l_s \|^2_{L^2} \mathd s\\
    & = & - \frac{1}{2} \int^{\infty}_0 \left\| {J_s}  \xi_{\varepsilon} W_s
    + J_s (\xi_{\varepsilon} \succ Z_s) \right\|^2_{L^2} \mathd s\\
    &  & - \int^{\infty}_0 \left( {J_s}  \xi_{\varepsilon} W_s + J_s
    (\xi_{\varepsilon} \succ Z_s) \right) u_s \mathd s\\
    &  & + \frac{1}{2} \int^{\infty}_0 \| l_s \|^2_{L^2} \mathd s
  \end{eqnarray*}
  where we have used that
  \begin{eqnarray*}
    &  & - \int^{\infty}_0 \left\| \left( {J_s}  \xi_{\varepsilon} W_s + J_s
    (\xi_{\varepsilon} \succ Z_s) \right) \right\|^2_{L^2} \mathd x -
    \int^{\infty}_0 \left( {J_s}  \xi_{\varepsilon} W_s + J_s
    (\xi_{\varepsilon} \succ Z_s) \right) u_s \mathd s\\
    & = & - \int^{\infty}_0 \left( {J_s}  \xi_{\varepsilon} W_s + J_s
    (\xi_{\varepsilon} \succ Z_s) \right) l_s \mathd x.
  \end{eqnarray*}
  We compute that
  \begin{eqnarray*}
    &  & \int^{\infty}_0 \left\| {J_s}  \xi_{\varepsilon} W_s + J_s
    (\xi_{\varepsilon} \succ Z_s) \right\|^2_{L^2} \mathd x \mathd s\\
    & = & \int^{\infty}_0 \int \left( {J_s}  \xi_{\varepsilon} W_s \right)^2
    + 2 \left( {J_s}  \xi_{\varepsilon} W_s \right) (J_s (\xi_{\varepsilon}
    \succ Z_s)) + (J_s (\xi_{\varepsilon} \succ Z_s))^2 \mathd x \mathd s
  \end{eqnarray*}
  Now the first term is the last term on the r.h.s of
  \eqref{eq:functional-renom}. We also have that
  \[ \int^{\infty}_0 \int 2 \left( {J_s}  \xi_{\varepsilon} W_s \right) (J_s
     (\xi_{\varepsilon} \succ Z_s)) \mathd x \mathd s = 2 \int^{\infty}_0 \int
     \left( {J_s}  \xi_{\varepsilon} W_s \circ J_s \xi_{\varepsilon} \right)
     Z_s \mathd x \mathd s + \Gamma_5 \label{eq:weird-term} \]
  and
  \[ \int^{\infty}_0 \int (J_s (\xi_{\varepsilon} \succ Z_s))^2 \mathd x
     \mathd s = \int^{\infty}_0 \int (J_s \xi_{\varepsilon} \circ J_s
     \xi_{\varepsilon}) Z^2_s \mathd x \mathd s + \Gamma_4 . \label{eq:square}
  \]
  Recall that we also have the counter terms
  \[ 2 (\gamma^{(1)} (\omega) - \gamma_{\varepsilon}^{(2)}) \int W_{\infty}
     Z_{\infty} \mathd x + (\gamma^{(1)} (\omega) -
     \gamma_{\varepsilon}^{(2)}) \int Z^2_{\infty} \mathd x \]
  available. Writing $\gamma_{\varepsilon, \infty} = \gamma^{(1)} (\omega) -
  \gamma_{\varepsilon}^{(2)}$ and using that $\gamma_{\varepsilon, t} :
  \mathbb{R}_+ \times \Omega \rightarrow \mathbb{R}$ is a $C^1$ function, by
  Ito's formula we get
  \[ 2 \gamma_{\varepsilon, \infty} \int W_{\infty} Z_{\infty} \mathd x = 2
     \int^{\infty}_0 \int \dot{\gamma}_{\varepsilon, t} W_t Z_t \mathd x + 2
     \int^{\infty}_0 \int \gamma_{\varepsilon, t} W_t J_t u_t \mathd x +
     \text{martingale} \]
  Now the first term on the r.h.s. is put together with the first term on the
  r.h.s of \eqref{eq:weird-term} to form $\Gamma_2$. We have also
  \[ \gamma_{\varepsilon, \infty} \int Z^2_{\infty} \mathd x =
     \int^{\infty}_0 \int \dot{\gamma}_{\varepsilon, t} Z^2_t \mathd x + 2
     \int^{\infty}_0 \int \gamma_{\varepsilon, t} Z_t J_t u_t \mathd x \mathd
     t \]
  and the first term on the r.h.s goes together with the first term on the
  r.h.s of \eqref{eq:square} to form $\Gamma_3$. The respective remainders are
  collected in $\Gamma_6$, $\Gamma_7$ and $\mathfrak{G}$. \ 
\end{proof}

\subsection{Some stochastic estimates }\label{section:stochasticestimates}

The main reason for the formulation in Proposition
\ref{proposition:reformulationF} is that, up to a diverging constant, every
term of $\Gamma_i$ converges to something finite as $\varepsilon \rightarrow
0$. In order to guarantee this convergence we need to give some estimates for
the stochastic terms involving $\xi_{\varepsilon}$ and $W_s$ in equation
\eqref{eq:functional-renom}. The proof of this kind of estimates is the chief
aim of the current section.

\begin{lemma}
  \label{lemma:stochasticestimates1}Consider
  \begin{equation}
    \gamma_{\varepsilon, t}^{(2)} \assign \int_0^t \mathbb{E} [J_s
    (\xi_{\varepsilon}) \circ J_s (\xi_{\varepsilon})] \mathd s
    \label{eq:definitiongamma2}
  \end{equation}
  then for every for any $0 < \kappa < \delta$, $\alpha > 0$ such that $\kappa
  + \alpha < \delta$ small enough, and for every $p \geqslant 1$, we have that
  \begin{equation}
    \sup_{\varepsilon \in (0, 1)} \mathbb{E} [\| J_s (\xi_{\varepsilon}) \circ
    J_s (\xi_{\varepsilon}) - \dot{\gamma}_{\varepsilon, t}^{(2)} \|_{B^{-
    \delta}_{p, p}}^p] \lesssim (1 + s)^{- (1 + \kappa) p}
    \label{eq:stochasticbound1}
  \end{equation}
  
  \begin{equation}
    \begin{array}{cc}
      & \sup_{\varepsilon \in (0, 1)} \mathbb{E} \left[ \int_0^1 \frac{\| J_s
      (\xi_{\varepsilon}) \circ J_s (\xi_{\varepsilon}) -
      \dot{\gamma}_{\varepsilon, s}^{(2)} - J_{s + \Delta s}
      (\xi_{\varepsilon}) \circ J_{s + \Delta s} (\xi_{\varepsilon}) -
      \dot{\gamma}_{\varepsilon, s + \Delta s}^{(2)} \|^p_{B^{- \delta}_{p,
      p}}}{(\Delta s)^{1 + \alpha p / 2}} \mathd \Delta s \right]\\
      & \lesssim (1 + s)^{(- 1 - \kappa + \alpha) p} .
    \end{array} \label{eq:stochasticbound2}
  \end{equation}
\end{lemma}

\begin{remark}
  \label{remark:gamma2}From equation \eqref{eq:definitiongamma2}, defining the
  function $\gamma^{(2)}_{\varepsilon, t}$, we get the bound
  \[ \gamma^{(2)}_{\varepsilon, t} \lesssim \log (\min (2 + t, \varepsilon^{-
     1})) . \]
\end{remark}

\begin{proof*}{Proof of Lemma \ref{lemma:stochasticestimates1}}
  We start the proof in the case that $p = 2$. In order to simplify the
  notation we drop the upper index $(2)$ from $\gamma^{(2)}_{\varepsilon}$.
  First we prove that
  \begin{equation}
    \sup_{\varepsilon \in (0, 1)} \mathbb{E} [(\| J_s (\xi_{\varepsilon})
    \circ J_s (\xi_{\varepsilon}) - \dot{\gamma}_{\varepsilon, s} \|_{H^{-
    \delta}})^2] \lesssim (1 + s)^{- 2 - 2 \kappa} \label{eq:stochasticbound3}
    .
  \end{equation}
  Indeed, if $K_j = \mathcal{F}^{- 1} (\varphi_j)$ (where $\{ \varphi_j \}_{j
  \geqslant - 1}$ is the dyadic partition of unity in the definition of Besov
  space $B^{\kappa}_{p, p}$, see Appendix \ref{app:Besov}) we have
  \begin{eqnarray}
    &  & \mathbb{E} [| K_r \ast (J_s (\xi_{\varepsilon}) \circ J_s
    (\xi_{\varepsilon}) - \dot{\gamma}_{\varepsilon, s}) |^2] \nonumber\\
    & \lesssim & \sum_{j \sim \log (s), r \lesssim j} \int_{\mathbb{T}^2}
    \int_{\mathbb{T}^2} K_r (x) K_r (y) (\mathbb{E} [J_s (\Delta_j
    (\xi_{\varepsilon})) (x) J_s (\Delta_j (\xi_{\varepsilon})) (y)])^2
    \nonumber
  \end{eqnarray}
  Clearly
  \[ J_s (\xi_{\varepsilon}) \circ J_s (\xi_{\varepsilon}) -
     \dot{\gamma}_{\varepsilon, s} = \sum_{j \sim \log (s)} (J_s (\Delta_j
     (\xi_{\varepsilon})))^2 - \dot{\gamma}_{\varepsilon, s} = \sum_{j \sim
     \log (s)} : J_s (\Delta_j (\xi_{\varepsilon}))^2 : \]
  and, by the properties of Wick product (see Proposition \ref{prop:Wickexp}),
  \begin{eqnarray}
    &  & \mathbb{E} [J_s (\Delta_j \xi_{\varepsilon}) (x) J_s (\Delta_j
    \xi_{\varepsilon}) (y)] \nonumber\\
    & \lesssim & \sum_{k \in \mathbb{Z}^2} \frac{\sigma_s^2 (k)  | \varphi_j
    (k) |^2}{(| k |^2 + m^2)} \lesssim - \int_{R' s}^{R s} \frac{r}{s^2}
    \left( \frac{\mathd}{\mathd x} \rho \right) \left( \frac{r}{s} \right)
    \frac{r}{r^2 + m^2} \mathd r \nonumber\\
    & \lesssim & \left( \frac{1}{s} \int_{R' s}^{R s} \rho \left( \frac{r}{s}
    \right) \frac{\mathd}{\mathd r} \left( \frac{r^2}{r^2 + m^2} \right)
    \mathd r + \frac{{R'}^2 s^2}{{R'}^2 s^2 + m^2} \right) \nonumber\\
    & \lesssim & \frac{1}{s} \int_{R' s}^{R s} \frac{\mathd}{\mathd r} \left(
    \frac{r^2}{r^2 + m^2} \right) \mathd r \lesssim \frac{1}{s} \left(
    \frac{R^2 s^2}{R^2 s^2 + m^2} \right)  \label{eq:inequalityJS}
  \end{eqnarray}
  and the bound is uniform in $0 \leqslant \varepsilon < 1$. Recalling that
  \begin{align*}
  	\| J_s (\xi_{\varepsilon}) \circ J_s (\xi_{\varepsilon}) -
  	\dot{\gamma}_{\varepsilon, s} \|_{H^{- \delta}}^2 &\sim \sum_{j \sim \log
  		(s)} 2^{- 2 \delta j} \| K_j \ast (J_s (\xi_{\varepsilon}) \circ J_s
  	(\xi_{\varepsilon}) - \dot{\gamma}_{\varepsilon, s}) \|_{L^2}^2 \\
  	&\lesssim_{\kappa, \delta} \frac{1}{s^{2 + 2 \kappa}} \sum_{j \geqslant -
  		1} {2^{- 2 (\delta - \kappa) j}},
  \end{align*}
  where $\kappa < \delta$, inequality \eqref{eq:inequalityJS} implies the
  bound \eqref{eq:stochasticbound1}. For the bound
  \eqref{eq:stochasticbound2}, it is sufficient to consider the case $\log (s)
  \sim \log (s + 1)$ (which holds whenever $s$ is big enough)
  \begin{eqnarray}
    &  & \mathbb{E} [\| K_j \ast [(J_s (\xi_{\varepsilon}) \circ J_s
    (\xi_{\varepsilon}) - \gamma_{\varepsilon, s}) - (J_{s + \Delta s}
    (\xi_{\varepsilon}) \circ J_{s + \Delta s} (\xi_{\varepsilon}) -
    \gamma_{\varepsilon, s + \Delta s})] \|^2_{L^2}] \nonumber\\
    & \lesssim & \sum_{j \sim \log (s)} \int_{\mathbb{T}^2}
    \int_{\mathbb{T}^2} | K_j (x) | | K_j (y) | (\mathbb{E} [\Delta_j (J_s
    (\xi_{\varepsilon}) - J_{s + \Delta s} (\xi_{\varepsilon}))^2])^{1 / 2}
    \times \nonumber\\
    &  & (\mathbb{E} [| \Delta_j J_s (\xi_{\varepsilon}) |^2] +\mathbb{E} [|
    \Delta_j J_{s + \Delta s} (\xi_{\varepsilon}) |^2])^{3 / 2} \nonumber
  \end{eqnarray}
  In order to estimate the term $\mathbb{E} [| \Delta_j (J_s
  (\xi_{\varepsilon}) - J_{s + \Delta s} (\xi_{\varepsilon})) |^2]$ we note
  that
  \begin{eqnarray}
    \mathbb{E} [| \Delta_j (J_s (\xi_{\varepsilon}) - J_{s + \Delta s}
    (\xi_{\varepsilon})) |^2] & = & \sum_{k \in \mathbb{Z}^2} \frac{| \sigma_s
    (k) - \sigma_{s + \Delta s} (k) |^2 \varphi_j (k)}{(| k |^2 + m^2)}
    \nonumber\\
    & \lesssim & \int_0^{R (s + \Delta s)} | \sigma_s (r) - \sigma_{s +
    \Delta s} (r) |^{2 \alpha} \frac{r (| \sigma_s (r) |^2 + | \sigma_{s +
    \Delta s} (r) |^2)^{1 - \alpha}}{(r^2 + m^2)} \mathd r \nonumber\\
    & \lesssim & | \Delta s |^{\alpha} (\log (s + \Delta s))^{\alpha}
    \nonumber\\
    &  & \left( \int_0^{R (s + \Delta s)} \frac{r (| \sigma_s (r) |^2 + |
    \sigma_{s + \Delta s} (r) |^2)^{1 - \alpha}}{(r^2 + m^2)} \mathd r \right)
    \nonumber\\
    & \lesssim & | \Delta s |^{\alpha} \frac{(\log (s + \Delta
    s))^{\alpha}}{(1 + s)^{(1 - \alpha)}} \lesssim | \Delta s |^{\alpha}
    \frac{1}{(1 + s)^{1 - 2 \alpha}},  \label{eq:inequalityJs2}
  \end{eqnarray}
  where in the last step we do the same computation of
  \eqref{eq:inequalityJS}, taking into account the loss of a power $\alpha$.
  Furthermore we have
  \begin{eqnarray}
    &  & : J_s (\xi_{\varepsilon}) \circ J_s (\xi_{\varepsilon}) - J_{s +
    \Delta s} (\xi_{\varepsilon}) \circ J_{s + \Delta s} (\xi_{\varepsilon}) :
    \nonumber\\
    & = & : (J_s (\xi_{\varepsilon}) - J_{s + \Delta s} (\xi_{\varepsilon}))
    \circ J_s (\xi_{\varepsilon}) : + : J_{s + \Delta s} (\xi_{\varepsilon})
    \circ (J_s (\xi_{\varepsilon}) - J_{s + \Delta s} (\xi_{\varepsilon})) :
    \nonumber
  \end{eqnarray}
  Thus, by the proof of the first part of the present lemma, inequality
  \eqref{eq:inequalityJS} and inequality \eqref{eq:inequalityJs2},
  \begin{eqnarray}
    &  & \mathbb{E} [\| K_j \ast [(J_s (\xi_{\varepsilon}) \circ J_s
    (\xi_{\varepsilon}) - \gamma_{\varepsilon, s}) - (J_{s + \Delta s}
    (\xi_{\varepsilon}) \circ J_{s + \Delta s} (\xi_{\varepsilon}) -
    \gamma_{\varepsilon, s + \Delta s})] \|^2_{L^2}] \nonumber\\
    & \lesssim & \mathbb{E} [\| K_j \ast [: (J_s (\xi_{\varepsilon}) - J_{s +
    \Delta s} (\xi_{\varepsilon})) \circ J_s (\xi_{\varepsilon}) :]
    \|^2_{L^2}] + \nonumber\\
    &  & +\mathbb{E} [\| K_j \ast [: J_{s + \Delta s} (\xi_{\varepsilon})
    \circ (J_s (\xi_{\varepsilon}) - J_{s + \Delta s} (\xi_{\varepsilon})) :]
    \|^2_{L^2}] \nonumber\\
    & \lesssim & | \Delta s |^{\alpha} \frac{1}{(1 + s)^{2 - 2 \alpha}} .
    \nonumber
  \end{eqnarray}
  This concludes the proof for the case $p = 2$. The case of $p > 2$ can be
  obtained by the previous bounds and hypercontractivity applied to the second
  degree Gaussian polynomial $(J_s (\xi_{\varepsilon}) \circ J_s
  (\xi_{\varepsilon}) - \gamma_{\varepsilon, s})$.
\end{proof*}

\begin{remark}
  \label{remark:stochasticestimates2}The result of Lemma
  \ref{lemma:stochasticestimates1} can be easily extended to the case where
  the resonant product $J_s (\xi_{\varepsilon}) \circ J_{\varepsilon}
  (\xi_{\varepsilon})$ is replaced by the standard product $J_s
  (\xi_{\varepsilon}) J_s (\xi_{\varepsilon})$, namely we have the following
  result: for every $\delta > 0, \ell < 1$, there is $c > 0$ such that for any
  $\delta > 0$, $p \geqslant 2$ and there is we have
  \[ \sup_{\varepsilon \in (0, 1)} \mathbb{E} [\| (1 + s)^{\ell} \| J_s
     (\xi_{\varepsilon}) J_s (\xi_{\varepsilon}) - \gamma_{\varepsilon,
     s}^{(2)} \|_{B^{- \delta}_{p, p}} \|_{B^{c \delta}_{p, p}
     (\mathbb{R}_+)}^p] < + \infty \]
\end{remark}

\begin{lemma}
  \label{lemma:inequalitydoubleexpectation1}For any $\delta > 0$ there is
  $\kappa < \frac{1}{2}$ small enough such that we have
  \[ \sup_{\varepsilon \in (0, 1)} \mathbb{E} \left[ \left\| \left( {J_s} 
     (\xi_{\varepsilon} W_s) J_s \xi_{\varepsilon} {- J_s}  \xi_{\varepsilon}
     J_s \xi_{\varepsilon} W_s \right) \right\|_{H^{- \delta}
     (\mathbb{T}^2)}^2 \right] \lesssim (1 + s)^{- 3 + 2 \kappa} . \]
\end{lemma}

\begin{proof}
  We do the explicit computation in the case $\varepsilon = 0$. The case
  $\varepsilon < 1$ and the uniformity of the inequality with respect to $0 <
  \varepsilon < 1$ can be obtained with a similar method. We compute
  \begin{eqnarray}
    &  & \mathbb{E} [\| (J_s (\xi W_s) - (J_s \xi) W_s) J_s \xi \|^2_{H^{-
    \delta}}] \nonumber\\
    & = & \sum_n \langle n \rangle^{- 2 \delta} \mathbb{E} \left( \sum_{n_1}
    \sum_{m_1} (J_s (n_1 - n) - J_s (n_1 - n - m_1)) \xi (m_1 - n_1 - n) W_s
    (m_1) J_s (n_1) \xi (n_1) \right)^2 \nonumber\\
    & = & \sum_n \langle n \rangle^{- 2 \delta} \mathbb{E} \left. \left[
    \sum_{n_1, n_2} \sum_{m_1, m_2} (J_s (n_1 - n) - J_s (n_1 - n - m_1)) (J_s
    (n_2 - n) - J_s (n_2 - n - m_2)) \times \right. \right. \nonumber\\
    &  & \left. \left. \phantom{\sum_m} \hat{\xi} (n_1 - m_1 - n) \hat{\xi}
    (n_2 - m_2 - n) \hat{W}_s (m_1) \hat{W}_s (m_2) J_s (n_1) \hat{\xi} (n_1)
    J_s (n_2) \hat{\xi} (n_2) \right] \right. \nonumber\\
    & = & \sum_n \langle n \rangle^{- 2 \delta} \sum_{n_1, n_2} \sum_{m_1,
    m_2} \left. \left( (J_s (n_1 - n) - J_s (n_1 - n - m_1)) (J_s (n_2 - n) -
    J_s (n_2 - n - m_2)) \times \phantom{\int} \right. \right. \nonumber\\
    &  & \left. \left. \phantom{\int} J_s (n_1) J_s (n_2) \mathbb{E} \left[
    \hat{\xi} (n_1) \hat{\xi} (n_2) \hat{\xi} \left( n_1 - m_1 - n \right)
    \hat{\xi} (n_2 - m_2 - n) \right] \mathbb{E} [\hat{W}_s (m_1) \hat{W}_s
    (m_2)] \right) \right. \nonumber\\
    & = & \sum_n \langle n \rangle^{- 2 \delta} \sum_{n_1} \sum_{m_1
    \leqslant s} (J_s (n_1 - n) - J_s (n_1 - m_1 - n))^2 \frac{1}{| m + m_1
    |^2} J^2_s (n_1) \nonumber\\
    &  & + \sum_{n \leqslant s} \langle n \rangle^{- 2 \delta} \sum_{n_1,
    n_2} \frac{1}{| m + n |^2} (J_s (n_1 - n) - J_s (n_1)) (J_s (n_2 - n) -
    J_s (n_2)) J_s (n_1) J_s (n_2) \nonumber\\
    &  & + \sum_{n \leqslant s} \langle n \rangle^{- 2 \delta} \sum_{n_1}
    \frac{1}{| m + n |^2} (J_s (n_1 - n) - J_s (n_1)) (J_s (n_2 - n) - J_s
    (n_2)) J_s (n_1) J_s (n_2) \nonumber\\
    & = & \Iota + \Iota \Iota + \Iota \Iota \Iota 
    \label{eq:inequalitydoubleexpectation1}
  \end{eqnarray}
  Estimate on $\Iota$:
  
  Observe that $J_s$ is supported in an annulus of radius $s$ and we are
  restricting to $m_1 \leqslant s$. This means we can rewrite
  \begin{eqnarray*}
    &  & \mathbbm{1}_{\{ | m_1 | \leqslant s \}} (J_s (n_1 - n) - J_s (n_1 -
    m_1 - n))^2\\
    & = & \mathbbm{1}_{\{ | m_1 | \leqslant 2 | n - n_1 | \}} (J_s (n_1 - n)
    - J_s (n_1 - m_1 - n))^2 + \left( {\mathbbm{1}_{\{ | n - n_1 | \leqslant |
    m_1 | / 2 \}}}  \right) J_s (n_1 - n) - J_s (n_1 - m_1 - n)\\
    & \leqslant & \langle s \rangle^{- 1 - \kappa} \langle n_1 - n \rangle^{-
    2 + 2 \kappa} \langle m_1 \rangle^{- \kappa} .
  \end{eqnarray*}
  \begin{eqnarray*}
    &  & \mathbbm{1}_{\{ | m_1 | \leqslant s \}} (J_s (n_1 - n) - J_s (n_1 -
    m_1 - n))^2\\
    & \leqslant & \mathbbm{1}_{\{ | m_1 | \leqslant 2 | n - n_1 | \}} (J_s
    (n_1 - n) - J_s (n_1 - m_1 - n))^2 + \left( {\mathbbm{1}_{\{ | n - n_1 |
    \leqslant | m_1 | / 2 \}}}  \right) J^2_s (n_1 - n - m_1)\\
    & \leqslant & \langle s \rangle^{- 1 - \kappa} \langle n_1 - n \rangle^{-
    2 + 2 \kappa} \langle m_1 \rangle^{- \kappa}
  \end{eqnarray*}
  for some $\kappa > 0$ small enough. Then plugging the previous inequality in
  \eqref{eq:inequalitydoubleexpectation1} we get
  \begin{eqnarray*}
    &  & \mathbb{E} [\| (J_s (\xi W_s) - (J_s \xi) W_s) J_s \xi \|^2_{H^{-
    \delta}}]\\
    & \leqslant &  \sum_n \langle n \rangle^{- 2 \delta} \sum_{n_1} \sum_{m_1
    \leqslant s} \langle s \rangle^{- 3 - 2 \kappa} \langle n_1 - n \rangle^{-
    2 + 2 \kappa} \langle m_1 \rangle^{- 2 - \kappa} \langle n_1 \rangle^{-
    2}\\
    & \lesssim & \langle s \rangle^{- 3 + 2 \kappa}
  \end{eqnarray*}
  Estimate on $\Iota \Iota$: One can easily check that
  \begin{eqnarray*}
    \nabla_k J_s (k) & \lesssim & \mathbbm{1}_{k \backsim s} (\langle s
    \rangle^{- 1 / 2} \langle k \rangle^{- 2} + \langle s \rangle^{- 3 / 2}
    \langle k \rangle^{- 1})\\
    & \lesssim & \langle t \rangle^{- 5 / 2}
  \end{eqnarray*}
  So
  \[ | (J_s (n_1 - n) - J_s (n_1)) | \lesssim \langle t \rangle^{- 5 / 2} | n
     | . \]
  and we have $J_s (k) \lesssim t^{- 3 / 2}$.
  
  Recall also that $| J_s (n_1) | \lesssim \mathbbm{1}_{| n_1 | \lesssim s}
  \langle s \rangle^{- 1 / 2} \langle n_1 \rangle$ so $\sum_{n_1} J_s (n_1)
  \lesssim \langle s \rangle^{1 / 2}$
  
  Plugging this into the sum we get
  \begin{eqnarray*}
    &  & \sum_{n \leqslant s} \langle n \rangle^{- 2 s} \sum_{n_1} \frac{1}{|
    m + n |^2} (J_s (n_1 - n) - J_s (n_1))\\
    &  & \times (J_s (n_2 - n) - J_s (n_2)) J_s (n_1) J_s (n_2)\\
    & \lesssim & \sum_{n \leqslant s} \langle n \rangle^{- 2 s} \frac{1}{| m
    + n |^2} \langle s \rangle^{- 8 / 2} | n | \sum_{n_1, n_2} | J_s (n_1) J_s
    (n_2) |\\
    & \lesssim & \sum_{n \leqslant s} \langle n \rangle^{- 2 s} \frac{1}{| m
    + n |^2} \langle s \rangle^{- 4} \langle s \rangle | n |\\
    & \lesssim & \sum_{n \leqslant s} \langle n \rangle^{- 2 s} \frac{1}{| m
    + n |} \langle s \rangle^{- 3}\\
    & \lesssim & \langle s \rangle^{- 3}
  \end{eqnarray*}
  So
  \begin{eqnarray*}
    &  & \int^{\infty}_0 \mathbb{E} [\| (J_s (\xi_{\varepsilon} W_s) - (J_s
    \xi_{\varepsilon}) W_s) J_s \xi_{\varepsilon} \|_{H^{- s}}] \mathd s\\
    & \leqslant & \int^{\infty}_0 \mathbb{E} [\| (J_s (\xi_{\varepsilon} W_s)
    - (J_s \xi) W_s) J_s \xi_{\varepsilon} \|^2_{H^{- s}}]^{1 / 2} \mathd s\\
    & \lesssim & \int \langle s \rangle^{- 3 / 2} \mathd s.
  \end{eqnarray*}
  Finally the estimate on $\Iota \Iota \Iota$ is a simpler version of the
  estimate on $\Iota \Iota$.
\end{proof}

\begin{lemma}
  \label{lemma:stochasticestimates3}For any $0 < \kappa < \delta$ and for any
  $p \geqslant 2$ we have
  \[ \sup_{\varepsilon \in (0, 1)} \mathbb{E} \left[ \left\| \left( {J_s} 
     (\xi_{\varepsilon} W_s) J_s \xi_{\varepsilon} -
     \dot{\gamma}_{\varepsilon, s}^{(2)} W_s \right) \right\|_{H^{- \delta}}^p
     \right] \lesssim (1 + s)^{(- 1 - \kappa) p} . \]
\end{lemma}

\begin{proof}
  We have that
  \begin{eqnarray}
    &  & \left\| \left( {J_s}  (\xi_{\varepsilon} W_s) J_s \xi_{\varepsilon}
    - \dot{\gamma}_{\varepsilon, s}^{(2)} W_s \right) \right\|_{H^{- \delta}}
    \nonumber\\
    & \leqslant & \left\| \left( {J_s}  (\xi_{\varepsilon} W_s) J_s
    \xi_{\varepsilon} {- J_s}  \xi_{\varepsilon} J_s \xi_{\varepsilon} W_s
    \right) \right\|_{H^{- \delta}} + \left\| \left( {J_s} 
    (\xi_{\varepsilon}) J_s \xi_{\varepsilon} - \dot{\gamma}_{\varepsilon,
    s}^{(2)} \right) W_s \right\|_{H^{- \delta}} . \nonumber
  \end{eqnarray}
  The first term has been estimated in Lemma
  \ref{lemma:inequalitydoubleexpectation1}. To estimate the second part denote
  \[ f = \left( {J_s}  (\xi_{\varepsilon}) J_s \xi_{\varepsilon} -
     \dot{\gamma}_{\varepsilon, s}^{(2)} \right) . \]
  Note that $f$ is independent of $\mathcal{F}'$. We will show that for any
  $\kappa > 0$
  \[ \mathbb{E}_{w'} [\| f W_s \|^2_{H^{- \delta}}] \lesssim s^{\kappa / 2 +
     \delta} \| f \|^2_{H^{- \delta}} . \]
  Indeed
  \begin{eqnarray*}
    &  & \mathbb{E}_{w'} [\| f W_s \|^2_{H^{- \delta}}]\\
    & = & \sum_{n \in \mathbb{Z}^2} \langle n \rangle^{- 2 \delta} \mathbb{E}
    \left( \sum_{k \in \mathbb{Z}^2} \hat{f} (n - k) \hat{W}_s (k) \right)^2\\
    & \lesssim & \sum_{n \in \mathbb{Z}^2} \langle n \rangle^{- 2 \delta}
    \sum_{k \in \mathbb{Z}^2, | k | \lesssim s} \hat{f} (n - k)^2
    \frac{1}{\langle k \rangle^2}\\
    & \lesssim & s^{\kappa} \sum_{n \in \mathbb{Z}^2} \sum_{k \in
    \mathbb{Z}^2} \langle n \rangle^{- 2 \delta} \frac{1}{\langle k \rangle^{2
    + \kappa}} \hat{f} (n - k)^2\\
    & \lesssim & \sum_{n \in \mathbb{Z}^2} \sum_{k \in \mathbb{Z}^2}
    \frac{1}{\langle k \rangle^{2 + \kappa - 2 \delta}} \frac{1}{\langle n - k
    \rangle^{2 \delta}} \hat{f} (n - k)^2
  \end{eqnarray*}
  Now by Young's convolution inequality
  \[ \left\| \sum_{k \in \mathbb{Z}^2} \sum_{k \in \mathbb{Z}^2}
     \frac{1}{\langle k \rangle^{2 + \kappa - 2 \delta}} \frac{1}{\langle n -
     k \rangle^{2 \delta}} \hat{f} (n - k)^2 \right\|_{l_n^1} \lesssim \left\|
     \frac{1}{\langle \cdot \rangle^{2 \delta}} \hat{f} (\cdot)^2
     \right\|_{l^1} \left\| \frac{1}{\langle \cdot \rangle^{2 + \kappa - 2
     \delta}} \right\|_{l^1} \]
  from which we can conclude.
  
  In the case $p = 2$, the result then follows from Lemma
  \ref{lemma:stochasticestimates1} (see also Remark
  \ref{remark:stochasticestimates2}) and Lemma
  \ref{lemma:inequalitydoubleexpectation1}. The general case can be proved
  using the fact that $\left( {J_s}  (\xi_{\varepsilon}) J_s \xi_{\varepsilon}
  - \dot{\gamma}_{\varepsilon, s}^{(2)} \right) W_s$ is a third degree
  polynomial and then applying hypercontractivity.
\end{proof}

\begin{lemma}
  \label{lemma:stochasticestimates4}We have that for any $\delta > 0$ and $p
  \geqslant 1$ there is $0 < \kappa \ll 1$ for which
  \[ \sup_{\varepsilon \in (0, 1)} \mathbb{E} [\| \xi_{\varepsilon} W_s
     \|_{B^{- 1 - \delta}_{p, p} (\mathbb{T}^2)}^p] \lesssim (\log (1 + s))^p,
  \]
  \[ \sup_{\varepsilon \in (0, 1)} \mathbb{E} \left[ \int_{- 1}^1 \frac{\|
     \xi_{\varepsilon} W_s - \xi_{\varepsilon} W_{s + \Delta s} \|_{B^{- 1 -
     \delta}_{p, p} (\mathbb{T}^2)}^p}{| \Delta s |^{1 + \alpha p}} \mathd
     \Delta s \right] \lesssim (1 + s)^{2 \kappa p} . \]
\end{lemma}

\begin{proof}
  The proof is similar to the one of Lemma \ref{lemma:stochasticestimates1},
  we report here only the main steps. First we note that, since $W_s$ is
  independent of $\xi_{\varepsilon}$, we have $W_s \xi_{\varepsilon} = : W_s
  \xi_{\varepsilon} :$ this means that
  \begin{eqnarray}
    \mathbb{E} [\| K_j \ast : W_s \xi_{\varepsilon} : \|_{L^2}^2] & \lesssim &
    \int_0^s \int_{\mathbb{T}^2} \int_{\mathbb{T}^2} (J_{\tau}^{\ast 2} (0))
    K_j (x) K_j (x) \mathd x \mathd \tau \nonumber\\
    & \lesssim & \left( \int_0^s \| J_{\tau}^{\ast 2} \|_{L^{\infty}} \mathd
    \tau \right) \left( \int_{\mathbb{T}^2} | K_j |^2 \mathd x \right)
    \lesssim \left( \int_0^s \| J_{\tau}^{\ast 2} \|_{L^{\infty}} \mathd \tau
    \right) 2^{2 j} \nonumber
  \end{eqnarray}
  We also have
  \[ \left( \int_0^s \| J_{\tau}^{\ast 2} \|_{L^{\infty}} \mathd \tau \right)
     = \int_0^t \sum_{k \in \mathbb{Z}^2} \frac{| \sigma_{\tau} (k) |^2}{(| k
     |^2 + 1)} \lesssim \sum_{k \in \mathbb{Z}^2} \frac{\rho_t (k)}{(| k |^2 +
     1)} \lesssim \log (1 + s) . \]
  The second bound and the generic case $p > 2$ can be obtained as in Lemma
  \ref{lemma:stochasticestimates1}.
\end{proof}

\subsection{Analytical estimates}\label{subsection:analyticalestimates}

We want now prove some estimates on the terms $\Gamma_i$, $i = 1, \ldots, 7$,
appearing in the expansion \eqref{eq:functional-renom} of $F$. The main aim is
to prove some upper bounds depending on the sums involving either purely
stochastic terms or the positive terms$\int_0^{\infty} \| u_s \|_{L^2}^2
\mathd s$ or $\int_0^{\infty} \frac{\| Z_s \|_{L^2}^2}{(1 + s)^{1 + \kappa}}
\mathd s$ since this kind of term can be compensated by the positive parts of
$F$ (namely $\mathfrak{G} + \frac{1}{2} \int^{\infty}_0 \| l_t (u) \|^2_{L^2}
\mathd s$). \ \

\begin{lemma}
  \label{lemma:inequalityZu}For every $0 < \delta$ we have
  \[ \sup_{t \in [0, \infty]} \| Z_t (u) \|^2_{H^{1 - \delta}} \leqslant
     \int^{\infty}_0 \| u_s \|^2_{H^{- \delta}} \mathd s \]
\end{lemma}

\begin{proof}
  We have that
  \begin{eqnarray}
    \| Z_t (u) \|^2_{H^{1 - \delta}} & = & \sum_{k \in \mathbb{Z}^2} (m^2 + |
    k |^2)^{1 - \delta} | \mathcal{F} (Z_t) (k) |^2 \nonumber\\
    & = & \sum_{k \in \mathbb{Z}^2} (m^2 + | k |^2)^{1 - \delta} \left|
    \int_0^t \left( \frac{\mathd}{\mathd s} \rho_s (k) \right)^{1 / 2}
    \frac{\mathcal{F} (u_s) (k)}{\sqrt{(m^2 + | k |^2)}} \mathd s \right|^2
    \nonumber\\
    & = & \sum_{k \in \mathbb{Z}^2} (m^2 + | k |^2)^{- \delta} \left(
    \int_0^t \frac{\mathd}{\mathd s} \rho_s (k) \mathd s \right) \left(
    \int_0^t | \mathcal{F} (u_s) (k) |^2 \mathd s \right) \nonumber\\
    & = & \int_0^t \sum_{k \in \mathbb{Z}^2} | \rho_t (k) | (m^2 + | k
    |^2)^{- \delta} | \mathcal{F} (u_s) (k) |^2 \mathd s \nonumber\\
    & \leqslant & \int_0^t \sum_{k \in \mathbb{Z}^2} (m^2 + | k |^2)^{-
    \delta} | \mathcal{F} (u_s) (k) |^2 \mathd s = \int_0^t \| u_s \|_{H^{-
    \delta}}^2 \mathd s \leqslant \int_0^{\infty} \| u_s \|_{H^{- \delta}}^2
    \mathd s. \nonumber
  \end{eqnarray}
  
\end{proof}

We collect the bounds of the $\Gamma_i$ appearing in Proposition
\ref{proposition:reformulationF}.

\begin{lemma}
  For any $\kappa \in (0, 1)$ and for any $0 < \delta \ll 1$ there are $\alpha
  > 0$ and $0 < \theta < 1$, $0 < \lambda < \tau < \delta \ll 1$, $\eta > 0$,
  and $0 < \ell \ll 1$ such that we get
  \begin{eqnarray}
    | \Gamma_1 | & \lesssim & \kappa \int_0^{\infty} \| u_s \|_{H^{-
    \delta}}^2 \mathd s + \frac{1}{\kappa^{\alpha}} \| \xi_{\varepsilon}
    \|_{\mathcal{C}^{- 1 - \delta}}^{\frac{2}{\delta}} \int_0^{+ \infty}
    \frac{\| Z_t \|_{L^2}^2}{(1 + t)^{3 / \delta - 8}} \mathd t \nonumber\\
    | \Gamma_2 | & \lesssim & \kappa \int_0^{+ \infty} \| u_t \|_{H^{-
    \delta}}^2 \mathd t + \frac{1}{\kappa^{\alpha}} \left( \int^{\infty}_0
    \left\| \left( {J_s}  \xi_{\varepsilon} W_s \circ J_s \xi_{\varepsilon} -
    \dot{\gamma}_{\varepsilon, s}^{(2)} W_s \right) \right\|_{H^{- \delta}}
    \mathd s \right)^2 \nonumber\\
    | \Gamma_3 | & \lesssim & \kappa \int_0^{+ \infty} \| u_t \|_{H^{-
    \delta}}^2 \mathd t + \frac{1}{\kappa^{\alpha}} \int^{\infty}_0
    \frac{(\sup (1 + s)^{1 + \tau - \lambda} \| ((J_s \xi_{\varepsilon} \circ
    J_s \xi_{\varepsilon}) - \dot{\gamma}_{\varepsilon, s}^{(2)})
    \|_{\mathcal{C}^{- \delta}})^{\frac{1}{1 - \theta}}}{(1 + s)^{1 + \tau -
    \lambda}} \| Z_s \|_{L^2}^2 \mathd s \nonumber\\
    | \Gamma_4 | & \lesssim & \kappa \int_0^{+ \infty} \| u_t \|_{H^{-
    \delta}}^2 \mathd t + \frac{1}{\kappa^{\alpha}} \int_0^{+ \infty} \frac{\|
    \xi_{\varepsilon} \|_{\mathcal{C}^{- 1 - \delta}}^2 \| Z_s \|_{L^2}^2}{(1
    + t)^{1 + \eta}} \mathd t \nonumber
\end{eqnarray}
as well as
    \begin{eqnarray}
    | \Gamma_5 | & \lesssim & \kappa \int_0^t \| u \|_{H^{- \delta}}^2 \mathd
    t + \frac{1}{\kappa^{\alpha}} \| \xi \|_{\mathcal{C}^{- 1 - \delta}}^2
    (\sup_{s \in \mathbb{R}_+} (1 + s)^{\ell} \| \xi W_s \|_{\mathcal{C}^{- 1
    - \delta}}^2) \nonumber\\
    | \Gamma_6 | & \lesssim & \kappa \int^{\infty}_0 \| u_t \|_{H^{-
    \delta}}^2 \mathd t + \frac{1}{\kappa^{\alpha}} \left( \int^{\infty}_0
    \frac{| \gamma_{\varepsilon, t}^{(2)} |^2}{(1 + t)^{3 / 2 - 2 \delta}} \|
    W_t \|_{\mathcal{C}^{- \delta}}^2 \mathd t + \int^{\infty}_0 \frac{|
    \gamma_{\varepsilon, t}^{(2)} |^{\frac{2}{1 - \delta}}}{(1 + t)^{3 / 2}}
    \| Z_t \|_{L^2}^2 \mathd t \right) \nonumber\\
    | \Gamma_7 | & \lesssim & \kappa \int_0^{+ \infty} \| u_t \|_{H^{-
    \delta}}^2 \mathd t + \frac{1}{\kappa^{\alpha}} \left( \int_0^{+ \infty} |
    \dot{\gamma}^{(1)}_t (\omega) | \| W_t \|_{\mathcal{C}^{- \delta}} \mathd
    t \right)^2 \nonumber\\
    &  & + \frac{1}{\kappa^{\alpha}} \int_0^{+ \infty} | \gamma_t^{(1)} |^2
    \frac{\| W_t \|_{\mathcal{C}^{- \delta}}^2}{(1 + t)^{3 / 2 - 2 \delta}}
    \mathd t + \frac{1}{\kappa^{\alpha}} \int^{\infty}_0 \frac{|
    \gamma_t^{(1)} |^{\frac{2}{1 - \delta}}}{(1 + t)^{3 / 2}} \| Z_t
    \|_{L^2}^2 \mathd t. \nonumber
  \end{eqnarray}
  where the implied constants in the $\lesssim$ do not depend on $\kappa > 0$,
  where $\Gamma_1, \ldots, \Gamma_7$ are defined as in Proposition
  \ref{proposition:reformulationF}.
\end{lemma}

\begin{proof}
  The proof is essentially an application of the results in Section
  \ref{section:stochasticestimates}, Lemma \ref{lemma:inequalityZu}, Besov
  embeddings, products properties and Young's inequality, see Appendix
  \ref{app:Besov}. We report here only the main passages of the computations.
  
  $\Gamma_1$ can be bounded as follows
  \begin{eqnarray}
    | \Gamma_1 | & \lesssim & \int^{\infty}_0 (1 + t)^{- 3 / 2 + 4 \delta} \|
    \xi_{\varepsilon} \preccurlyeq Z_t \|_{H^{- 3 \delta}} \| u_t \|_{H^{-
    \delta}} \mathd t \nonumber\\
    & \lesssim & \int^{\infty}_0 (1 + t)^{- 3 / 2 + 4 \delta} \|
    \xi_{\varepsilon} \|_{\mathcal{C}^{- 1 - \delta}} \| Z_t \|_{H^{1 -
    \delta}}^{1 - \delta} \| Z \|^{\delta}_{L^2} \| u_t \|_{H^{- \delta}}
    \mathd t \nonumber\\
    & \lesssim & \kappa \int_0^{\infty} \| u_s \|_{H^{- \delta}}^2 \mathd s +
    \frac{1}{\kappa^{\alpha}} \| \xi_{\varepsilon} \|_{\mathcal{C}^{- 1 -
    \delta}}^{\frac{2}{\delta}} \int_0^{+ \infty} (1 + t)^{- 3 / \delta + 8}
    \| Z_t \|_{L^2}^2 \mathd t. \nonumber
  \end{eqnarray}
  For $\Gamma_2$ we have
  \begin{eqnarray}
    | \Gamma_2 | & \lesssim & \sup_{s \in \mathbb{R}_+} \| Z_s \|_{H^{\delta}}
    \int^{\infty}_0 \left\| \left( {J_s}  \xi_{\varepsilon} W_s \circ J_s
    \xi_{\varepsilon} - \dot{\gamma}_{\varepsilon, s}^{(2)} W_s \right)
    \right\|_{H^{- \delta}} \mathd s \nonumber\\
    & \lesssim & \kappa \int_0^{+ \infty} \| u_t \|_{H^{- \delta}}^2 \mathd t
    + \frac{1}{\kappa^{\alpha}} \left( \int^{\infty}_0 \left\| \left( {J_s} 
    \xi_{\varepsilon} W_s \circ J_s \xi_{\varepsilon} -
    \dot{\gamma}_{\varepsilon, s}^{(2)} W_s \right) \right\|_{H^{- \delta}}
    \mathd s \right)^2 . \nonumber
  \end{eqnarray}
  We bound $\Gamma_3$ as follows
  \begin{eqnarray}
    | \Gamma_3 | & \lesssim & \int^{\infty}_0 \| ((J_s \xi_{\varepsilon} \circ
    J_s \xi_{\varepsilon}) - \dot{\gamma}_{\varepsilon, s}^{(2)})
    \|_{\mathcal{C}^{- \delta}} \| Z^2_s \|_{B^{\delta}_{1, 1}} \mathd s
    \nonumber\\
    & \lesssim & \int^{\infty}_0 \frac{1}{(1 + s)^{1 + \tau - \lambda}} (\sup
    (1 + s)^{1 + \tau - \lambda} \| ((J_s \xi_{\varepsilon} \circ J_s
    \xi_{\varepsilon}) - \dot{\gamma}_{\varepsilon, s}^{(2)})
    \|_{\mathcal{C}^{- \delta}}) \times \nonumber\\
    &  & \times \| Z _s \|_{H^{1 - \delta}}^{2 \theta} \| Z_s \|_{L^2}^{2 (1
    - \theta)} \mathd s \lesssim \kappa \int_0^{+ \infty} \| u_t \|_{H^{-
    \delta}}^2 \mathd t + \nonumber\\
    &  & + \frac{1}{\kappa^{\alpha}} \int^{\infty}_0 \frac{(\sup (1 + s)^{1 +
    \tau - \lambda} \| ((J_s \xi_{\varepsilon} \circ J_s \xi_{\varepsilon}) -
    \dot{\gamma}_{\varepsilon, s}^{(2)}) \|_{\mathcal{C}^{-
    \delta}})^{\frac{1}{1 - \theta}}}{(1 + s)^{1 + \tau - \lambda}} \| Z_s
    \|_{L^2}^2 \mathd s. \nonumber
  \end{eqnarray}
  By Proposition \ref{proposition:commutatorJs1} equation \eqref{eq:Nikolay1}
  we have for $\Gamma_4$ the bound
  \begin{eqnarray}
    | \Gamma_4 | & \lesssim & \int_0^{+ \infty} (1 + t)^{- 1 + \delta} \|
    \xi_{\varepsilon} \|_{\mathcal{C}^{- 1 - \delta}} \| Z_t \|_{H^{1 / 2 -
    \delta / 2}}^2 \mathd t \nonumber\\
    & \lesssim & \int_0^{+ \infty} (\sup_{s \in \mathbb{R}_+} \| Z_s \|_{H^{1
    - \delta}}) (1 + t)^{- 1 - \eta} \| \xi_{\varepsilon} \|_{\mathcal{C}^{- 1
    - \delta}} \| Z_t \|_{L^2} \mathd t \nonumber\\
    & \lesssim & \kappa \int_0^{+ \infty} \| u_t \|_{H^{- \delta}}^2 \mathd t
    + \frac{1}{\kappa} \int_0^{+ \infty} (1 + t)^{- 1 - \eta} \|
    \xi_{\varepsilon} \|_{\mathcal{C}^{- 1 - \delta}}^2 \| Z_t \|_{L^2}^2
    \mathd t. \nonumber
  \end{eqnarray}
  Using \eqref{eq:Nikolay2} from Proposition \ref{proposition:commutatorJs1}
  $\Gamma_5$ can be estimated as
   \begin{eqnarray*}
       | \Gamma_5 | & \lesssim & \int_0^{+ \infty} (1 + s)^{- \frac{3}{2}} \|
       Z \|_{H^{1 - \delta}} \| \xi \|_{\mathcal{C}^{- 1 - \delta}} \| \xi W_s
       \|_{\mathcal{C}^{- 1 - \delta}} \mathd s \nonumber\\
       & \lesssim & \kappa \int_0^t \| u \|_{H^{- \delta}}^2 \mathd t +
       \frac{1}{\kappa} \| \xi_{\varepsilon} \|_{\mathcal{C}^{- 1 - \delta}}^2
       (\sup_{s \in \mathbb{R}_+} (1 + s)^{- 2 \ell} \| \xi_{\varepsilon} W_s
       \|_{\mathcal{C}^{- 1 - \delta}}^2) . \nonumber
     \end{eqnarray*} 
  For $\Gamma_6$ we bound
  \begin{eqnarray}
    | \Gamma_6 | & \lesssim & \int^{\infty}_0 \frac{\gamma_{\varepsilon,
    t}^{(2)}}{(1 + t)^{3 / 2 - 2 \delta}} \| W_t \|_{\mathcal{C}^{- \delta}}
    \| u_t \|_{H^{- \delta}} + \int^{\infty}_0 \frac{\gamma_{\varepsilon,
    t}^{(2)}}{(1 + t)^{3 / 2}} \| Z_t \|_{H^{\delta}} \| u_t \|_{H^{- \delta}}
    \mathd t \nonumber\\
    & \lesssim & \kappa \int^{\infty}_0 \| u_t \|_{H^{- \delta}}^2 \mathd t +
    \frac{1}{\kappa^{\alpha}} \left( \int^{\infty}_0 \frac{|
    \gamma_{\varepsilon, t}^{(2)} |^2}{(1 + t)^{3 / 2 - 2 \delta}} \| W_t
    \|_{\mathcal{C}^{- \delta}}^2 \mathd t + \int^{\infty}_0 \frac{|
    \gamma_{\varepsilon, t}^{(2)} |^{\frac{2}{1 - \delta}}}{(1 + t)^{3 / 2}}
    \| Z_t \|_{L^2}^2 \mathd t \right) . \nonumber
  \end{eqnarray}
  Lastly, for $\Gamma_7$ we estimate
  \begin{eqnarray}
    | \Gamma_7 | & \lesssim & \int_0^{+ \infty} | \dot{\gamma}^{(1)}_t
    (\omega) | \| W_t \|_{\mathcal{C}^{- \delta}} \| Z_t \|_{H^{\delta}}
    \mathd t + \int^{\infty}_0 | \gamma_t^{(1)} | \frac{\| W_t
    \|_{\mathcal{C}^{- \delta}}}{(1 + t)^{3 / 2 - 2 \delta}} \| u_t \|_{H^{-
    \delta}} \mathd t + \nonumber\\
    &  & + \int^{\infty}_0 \frac{| \gamma_t^{(1)} |}{(1 + t)^{3 / 2}} \| Z_t
    \|_{H^{\delta}} \| u_t \|_{H^{- \delta}} \mathd t \nonumber\\
    & \lesssim & \kappa \int_0^{+ \infty} \| u_t \|_{H^{- \delta}}^2 \mathd t
    + \frac{1}{\kappa^{\alpha}} \left( \int_0^{+ \infty} |
    \dot{\gamma}^{(1)}_t (\omega) | \| W_t \|_{\mathcal{C}^{- \delta}} \mathd
    t \right)^2 \nonumber\\
    &  & + \frac{1}{\kappa^{\alpha}} \int_0^{+ \infty} | \gamma_t^{(1)} |^2
    \frac{\| W_t \|_{\mathcal{C}^{- \delta}}^2}{(1 + t)^{3 / 2 - 2 \delta}}
    \mathd t + \frac{1}{\kappa^{\alpha}} \int^{\infty}_0 \frac{|
    \gamma_t^{(1)} |^{\frac{2}{1 - \delta}}}{(1 + t)^{3 / 2}} \| Z_t
    \|_{L^2}^2 \mathd t. \nonumber
  \end{eqnarray}
  
\end{proof}

\subsection{Bounds on $F^{\varepsilon}$}

We start with a preliminary result that gives us the existence of a family of
reference drifts depending on $\xi (\omega)$ which will be needed later on.

\begin{lemma}
  \label{lemma:drift}There exists a family of adapted processes
  $u^{\varepsilon}$ such that
  \[ \sup_{\varepsilon > 0} \mathbb{E} \left[ \int^{\infty}_0 \|
     l^{\varepsilon}_s (u^{\varepsilon}) \|^2_{L^2} \mathd t \right] < \infty
     \quad \text{and} \quad \sup_{\varepsilon > 0} F^{\varepsilon}
     (u^{\varepsilon}, \omega) < \infty, \]
  almost surely with respect to $\omega$.
\end{lemma}

\begin{proof}
  Take $u^{\varepsilon}$ to be a solution to the equation
  \begin{equation}
    u_s^{\varepsilon} = - \mathbbm{1}_{s > T} J_s W_s \xi_{\varepsilon} -
    \mathbbm{1}_{s > T} J_s (\xi_{\varepsilon} \succ Z_s (u))
    \label{eq:test-drift}
  \end{equation}
  for some fixed $T > 0$ to be chosen later, independently of $\varepsilon$.
  Assume for the moment that this solution exists and satisfies
  \begin{equation}
    \sup_{\varepsilon > 0} \mathbb{E} \left[ \int^{\infty}_0 \|
    u_t^{\varepsilon} \|^2_{H^{- \delta}} \mathd t \right] < \infty
    \label{bound:test-drift}
  \end{equation}
  for some small $\delta > 0$. \ Then
  \[ l_s^{\varepsilon} (u) = - \mathbbm{1}_{s \leqslant T} J_s W_s
     \xi_{\varepsilon} - \mathbbm{1}_{s \leqslant T} J_s (\xi_{\varepsilon}
     \succ Z_s (u)) \]
  so
  \begin{eqnarray*}
    \mathbb{E} \left[ \int^{\infty}_0 \| l^{\varepsilon}_s (u^{\varepsilon})
    \|^2_{L^2} \mathd t \right] & \lesssim & \mathbb{E} \int^T_0 \| J_s W_s
    \xi_{\varepsilon} \|^2_{L^2} \mathd s +\mathbb{E} \int^T_0 \| J_s
    (\xi_{\varepsilon} \succ Z_s (u)) \|^2_{L^2} \mathd s\\
    & \lesssim & C (T) + C (T) \| Z_s (u) \|^2_{H^{1 - \delta}}\\
    & \lesssim & C (T) \left( 1 + \sup_{\varepsilon} \mathbb{E} \left[
    \int^{\infty}_0 \| u_t^{\varepsilon} \|^2_{H^{- \delta}} \mathd t \right]
    \right)
  \end{eqnarray*}
  From Section \ref{subsection:analyticalestimates} we have that
  \[ | F^{\varepsilon} (u^{\varepsilon}, \omega) | \lesssim 1 +\mathbb{E}
     \left[ \int^{\infty}_0 \| l^{\varepsilon}_s (u^{\varepsilon}) \|^2_{L^2}
     \mathd t \right] . \]
  This proves the assertion.
  
  \
  
  Now let us establish that \eqref{eq:test-drift} has a solution. Consider
  the map
  \[ \Phi (u) = - \mathbbm{1}_{s > T} J_s W_s \xi_{\varepsilon} -
     \mathbbm{1}_{s > T} J_s (\xi_{\varepsilon} \succ Z_s (u)) . \]
  We show that $\Phi$ is a contraction for $T$ large enough. Indeed
  \begin{eqnarray*}
    \mathbb{E} \| \Phi (u) \|^2_{H^{- \delta}} & \lesssim & - \mathbbm{1}_{s >
    T} \mathbb{E} [\| J_s W_s \xi_{\varepsilon} \|^2_{H^{- \delta}}]
    +\mathbb{E} \| J_s (\xi_{\varepsilon} \succ Z_s (u)) \|^2_{H^{- \delta}}\\
    & \lesssim & T^{- \delta} \sup_s \frac{1}{s^{2 + \delta}} \mathbb{E} [\|
    W_s \xi_{\varepsilon} \|^2_{H^{- 1 - \delta / 2}}] + T^{- \delta}
    \frac{1}{s^{2 + \delta}} \| \xi \|^2_{\mathcal{C}^{- 1 - \delta / 2}}
    \mathbb{E} \left[ \int^s_0 \| u_t \|^2_{H^{- \delta}} \mathd t \right]
  \end{eqnarray*}
  Using Lemma \ref{lemma:stochasticestimates4} it is not hard to see that
  $\Phi (u)$ is a contraction in a large enough ball $B (0, K) \subset L^2
  (\mathbb{P}', L^2 (\mathbb{R}_+, H^{- \delta}))$ and thus also satisfies
  \eqref{bound:test-drift}.
\end{proof}

\

With this in hand, we prove a uniform in $\varepsilon$ coercivity result for
$F^{\varepsilon} (\cdot, \omega)$ up to a renormalization constant (which in
this particular case is simply $F^{\varepsilon}
(u^{\varepsilon}, \omega)$ and corresponds to the diverging normalization
constant $\log (\mathcal{Z}^{\varepsilon} (\omega))$ in Theorem
\ref{theorem:functionalF}).

\begin{theorem}
  \label{theorem:Fbounds}There is a sequence $(\varepsilon_n)_{n \in
  \mathbb{N}} \subset (0, 1)$, $\varepsilon_n \rightarrow 0$, such that there
  exists a (non-negative) random variable $C : \Omega \rightarrow
  \mathbb{R}_+$ which is almost surely finite and such that for any $\omega
  \in \Omega$ we have
  \begin{equation}
   F^{\varepsilon_n} (u, \omega) - F^{\varepsilon_n}
    (u^{\varepsilon_n}, \omega) \geqslant \frac{1}{4} \mathbb{E} \left[
    \int^{\infty}_0 \| l^{\varepsilon_n}_s (u) \|^2_{L^2} \mathd s \right] - C
    (\omega) \geqslant - C (\omega) \label{eq:inqeualityFe},
  \end{equation}
 where $u^{\varepsilon_n} $is as in Lemma
  \ref{lemma:drift}.
\end{theorem}

We begin by proving the following useful lemmas.

\begin{lemma}
  \label{lemma:Lpconvergence}Let $C_{\cdot, \varepsilon} : \mathbb{R}_+ \times
  \Omega \rightarrow \mathbb{R}$ be a random process depending on $\varepsilon
  \in \mathcal{I} \subset [0, 1)$ (where $\mathcal{I}$ has zero as an
  accumulation point) such that
  \begin{enumerate}
    \item For any $p \geqslant 1$ and $s$ we have
    \[ (\sup_{\varepsilon \in \mathcal{I}} \mathbb{E} [| C_{s, \varepsilon}
       |^p])^{1 / p} \leqslant f_p (s), \]
    for some integrable function $f_p : \mathbb{R}_+ \rightarrow
    \mathbb{R}_+$;
    
    \item $C_{s, \varepsilon}$ is continuous as $\varepsilon \rightarrow 0$
    almost surely, i.e. $C_{s, \varepsilon} \rightarrow C_{s, 0}$ almost
    surely.
  \end{enumerate}
  Then there exists a sequence $\varepsilon_n \subset \mathcal{I}$ such that,
  for any $\sigma$-algebra $\mathcal{G}$, the random variable $\sup_{n \in
  \mathbb{N}} \mathbb{E} \left[ \int_0^{\infty} C_{s, \varepsilon_n} \mathd s|
  \mathcal{G} \right]$ is bounded almost surely.
\end{lemma}

\begin{proof}
  First we prove that $\mathbb{E} \left[ \int_0^{\infty} C_{s, \varepsilon}
  \mathd s| \mathcal{G} \right]$ converges to $\mathbb{E} \left[
  \int_0^{\infty} C_{s, 0} \mathd s| \mathcal{G} \right]$ in $L^p$. By
  Minkowski, we have that
  \[ \mathbb{E} \left[ \left| \mathbb{E} \left[ \int_0^{\infty} C_{s, 0}
     \mathd s - \int_0^{\infty} C_{s, \varepsilon_n} \mathd s | \mathcal{G}
     \right] \right|^p \right] \leqslant \int_0^{+ \infty} (\mathbb{E} [|
     C_{s, \varepsilon} - C_{s, 0} |^p]^{1 / p}) \mathd s. \]
  On the other hand, for some $\kappa > 0$, we have $\sup_{\varepsilon \in
  \mathcal{I}} \mathbb{E} [| C_{s, \varepsilon} |^{p + \kappa}] < + \infty$,
  and thus the family $\{ | C_{s, \varepsilon} - C_{s, 0} | \}_{\varepsilon
  \in \mathcal{I}}$ of random variable is uniformly integrable. This means
  that, since $| C_{s, \varepsilon} - C_{s, 0} |$ converges to $0$ almost
  surely, then $| C_{s, \varepsilon} - C_{s, 0} |$ converges to $0$ in $L^p$.
  This implies that for every $s \in \mathbb{R}_+$ $\mathbb{E} [| C_{s,
  \varepsilon} - C_{s, 0} |^p]^{1 / p}$ converges to $0$ and thus, since
  $\mathbb{E} [| C_{s, \varepsilon} - C_{s, 0} |^p]^{1 / p} \leqslant f_p$, by
  Lebesgue dominated convergence theorem that $\int_0^{+ \infty} (\mathbb{E}
  [| C_{s, \varepsilon} - C_{s, 0} |^p]^{1 / p}) \mathd s \rightarrow 0$ as
  $\varepsilon \rightarrow 0$.
  
  Since $\mathbb{E} \left[ \int_0^{\infty} C_{s, \varepsilon} \mathd s|
  \mathcal{G} \right]$ goes to $\mathbb{E} \left[ \int_0^{\infty} C_{s, 0}
  \mathd s| \mathcal{G} \right]$ in $L^p$, there is a subsequence $\{
  \varepsilon_n \}_{n \in \mathbb{N}} \subset \mathcal{I}$ with $\varepsilon_n
  \rightarrow 0$ such that $\mathbb{E} \left[ \int_0^{\infty} C_{s,
  \varepsilon} \mathd s| \mathcal{G} \right]$ goes to $\mathbb{E} \left[
  \int_0^{\infty} C_{s, 0} \mathd s| \mathcal{G} \right]$ almost surely. This
  concludes the proof.
\end{proof}

\begin{remark}
  The expectation $\mathbb{E}^{\omega'}$ introduced in Notation
  \ref{notation:omegaprime}, can be understood as a conditional expectation
  with respect to the $\sigma$-algebra generated by random field $\omega
  \mapsto \xi (\omega)$. In this way, we can exploit Lemma
  \ref{lemma:Lpconvergence} for random fields of the form
  $\mathbb{E}^{\omega'} [\cdot]$.
\end{remark}

A consequence of the previous result is the following lemma.

\begin{lemma}
  \label{lemma:sup}For every $0 < \kappa \ll 1$,we have that there is a
  sequence $\varepsilon_n \in (0, 1)$ such that $\varepsilon_n \rightarrow 0$,
  as $n \rightarrow \infty$ for which
  \[ \sup_{n \in \mathbb{N}} \mathbb{E}^{\omega'} \left[ \int^{\infty}_0
     \left\| \left( {J_s}  \xi_{\varepsilon_n} W_s \circ J_s
     \xi_{\varepsilon_n} - \dot{\gamma}_{\varepsilon_n, s}^{(2)} W_s \right)
     \right\|_{H^{- \delta}} \mathd s \right] \leqslant C_1 (\omega) \]
  \[ \sup_{n \in \mathbb{N}} (\sup_{s \in \mathbb{R}_+} ((1 + s)^{1 + \kappa}
     \| J_s \xi_{\varepsilon_n} (\omega) \circ J_s \xi_{\varepsilon_n}
     (\omega) - \dot{\gamma}^{(2)}_{\varepsilon_n, s} \|_{\mathcal{C}^{-
     \delta}})) \leqslant C_2 (\omega) \]
  \[ \sup_{n \in \mathbb{N}} \| \xi_{\varepsilon_n} (\omega)
     \|_{\mathcal{C}^{- 1 - \delta}}^2 \leqslant C_3 (\omega), \quad \sup_{n
     \in \mathbb{N}} (\mathbb{E}^{\omega'} [\sup_{s \in \mathbb{R}_+} (1 +
     s)^{- 2 \ell} \| \xi_{\varepsilon_n} W_s \|_{\mathcal{C}^{- 1 -
     \delta}}^2]) \leqslant C_4 (\omega) \]
  for some (positive) random variables $C_1, C_2, C_3, C_4$ which are almost
  surely finite.
\end{lemma}

\begin{proof}
  The proof of this lemma is a straightforward application of Lemma
  \ref{lemma:Lpconvergence} combined with the stochastic estimates from Lemma
  \ref{lemma:stochasticestimates1}, Remark \ref{remark:stochasticestimates2},
  Lemma \ref{lemma:stochasticestimates3} and Lemma
  \ref{lemma:stochasticestimates4} and the immersion properties of Besov
  spaces (namely the immersion of $B^{\kappa}_{p, p} (\mathbb{R}_+, B^s_{q,
  q})$ into $C^0 \left( \mathbb{R}_+, \mathcal{C}^{s - \frac{2}{q}} \right)$
  when $\kappa > \frac{1}{p}$).
\end{proof}

\begin{lemma}
  \label{lemma:differentialinequality}For every $\eta, C, P, K > 0$ and
  $\lambda > \frac{1}{2}$ there is a $C^1$ function $f : \mathbb{R}_+
  \rightarrow \mathbb{R}_+$ such that
  \begin{equation}
    \lim_{t \rightarrow + \infty} f (t) = P, \quad \dot{f} (t) \geqslant
    \frac{K (f (t))^{2 \lambda}}{(1 + t)^{1 + \eta}} + \frac{C}{(1 + t)^{1 +
    \eta}} \label{eq:differentialinequality}
  \end{equation}
\end{lemma}

\begin{proof}
  Suppose that $\sup_{t \in \mathbb{R}_+} f (t) \leqslant L$, and $f$ solves
  the equation
  \begin{equation}
    \dot{f} (t) = \frac{K L^{2 \lambda - 1} f (t)}{(1 + t)^{1 + \eta}} +
    \frac{C}{(1 + t)^{1 + \eta}}, \quad \lim_{t \rightarrow + \infty} f (t) =
    P, \label{eq:differentialinequality2}
  \end{equation}
  then $f$ satisfies the differential inequality
  \eqref{eq:differentialinequality}. 
\end{proof}

\begin{proof*}{Proof of Theorem \ref{theorem:Fbounds}}
  If we apply the analytical estimates of Section
  \ref{subsection:analyticalestimates} to the functional $F^{\varepsilon}$ (in
  the form given by Proposition \ref{proposition:reformulationF}) we get
  \begin{eqnarray}
    F^{\varepsilon} (u) - F^{\varepsilon} (0) & \geqslant & \frac{1}{2}
    \int_0^{\infty} \| l (u) \|^2_{L^2} \mathd t - 8 K_1 \kappa
    \int_0^{\infty} \| u \|_{H^{- \delta}}^2 \mathd t + \nonumber\\
    &  & - K_2 (\kappa) \left( \mathbb{E}^{\omega'} \left[ \left(
    \int^{\infty}_0 \left\| \left( {J_t}  \xi_{\varepsilon} W_t \circ J_t
    \xi_{\varepsilon} - \dot{\gamma}_{\varepsilon, t}^{(2)} W_t \right)
    \right\|_{H^{- \delta}} \mathd t \right)^2 \right] + \| \xi_{\varepsilon}
    \|_{\mathcal{C}^{- 1 - \delta}}^2 \right) \nonumber\\
    &  & - K_2 (\kappa) \| \xi \|_{\mathcal{C}^{- 1 - \delta}}^2
    \mathbb{E}^{\omega'} [(\sup_{s \in \mathbb{R}_+} (1 + s)^{\ell} \| \xi W_s
    \|_{\mathcal{C}^{- 1 - \delta}}^2)] \nonumber\\
    &  & - K_2 (\kappa) \mathbb{E}^{\omega'} \left[ \int^{\infty}_0 \frac{|
    \gamma_{\varepsilon, t}^{(2)} |^2}{(1 + t)^{3 / 2 - 2 \delta}} \| W_t
    \|_{\mathcal{C}^{- \delta}}^2 \mathd t + \left( \int_0^{+ \infty} |
    \dot{\gamma}^{(1)}_t (\omega) | \| W_t \|_{\mathcal{C}^{- \delta}} \mathd
    t \right)^2 \right] \nonumber\\
    &  & - K_2 (\kappa) \mathbb{E}^{\omega'} \left[ \int_0^{+ \infty} |
    \gamma_t^{(1)} |^2 \frac{\| W_t \|_{\mathcal{C}^{- \delta}}^2}{(1 + t)^{3
    / 2 - 2 \delta}} \mathd t \right] - K_2 (\kappa) \int^{\infty}_0 \frac{|
    \gamma_{\varepsilon, t}^{(2)} |^{\frac{2}{1 - \delta}}}{(1 + t)^{3 / 2}}
    \| Z_t \|_{L^2}^2 \mathd t \nonumber\\
    &  & - K_2 (\kappa) \int^{\infty}_0 \frac{\left( \underset{s}{\sup} (1 +
    s)^{1 + \tau - \lambda} \| ((J_s \xi_{\varepsilon} \circ J_s
    \xi_{\varepsilon}) - \dot{\gamma}_{\varepsilon, s}^{(2)})
    \|_{\mathcal{C}^{- \delta}} \right)^{\frac{1}{1 - \theta}}}{(1 + t)^{1 +
    \tau - \lambda}} \| Z_t \|_{L^2}^2 \mathd t \nonumber\\
    &  & - K_2 (\kappa) \int_0^{+ \infty} \frac{\| \xi_{\varepsilon}
    \|_{\mathcal{C}^{- 1 - \delta}}^2}{(1 + t)^{1 + \eta}} \| Z_t \|_{L^2}^2
    \mathd t - K_2 (\kappa) \int_0^{+ \infty} \frac{| \gamma_t^{(1)}
    |^{\frac{2}{1 - \delta}}}{(1 + t)^{3 / 2}} \| Z_t \|_{L^2}^2 \mathd t
    \nonumber\\
    &  & + K_3 (\kappa) \int_0^{+ \infty} \dot{\gamma}^{(1)}_t (\omega) \|
    Z_t \|_{L^2}^2 \mathd t \nonumber
  \end{eqnarray}
  Where $K_2, K_3 : \mathbb{R}_+ \backslash \{ 0 \} \rightarrow \mathbb{R}_+$
  are continuous (decreasing) functions of $\kappa$ and $K_1 > 0$ is a
  suitable constant. First we note that
  \begin{eqnarray}
    \int_0^{\infty} \| u \|_{H^{- \delta}}^2 \mathd t & \lesssim &
    \int_0^{\infty} \left\| {J_t}  \xi_{\varepsilon} W_t \right\|_{H^{-
    \delta}}^2 \mathd t + \int_0^{\infty} \| J_t (\xi_{\varepsilon} \succ Z_t
    (u)) \|_{H^{- \delta}}^2 \mathd t + \int_0^{\infty} \| l (u_t) \|_{L^2}
    \mathd t \nonumber\\
    & \lesssim & \left( \sup_{s \in \mathbb{R}_+} (1 + s)^{\frac{\delta}{4}}
    \| \xi_{\varepsilon} W_s \|_{H^{- 1 - \frac{\delta}{2}}}^2 \right)
    \int_0^{\infty} \frac{1}{(1 + t)^{1 + \frac{\delta}{4}}} \mathd t +
    \nonumber\\
    &  & + \int_0^{\infty} \frac{1}{(1 + t)^{1 + \delta}} \|
    \xi_{\varepsilon} \|_{\mathcal{C}^{- 1 - \frac{\delta}{2}}} \| Z \|_{H^{1
    - \delta}} \mathd t + \int_0^{\infty} \| l (u_t) \|_{L^2} \mathd t
    \nonumber\\
    & \lesssim & \kappa' \int_0^{\infty} \| u \|_{H^{- \delta}}^2 \mathd t +
    \left( \sup_{s \in \mathbb{R}_+} (1 + s)^{\frac{\delta}{4}} \|
    \xi_{\varepsilon} W_s \|_{H^{- 1 - \frac{\delta}{2}}}^2 \right) + \|
    \xi_{\varepsilon} \|_{\mathcal{C}^{1 - \frac{\delta}{2}}}^2 +
    \int_0^{\infty} \| l (u_t) \|_{L^2} \mathd t. \nonumber
  \end{eqnarray}
  In other words there is $K_4 > 0$ for which
  \begin{equation}
    \int_0^{\infty} \| u \|_{H^{- \delta}}^2 \mathd t \lesssim K_4
    \int_0^{\infty} \| l (u_t) \|_{L^2} \mathd t + K_4 \left( \left( \sup_{s
    \in \mathbb{R}_+} (1 + s)^{\frac{\delta}{4}} \| \xi_{\varepsilon} W_s
    \|_{H^{- 1 - \frac{\delta}{2}}}^2 \right) + \| \xi_{\varepsilon}
    \|_{\mathcal{C}^{1 - \frac{\delta}{2}}}^2 \right) .
    \label{eq:inequalityul}
  \end{equation}
  This means that if we choose $\kappa < \frac{1}{16 K_1 K_4}$, then
  \[ \frac{1}{2} \int_0^{\infty} \| l (u) \|^2_{L^2} \mathd t - 8 K_1 \kappa
     \int_0^{\infty} \| u \|_{H^{- \delta}}^2 \mathd t \geqslant \frac{1}{4}
     \int_0^{\infty} \| l (u) \|^2_{L^2} \mathd t + \]
  \[ - \frac{1}{4} \left( \left( \sup_{s \in \mathbb{R}_+} (1 +
     s)^{\frac{\delta}{4}} \| \xi_{\varepsilon} W_s \|_{H^{- 1 -
     \frac{\delta}{2}}}^2 \right) + \| \xi_{\varepsilon} \|_{\mathcal{C}^{1 -
     \frac{\delta}{2}}}^2 \right) . \]
  Now we fix $\kappa > 0$ such that $\kappa < \frac{1}{16 K_1 K_4}$. In this
  way $K_2 (\kappa), K_3 (\kappa)$ are (fixed) numbers (hereafter we drop the
  dependence on $\kappa$ of $K_2 (\kappa), K_3 (\kappa)$). We focus on the
  parts depending on $\| Z \|_{L^2}^2$. If we take $\gamma^{(1)}_t$ such that
  \begin{equation}
    \dot{\gamma}^{(1)}_t (\omega) \geqslant (K_2 K_3 + 1) \frac{|
    \gamma_t^{(1)} (\omega) |^{\frac{2}{1 - \delta}}}{(1 + t)^{1 + \eta'}} +
    \frac{\left( {K_3}  H (\omega) + 1 \right)}{(1 + t)^{1 + \eta'}}, \quad
    \gamma^{(1)}_{\infty} (\omega) = \gamma^{(1)} (\omega)
    \label{eq:inequalitygamma}
  \end{equation}
  where and
  \begin{eqnarray}
    H (\omega) & = & \sup_{n \in \mathbb{N}, s \in \mathbb{R}_+} \left( (1 +
    s)^{1 + \tau - \lambda} \| ((J_s \xi_{\varepsilon_n} \circ J_s
    \xi_{\varepsilon_n}) - \dot{\gamma}_{\varepsilon_n, s}^{(2)})
    \|_{\mathcal{C}^{- \delta}}^{\frac{1}{1 - \theta}} + \vphantom{+ \sup_{t
    \in \mathbb{R}_+} ((\log (2 + t))^{- 1} | \gamma_{\varepsilon_n, t}^{(2)}
    |)^{\frac{2}{1 - \delta}} + \| \xi_{\varepsilon_n} \|_{\mathcal{C}^{- 1 -
    \delta}}^2} \right. \nonumber\\
    &  & \left. + \frac{(\log (2 + s))^{\frac{2}{1 - \delta}}}{(1 + s)^{1 / 2
    - \eta'}} ((\log (2 + s))^{- 1} | \gamma_{\varepsilon_n, s}^{(2)}
    |)^{\frac{2}{1 - \delta}} + \| \xi_{\varepsilon_n} \|_{\mathcal{C}^{- 1 -
    \delta}}^2 \right) \label{eq:Homega} 
  \end{eqnarray}
  which by Lemma \ref{lemma:sup} and Remark \ref{remark:gamma2} is almost
  surely finite and $\eta' < \min \left( \eta, \frac{1}{2}, \delta, - \tau +
  \lambda \right)$. The existence of some solution to the differential
  inequality \eqref{eq:inequalitygamma} is given by Lemma
  \ref{lemma:differentialinequality}. Furthermore, again by Lemma
  \ref{lemma:differentialinequality}, we can choose $\gamma_t^{(1)}$ which is
  bounded (for every fixed $\omega$, not uniformly with respect to $\omega$)
  and such that $\dot{\gamma}^{(1)}_t (\omega) \lesssim \frac{1}{(1 + t)^{1 +
  \eta'}}$. Fix now a solution $\gamma^{(1)}_t$ to the differential inequality
  \eqref{eq:inequalitygamma} satisfying the previous two conditions.
  
  For such a function $\gamma_t^{(1)}$, the sum of the terms involving $\| Z_t
  \|_{L^2}^2$ is strictly positive.
  
  If we now took
  \begin{eqnarray}
    C (\omega) & = & K_2 \sup_{n \in \mathbb{N}} \left. \left(
    \mathbb{E}^{\omega'} \left[ \left( \int^{\infty}_0 \left\| \left( {J_t} 
    \xi_{\varepsilon_n} W_t \circ J_t \xi_{\varepsilon_n} -
    \dot{\gamma}_{\varepsilon_n, t}^{(2)} W_t \right) \right\|_{H^{- \delta}}
    \mathd t \right)^2 \right] + \| \xi_{\varepsilon_n} \|_{\mathcal{C}^{- 1 -
    \delta}}^2 + \right. \right. \nonumber\\
    &  & + \| \xi_{\varepsilon_n} \|_{\mathcal{C}^{- 1 - \delta}}^2
    \mathbb{E}^{\omega'} [(\sup_{s \in \mathbb{R}_+} (1 + s)^{\ell} \|
    \xi_{\varepsilon_n} W_s \|_{\mathcal{C}^{- 1 - \delta}}^2)]
    +\mathbb{E}^{\omega'} \left[ \int_0^{+ \infty} | \gamma_t^{(1)} |^2
    \frac{\| W_t \|_{\mathcal{C}^{- \delta}}^2}{(1 + t)^{3 / 2 - 2 \delta}}
    \mathd t \right] \nonumber\\
    &  & \left. \left. +\mathbb{E}^{\omega'} \left[ \int^{\infty}_0 \frac{|
    \gamma_{\varepsilon_n, t}^{(2)} |^2}{(1 + t)^{3 / 2 - 2 \delta}} \| W_t
    \|_{\mathcal{C}^{- \delta}}^2 \mathd t + \left( \int_0^{+ \infty} |
    \dot{\gamma}^{(1)}_t (\omega) | \| W_t \|_{\mathcal{C}^{- \delta}} \mathd
    t \right)^2 \right] \right) \right. + \nonumber\\
    &  & + \sup_{n \in \mathbb{N}} \frac{1}{4} \left( \left( \sup_{s \in
    \mathbb{R}_+} (1 + s)^{\frac{\delta}{4}} \| \xi_{\varepsilon_n} W_s
    \|_{H^{- 1 - \frac{\delta}{2}}}^2 \right) + \| \xi_{\varepsilon_n}
    \|_{\mathcal{C}^{- 1 - \frac{\delta}{2}}}^2 \right) \nonumber
  \end{eqnarray}
  which is almost surely finite by Lemma \ref{lemma:sup} and thus we are
  finished.
\end{proof*}

\begin{remark}
  \label{rem:variationalAH}It is interesting to note that it is possible to
  use the proof of Theorem \ref{theorem:Fbounds}, in particular the expression
  of the constant \eqref{eq:Homega}, to deduce the existence of a limit
  operator $\mathbb{H}^{\omega, 0} = \lim_{\varepsilon_k \rightarrow 0} (-
  \Delta + \xi_{\varepsilon_k} - \gamma^{(2)}_{\varepsilon_k, \infty})$ (as
  the covariance operator of obtained as the limit of a Gaussian measure
  convergent subsequence $\mu^{\varepsilon_k}$ defined in equation
  \eqref{eq:nuepsilon}). Furthermore, using Lemma
  \ref{lemma:differentialinequality}, we get also that $\mathbb{H}^{\omega, 0}
  \geqslant - (\lim_{\varepsilon_k \rightarrow 0} \gamma^{(1)}_{\varepsilon_k,
  \infty} (\omega)) \mathbb{I}_{L^2}$. This means that the proof of Theorem
  \ref{theorem:Fbounds} in the current section can be seen as providing an
  autonomous proof of the main part of Theorem \ref{thm:anderson-K} using the
  variational techniques of {\cite{BG2020}}.
\end{remark}

\subsection{Construction of the coupling}

In this last subsection we will prove the following theorem.

\begin{theorem}
  \label{theorem:coupling}For any $\chi, \delta > 0$ $\chi > \delta$ and for
  almost every $\omega \in \Omega$, there is a probability measure
  $\overline{\nu_0}$ on $\mathcal{C}^{- \chi} \times H^{1 - \chi}$ such that
  \begin{enumerate}
    \item $P_{\mathcal{C}^{- \chi}, \ast} (\bar{\nu}_0) = \tmop{Law} (\mu^{-
    \Delta + 1})$ (where $P_{\mathcal{C}^{- \chi}} : \mathcal{C}^{- \chi}
    \times H^{1 - \chi} \rightarrow \mathcal{C}^{- \chi}$ is the natural
    projection);
    
    \item $\int \| Z \|_{H^{1 - \chi}}^2 \bar{\nu}_0 (\mathd \varphi, \mathd
    Z) < + \infty$ (where $(\varphi, Z) \in \mathcal{C}^{- \chi} \times H^{1 -
    \chi}$)
    
    \item $\tmop{Law}_{\nu_0} (\varphi + Z) \assign (P_{\mathcal{C}^{- \chi}}
    + P_{H^{1 - \chi}})_{\ast} (\bar{\nu}_0) = \mu^{\mathbb{H}^{\omega} + K
    (\omega)}$ (where $\lim_{\varepsilon_n \rightarrow 0} \mu^{\varepsilon_n}
    (\omega) = \mu^{\mathbb{H}^{\omega} + K (\omega)}$ weakly and where
    $\mu^{\varepsilon_n}$ are the Gaussian measure introduced in equation
    \eqref{eq:nuepsilon}).
  \end{enumerate}
\end{theorem}

\begin{remark}
  Theorem \ref{theorem:coupling} can be also reformulated in this way: for
  almost every $\omega \in \Omega$, there is a coupling $\widetilde{\nu_0}$ on
  $\mathcal{C}^{- \chi} (\mathbb{T}^2) \times \mathcal{C}^{- \chi}
  (\mathbb{T}^2)$ between the Gaussian measures $\mu^{- \Delta + 1}$ and
  $\mu^{\mathbb{H}^{\omega} + K (\omega)}$ such that if we write $(X, Y) \sim
  \widetilde{\nu_0}$, and thus $X \sim \mu^{- \Delta + 1}$ and $Y \sim
  \mu^{\mathbb{H}^{\omega} + K (\omega)}$, we have $X - Y \in H^{1 - \chi}$
  $\tilde{\nu}_0$-almost surely (see {\cite{DeVecchiGubinelliTurra}} for a
  formulation of this property using a Wasserstein-type distance between
  measures). In other words, there exists a $H^{1 - \chi} $regular coupling
  between the standard free field and the Anderson free field.
\end{remark}

In order to prove Theorem \ref{theorem:coupling} we employ bounds on the
functional $F^{\varepsilon}$ proved in Theorem \ref{theorem:Fbounds} to get
tightness of the measures $\mu^{\varepsilon}$. To achieve this aim, we need to
extend the functionals $F^{\varepsilon}$ to functionals depending on the laws
of $(\mathbb{W}, u)$ (where $\mathbb{W}$ is the Gaussian process defined in
Notation \ref{notation:W}) so that we may obtain some compactness properties
of the functionals which will allow us to apply the direct method of the
calculus of variations.

\begin{notation}
  \label{notation:mathcalX}Let us consider the space of Radon measures
  \[ \mathcal{X} \subset \mathcal{P} (C^0 (\mathbb{R}_+, C^{- \chi}
     (\mathbb{T}^2)) \times L^2 (\mathbb{R}_+ \times \mathbb{T}^2))
     \backassign \mathcal{P} (\mathfrak{S} \times L^2 (\mathbb{R}_+ \times
     \mathbb{T}^2)), \]
  defined as follows: We say that the measure $\sigma \in \mathcal{X}$ if,
  writing $\mathfrak{S} \times L^2 (\mathbb{R}_+ \times \mathbb{T}^2) \ni
  (\mathbb{W}, u) \sim \sigma$ for the random variable with law $\sigma$, we
  have that $\mathbb{W}$ is a Gaussian process with covariance as defined in
  Notation \ref{notation:W} and $u$ can be written (almost surely) as a
  progressively measurable process of $\mathbb{W}$ and finally $\| u \|_{L^2
  (\mathbb{R}_+ \times \mathbb{T}^2)} \in L^2 (\sigma)$.
\end{notation}

\begin{remark}
  Usually the space $\mathfrak{S}$, on which the process $\mathbb{W}$ takes
  values, is the space on enhanced noise (i.e. containing also the processes
  ${J_s}  (\xi_{\varepsilon} W_s) J_s \xi_{\varepsilon} {- J_s} 
  \xi_{\varepsilon} J_s \xi_{\varepsilon} W_s$ etc. considered in Section
  \ref{section:stochasticestimates}). Here we define $\mathfrak{S}$ to be only
  $C^0 (\mathbb{R}_+, C^{- \chi} (\mathbb{T}^2))$ (i.e. the space where $W_s$
  takes values) because we never consider directly the limit $\varepsilon
  \rightarrow 0$ but we ask merely for estimates of stochastic terms uniformly
  in $\varepsilon$. For this reason, when $\varepsilon > 0$, since the
  enhanced noise is a continuous function of $W_s$, we need only the space
  $\mathfrak{S} = C^0 (\mathbb{R}_+, C^{- \chi} (\mathbb{T}^2))$.
  
  In any case, since all the stochastic terms, considered in Section
  \ref{section:stochasticestimates}, by Lemma \ref{lemma:Lpconvergence} and
  Lemma \ref{lemma:sup}, converge (almost surely with respect to $\omega \in
  \Omega$), as $\varepsilon \rightarrow 0$, to some well defined adapted
  processes, our argument can be extended to the space of enhanced noise.
\end{remark}

Hereafter we write $L_w^2 (\mathbb{R}_+ \times \mathbb{T}^2)$ for the space
$L^2 (\mathbb{R}_+ \times \mathbb{T}^2)$ equipped with the weak topology.

\begin{definition}
  Let $\mathcal{X}$ be defined as in Notation \ref{notation:mathcalX}.
  Consider the space
  \[ \overline{\mathcal{X}} = \left\{ \sigma : \quad \exists \sigma_n \in
     \mathcal{X} : \sigma_n \rightarrow \sigma \text{ weakly on $\mathfrak{S}
     \times L_w^2 (\mathbb{R}_+ \times \mathbb{T}^2)$, and } \sup_n
     \mathbb{E}_{\sigma_n}^{\omega'} [\| u \|^2_{L^2 (\mathbb{R}_+ \times
     \mathbb{T}^2)}] < \infty \right\} \]
  We say that $\sigma_n \rightarrow \sigma$ in \={$\mathcal{X}$} if $\sigma_n
  \rightarrow \sigma$ weakly and $\sup_n \mathbb{E}_{\sigma_n}^{\omega'} [\| u
  \|^2_{L^2 (\mathbb{R}_+ \times \mathbb{T}^2)}] < \infty$.
\end{definition}

If $\sigma \in \overline{\mathcal{X}}$ we define
\[ \bar{F}^{\varepsilon} (\omega, \sigma) =\mathbb{E}_{\sigma}^{\omega'}
   \left[ \sum_{i = 1}^7 \Gamma_i (\omega) + \mathfrak{G} + \frac{1}{2}
   \int^{\infty}_0 \| l^{\varepsilon}_s (u) \|^2_{L^2} \mathd s \right] +
   \frac{1}{2} \mathbb{E}_{\sigma}^{\omega'} \left[ \int^{\infty}_0 \left\|
   {J_s}  \xi_{\varepsilon} (\omega) W_s \right\|^2_{L^2} \mathd s \right] .
\]
We have the following statement that says that $\overline{F}^{\varepsilon}$ has
the same minimum as $F^{\varepsilon} .$

\begin{lemma}
  \label{lemma:inf1}For almost every $\omega \in \Omega$ and $\varepsilon_n >
  0$ (where $\varepsilon_n$ is in the sequence defined in Theorem
  \ref{theorem:Fbounds}) \ we have
  \[ \inf_{u \in \mathbb{H}_a} F^{\varepsilon_n} (\omega, u) = \inf_{\sigma
     \in \mathcal{X}} \bar{F}^{\varepsilon_n} (\omega, \sigma) = \inf_{\sigma
     \in \overline{\mathcal{X}}} \bar{F}^{\varepsilon_n} (\omega, \sigma) \]
\end{lemma}

\begin{proof}
  The first equality is obvious from the definition of $\mathcal{X}$,
  $\bar{F}^{\varepsilon_n}$. The second inequality can be proved in the same
  way of Lemma 8 of {\cite{Barashkovwholespace}}.
\end{proof}

We introduce here a class of functional which is important in what follows.

\begin{definition}
  \label{definition:admissible}We say that the functional $G : \mathcal{X}
  \rightarrow \mathbb{R}$ is admissible, if $G$ is lower semicontinuous and
  \[ \mathbb{E}_{\sigma} \left[ \int_0^{+ \infty} \| u_t \|^2_{L^2} \mathd t
     \right] \lesssim 1 + G (\sigma) . \]
\end{definition}

\begin{lemma}
  For every fixed $\varepsilon_n > 0$, the functional
  $\bar{F}^{\varepsilon_n}$ is an admissible functional.
\end{lemma}

\begin{proof}
  For brevity we drop the $n$ index. We consider the form of $F^{\varepsilon}$
  given in Theorem \ref{theorem:functionalF}. In particular we have
  \[ \left| \int (\xi_{\varepsilon} (x) + \gamma^{(1)} +
     \gamma^{(2)}_{\varepsilon}) Z_{\infty} (x) W_{\infty} (x) \mathd x
     \right| \lesssim (\| \xi_{\varepsilon} \|_{\mathcal{C}^{\delta}} +
     \gamma^{(1)} + \gamma^{(2)}_{\varepsilon}) \| Z_{\infty} \|_{H^1} \|
     W_{\infty} \|_{\mathcal{C}^{- \delta}} . \]
  Thus we get
  \[ F (\mu) \geqslant \frac{1}{2} \mathbb{E}_{\mu} \left[ \int_0^{\infty} \|
     u_t \|^2_{L^2} \mathd t \right] - K (\| \xi_{\varepsilon}
     \|_{\mathcal{C}^{\delta}} + \gamma^{(1)} + \gamma^{(2)}_{\varepsilon})
     (\mathbb{E}_{\mu} [\| Z_{\infty} \|_{H^1} \| W_{\infty}
     \|_{\mathcal{C}^{- \delta}} + \| W_{\infty} \|_{\mathcal{C}^{-
     \delta}}^2]) . \]
  Thus, from Young's inequality, we get
  \[ \left( \frac{1}{2} - \kappa \right) \mathbb{E}_{\mu} \left[
     \int_0^{\infty} \| u_t \|^2_{L^2} \mathd t \right] \lesssim
     F^{\varepsilon} (\mu) + K^2 (\| \xi_{\varepsilon}
     \|_{\mathcal{C}^{\delta}} + \gamma^{(1)} + \gamma^{(2)}_{\varepsilon} +
     1)^2 \mathbb{E}_{\mu} [\| W_{\infty} \|_{\mathcal{C}^{- \delta}}^2] . \]
  Since for any $\varepsilon > 0$ and $\delta$, $\xi_{\varepsilon} \in
  \mathcal{C}^{\delta} (\mathbb{T}^2)$ this finishes the proof. The lower
  semicontinuity of $F^{\varepsilon}$ can be proved as in Lemma 17 of
  {\cite{Barashkovwholespace}}.
\end{proof}

\begin{lemma}
  \label{existence-min}For every $\varepsilon > 0$ there is
  $\sigma^{\varepsilon} \in \overline{\mathcal{X}}$ such that
  $\bar{F}^{\varepsilon} = \inf_{\sigma \in \overline{\mathcal{X}}}
  \bar{F}^{\varepsilon} (\omega, \sigma)$. Furthermore, for each
  $\sigma^{\varepsilon}$ as before we have
  \[ \mu^{\mathbb{H}_{\varepsilon}^{\omega, K}} = \mu^{\varepsilon} =
     \tmop{Law}_{\sigma^{\varepsilon}} (W_{\infty} + Z_{\infty}) . \]
\end{lemma}

\begin{proof}
  Since $\bar{F}^{\varepsilon}$ is admissible in the sense of Definition
  \ref{definition:admissible} (see also Definition 6 of
  {\cite{Barashkovwholespace}}) then the existence of a minimizer is
  guaranteed by Lemma 7 of {\cite{Barashkovwholespace}}. The fact that
  $\mu^{\varepsilon} = \tmop{Law}_{\sigma^{\varepsilon}} (W_{\infty} +
  Z_{\infty})$ is proved in Theorem 11 of {\cite{Barashkovwholespace}}.
  Finally the fact that $\mu^{\varepsilon}$ is a Gaussian free field related
  to the (regularized) Anderson Hamiltonian is proved in Lemma
  \ref{lemma:absolutecontinuity}.
\end{proof}

\begin{proof*}{Proof of Theorem \ref{theorem:coupling}}
  Consider $\varepsilon_n \in \mathbb{R}_+$, $\varepsilon_n \rightarrow 0$,
  and $C : \Omega \rightarrow \mathbb{R}_+$ as in Theorem
  \ref{theorem:Fbounds}, then we have
  \begin{equation}
    \sup_{n \in \mathbb{N}} \mathbb{E}_{\sigma^{\varepsilon_n}} \left[
    \int^{\infty}_0 \| l^{\varepsilon_n}_s (u) \|^2_{L^2} \mathd s \right]
    \leqslant 8 C (\omega) \label{eq:uniformboundl}
  \end{equation}
  for any $\mu^{\varepsilon_n}$ minimizer of $\bar{F}^{\varepsilon_n}$.
  Indeed, we know that , by Lemma \ref{lemma:drift} there exists a sequence of
  drifts $\bar{u}^{\varepsilon_n}$ such that
  \[ \sup_{n \in \mathbb{N}} F^{\varepsilon_n} (\omega,
     \bar{u}^{\varepsilon_n}) \backassign C (w) < \infty . \]
  Then, with $\tilde{C}$ being a constant changing from line to line, we have
  \begin{eqnarray*}
    0 & \geqslant & (\inf_{u \in \mathbb{H}_a} (F^{\varepsilon_n} (\omega,
    u))) - F^{\varepsilon_n} (\omega, \bar{u}^{\varepsilon_n})\\
    & \geqslant & \frac{1}{4} \mathbb{E}_{\sigma^{\varepsilon_n}} \left[
    \int^{\infty}_0 \| l^{\varepsilon_n}_s (u) \|^2_{L^2} \mathd s \right] -
    \sup_{n \in \mathbb{N}} F^{\varepsilon_n} (\omega,
    \bar{u}^{\varepsilon_n}) - \tilde{C} (\omega)\\
    & \geqslant & \frac{1}{4} \mathbb{E}_{\sigma^{\varepsilon_n}} \left[
    \int^{\infty}_0 \| l^{\varepsilon_n}_s (u) \|^2_{L^2} \mathd s \right] -
    \tilde{C} (\omega)
  \end{eqnarray*}
  Thus, taking the sup over $n \in \mathbb{N}$, we get inequality
  \eqref{eq:uniformboundl}. Consider
  \[ \bar{\nu}^{\varepsilon_n} \assign \tmop{Law}_{\sigma^{\varepsilon_n}}
     (W_{\infty}, Z_{\infty}) \in \mathcal{P} (\mathcal{C}^{- \delta} \times
     H^{1 - \delta}) . \]
  We want to prove that $\bar{\nu}^{\varepsilon_n}$ is a family of tight
  measures in $\mathcal{P} (\mathcal{C}^{- \delta} \times H^{1 - \delta})$.
  Since $\mathcal{C}^{- \delta'} \times H^{1 - \delta'}$ compactly embeds in
  $\mathcal{C}^{- \delta} \times H^{1 - \delta}$ (whenever $\delta >
  \delta'$), it is enough to prove that
  \[ \sup_{n \in \mathbb{N}} \mathbb{E}_{\bar{\nu}^{\varepsilon_n}} [\|
     W_{\infty} \|^2_{\mathcal{C}^{- \delta'}} + \| Z_{\infty} \|^2_{H^{1 -
     \delta'}}] < + \infty . \]
  Since the law of $W_{\infty}$ is the Gaussian free field, obviously $\sup_{n
  \in \mathbb{N}} \mathbb{E}_{\bar{\nu}^{\varepsilon_n}} [\| W_{\infty}
  \|^2_{\mathcal{C}^{- \delta'}}] < + \infty$. On the other hand, by Lemma
  \ref{lemma:inequalityZu}, inequality \eqref{eq:inequalityul} in the proof of
  Theorem \ref{theorem:Fbounds}, and inequality \eqref{eq:uniformboundl},
  obtain
  \begin{align*}
    \mathbb{E}_{\bar{\nu}^{\varepsilon_n}} [\| Z_{\infty} \|^2_{H^{1 -
    \delta'}}] & \leqslant  \mathbb{E}_{\sigma^{\varepsilon_n}} \left[
    \int_0^{+ \infty} \| u_s \|_{H^{1 - \delta'}}^2 \mathd s \right]
    \nonumber\\
    & \lesssim  \mathbb{E}_{\sigma^{\varepsilon_n}} \left[ \int_0^{\infty}
    \| l (u_s) \|_{L^2}^2 \mathd s \right]
    +\mathbb{E}_{\sigma^{\varepsilon_n}} \left[ \left( \sup_{s \in
    \mathbb{R}_+} (1 + s)^{\frac{\delta'}{4}} \| \xi_{\varepsilon} W_s
    \|_{H^{- 1 - \frac{\delta'}{2}}}^2 \right) \right] + \| \xi_{\varepsilon}
    \|_{\mathcal{C}^{- 1 - \frac{\delta'}{2}}}^2 \nonumber\\
    &\lesssim  8 \tilde{C} (\omega) + \tilde{C} (\omega) . \nonumber
  \end{align*}
  Since the previous bound is uniform in $\varepsilon_n$ the tightness of
  $\bar{\nu}^{\varepsilon_n}$ follows. Considering any weak limit $\bar{\nu}_0
  \in \mathcal{P} (\mathcal{C}^{- \delta} \times H^{1 - \delta})$ of a
  suitable subsequence of $\bar{\nu}^{\varepsilon_n}$, we have that
  $\bar{\nu}_0$ satisfies the point 1. and 2. of Theorem
  \ref{theorem:coupling}.
  
  The point 3. follows from the second part of Lemma \ref{existence-min}
  (namely that $\mu^{\mathbb{H}_{\varepsilon_n}^{\omega, K}} =
  \tmop{Law}_{\bar{\nu}^{{\varepsilon_n} }} (W_{\infty} + Z_{\infty})$) and
  the fact that $- \Delta + m^2 + \xi_{\varepsilon_n} +
  \gamma_{\varepsilon_n}^{(2)}$ converges to $\mathbb{H}^{\omega}$ in the norm
  resolvent sense (see Theorem 2.30 of {\cite{GUZ}}).
\end{proof*}

\begin{remark}
  \label{remark:Lph}Thanks to the existence of the coupling proved in Theorem
  \ref{theorem:coupling}, we can deduce some regularity properties of the
  Gaussian Anderson free field $\varphi^A$. Indeed, consider $\varphi^A =
  \varphi^G + h $, where $\varphi^G$ is the (standard) Gaussian free field and
  $h \in H^{1 - \delta} (\mathbb{T}^2)$ almost surely is a regular coupling
  between $\varphi^A$ and $\varphi^G$. Then, since $\varphi^G$ is supported on
  $\mathcal{C}^{- \delta'} (\mathbb{T}^2)$ and by Besov embedding (see Lemma
  \ref{lem:besovem}), $\mathcal{C}^{- \delta'} \supset H^{1 - \delta}$ for
  $\delta' > \delta > 0$, we have $\varphi^A \in \mathcal{C}^{- \delta'}$
  almost surely. This implies by Fernique's theorem for Gaussian measures that
  $\varphi^A \in L^p (\Omega', \mathcal{C}^{- \delta'} (\mathbb{T}^2))$ and
  thus $h \in L^p (\Omega', \mathcal{C}^{- \delta'} (\mathbb{T}^2))$ for any
  $1 \leqslant p < + \infty$. 
\end{remark}

\subsection{On the renormalization of the powers of
AGFF}\label{sec:AndersonWick}

In this section we talk about the renormalization of powers of the AGFF. First
we suppose that $\varphi^A$ is a Gaussian random distribution with covariance
$(\mathbb{H}^{\omega, K})^{- 1}$. Then by Theorem \ref{theorem:coupling} there
is a Gaussian free field (with mass $K$) $\varphi^G$ and a random field $h$
taking values in $H^{1 - \delta} (\mathbb{T}^2)$ (for any $\delta > 0$), with
$\mathbb{E} [\| h \|^2_{H^{1 - \delta}}] < + \infty$, such that $\varphi^A =
\varphi^G + h$. Let $\rho_{\varepsilon}$ be a mollifier and define
$\varphi^A_{\varepsilon} \assign \rho_{\varepsilon} \ast \varphi^A$,
$\varphi^G \assign \rho_{\varepsilon} \ast \varphi^G$ etc. For $M \in
\mathbb{N}$, let $H_M : \mathbb{R} \rightarrow \mathbb{R}$ be the $M$-th
Hermite polynomial and we define
\begin{equation}
  (\varphi^A_{\varepsilon})^{\circ M} = c_{\varepsilon}^{\frac{M}{2}} H_M
  \left( \frac{\varphi^A_{\varepsilon}}{\sqrt{c_{\varepsilon}}} \right)
\end{equation}
where $c_{\varepsilon} = \left( \sum_{k \in \mathbb{Z}^2} \frac{|
\hat{\rho}_{\varepsilon} (x) |^2}{(| k |^2 + m^2)} \right)^{1 / 2} \sim \log
\left( \frac{1}{\varepsilon} \right)$. By the properties of sums of Hermite
polynomials and the fact that $c_{\varepsilon} = (\mathbb{E} [|
\varphi^G_{\varepsilon} |^2])^{1 / 2},$ we get
\begin{equation}
  (\varphi^A_{\varepsilon})^{\circ M} = \sum_{k = 0}^M \left( \begin{array}{c}
    M\\
    k
  \end{array} \right) : (\varphi^G_{\varepsilon})^k : h_{\varepsilon}^{M - k}
  .
\end{equation}
\begin{remark}
  \label{lemma:convolutionoperator}Consider the operator $A_{\varepsilon} (f)
  \assign \rho_{\varepsilon} \ast f - f$, then there is $c > 0$ such that
  \[ \| A_{\varepsilon} \|_{\mathcal{L} (B^{s + \kappa}_{p, q}, B^s_{p, q})}
     \leqslant c \varepsilon^{\kappa}, \]
  where the constants in the symbol $\lesssim$ are independent of
  $\varepsilon$.
\end{remark}

\begin{lemma}
  \label{lemma:Wickfreefield}We have that for every $\delta > 0$, $k \in
  \mathbb{N}$ and $p \geqslant 1$ there is $c > 0$ such that
  \[ \mathbb{E} [\| : (\varphi^G_{\varepsilon})^k : - : (\varphi^G)^k :
     \|_{\mathcal{C}^{- \delta}}^p] \lesssim \varepsilon^c . \]
\end{lemma}

\begin{proof}
  See, e.g., Theorem V.3 in {\cite{Simonbookphi}} (see also Lemma 3.12 of
  {\cite{BarshkovDeVecchi}}).
\end{proof}

\begin{lemma}
  \label{lemma:pseudoWickpower}For every $\delta > 0$ and $p \geqslant 1$, we
  have that, for any $M \in \mathbb{N}$, $(\varphi^A_{\varepsilon})^{\circ M}$
  is a Cauchy sequence in $L^p (\Omega', B^{- \delta}_{p, p} (\mathbb{T}^2))$
  with a limit $(\varphi^A)^{\circ M}$. Furthermore we have
  \begin{equation}
    (\varphi^A)^{\circ M} = \sum_{k = 0}^M \left( \begin{array}{c}
      M\\
      k
    \end{array} \right) : (\varphi^G)^k : h^{M - k} .
    \label{eq:sumWickproduct}
  \end{equation}
\end{lemma}

\

\begin{remark}
  \label{rem:diff-Wick}We should stress that the singular product
  $(\varphi^A)^{\circ M}$ (defined thanks to Lemma
  \ref{lemma:pseudoWickpower}) is different from the Gaussian Wick product
  \begin{equation}
    : (\varphi^A)^M : = \lim_{\varepsilon \rightarrow 0} :
    (\varphi^A_{\varepsilon})^M : \label{eq:WickproductA}
  \end{equation}
  defined as in equation \eqref{eq:Wickproduct} and Theorem
  \ref{theorem:Wickproduct}. Indeed, the renormalization procedure is done
  through a limit of a function of the random field $\varphi^A_{\varepsilon}
  (x)$ and the variable $x$, and not only on $\varphi^A_{\varepsilon} (x)$ as
  for the product $(\varphi^A)^{\circ M}$ (see Section 6 of
  {\cite{bailleul2022analysis}} for a discussion on the product
  \eqref{eq:WickproductA}).
  
  However, it turns out that the difference of the renormalization functions
  should be (at least) an $L^p$ function for any $p \geqslant 1$. Take for
  example the square case, where we get
  \[ \mathbb{E} [(\varphi^A)^2 (\cdot) - (\varphi^G)^2 (\cdot)] = 2\mathbb{E}
     [(\varphi^G h) (\cdot)] +\mathbb{E} [h^2 (\cdot)] . \]
  Now the term $\mathbb{E} [h (x)^2]$ is in $L^p$, by Sobolev embedding, since
  $h \in H^{1 - \delta}$. If we decompose the other term as
  \[ \mathbb{E} [\varphi^G h] =\mathbb{E} [\varphi^G \succ h] +\mathbb{E}
     [\varphi^G \preccurlyeq h], \]
  we obtain that it is in $H^{1 - 2 \delta}$. In fact $\mathbb{E} [\varphi^G
  \preccurlyeq h]$ is in $H^{1 - 2 \delta}$ by the properties of $h$ and of
  the paraproduct $\preccurlyeq$,  see Appendix \ref{app:Besov}. To study
  $\mathbb{E} [\varphi^G \succ h]$, we observe
  \begin{eqnarray*}
    \mathbb{E} [\varphi^G \succ h] & = & \sum_{i \leqslant j - 1} \mathbb{E}
    [\Delta_j \varphi^G \Delta_i h]\\
    & = & \sum_{i \leqslant j - 1} \mathbb{E} [\Delta_j \varphi^G] \mathbb{E}
    [\Delta_i h]\\
    & = & 0
  \end{eqnarray*}
  where we have used the ``scale to scale'' property of Remark
  \ref{rem:scaletoscale}.
  
  This remark is an example of the renormalization of singular products
  through diverging constant (in space) functions of the singular field whose
  law is not invariant with respect to translation (see
  {\cite{bailleul2023global,hairer2023regularity}} for some examples of this
  kind of phenomenon in the case of smooth Riemannian manifolds).
\end{remark}

\begin{proof*}{Proof of Lemma \ref{lemma:pseudoWickpower}}
  We fix $p \geqslant 1$ and $\delta > 0$, we want to prove that
  \[ \lim_{\varepsilon \rightarrow 0} \mathbb{E} [\|
     (\varphi^A_{\varepsilon})^{\circ M} - (\varphi^A)^{\circ M} \|_{B^{-
     \delta}_{p, p}}^p] = 0. \]
  We have that
  \begin{eqnarray}
    \mathbb{E} [\| (\varphi^A_{\varepsilon})^{\circ M} - (\varphi^A)^{\circ M}
    \|_{B^{- \delta}_{p, p}}^p] & \lesssim & \sum_{k = 0}^M \mathbb{E} [\| :
    (\varphi^G_{\varepsilon})^k : h^{M - k}_{\varepsilon} - : (\varphi^G)^k :
    h^{M - k} \|_{B^{- \delta}_{p, p}}^p] \label{eq:sumWikproduct2} . 
  \end{eqnarray}
  Let us focus on each separate term in the previous sum. We have that
  \begin{eqnarray}
    &  & \mathbb{E} [\| : (\varphi^G_{\varepsilon})^k : h^{M -
    k}_{\varepsilon} - : (\varphi^G)^k : h^{M - k} \|_{B^{- \delta}_{p, p}}^p]
    \nonumber\\
    & \lesssim & \mathbb{E} [\| : (\varphi^G_{\varepsilon})^k : - :
    (\varphi^G)^k : \|_{\mathcal{C}^{- \delta}}^p \| h_{\varepsilon} \|_{B^{2
    \delta}_{p (M - k), p (M - k)}}^{p (M - k)}] + \nonumber\\
    &  & + \sum_{\ell = 1}^{M - k - 1} \mathbb{E} [\| : (\varphi^G)^k :
    \|_{\mathcal{C}^{- \delta}}^p \| h_{\varepsilon} - h \|_{B^{2 \delta}_{p
    (M - k), p (M - k)}}^p \| h_{\varepsilon} \|^{p (M - k - \ell)}_{B^{2
    \delta}_{p (M - k), p (M - k)}} \| h \|^{p \ell}_{B^{2 \delta}_{p (M - k),
    p (M - k)}}] . \nonumber
  \end{eqnarray}
  Fix $n \in \mathbb{N}$ and $k$, then there is $\delta', \delta'' > 0$, $p'
  \geqslant 1$ and $0 \leqslant \theta \leqslant 1$ such that
  \[ - \theta \delta' + (1 - \theta) (1 - \delta'') \geqslant 3 \delta, \quad
     \frac{(1 - \theta)}{2} \geqslant \frac{1}{p (M - k)}, \quad
     \frac{p'}{\theta} > p (M - k) . \]
  Since $h \in L^2 (\Omega', H^{1 - \delta''} (\mathbb{T}^2))$ and $h \in
  L^{p'} (\Omega', \mathcal{C}^{- \delta'})$, (see Remark \ref{remark:Lph}) we
  get that
  \[ h \in L^{\frac{p'}{\theta}} (\Omega', B^{3 \delta}_{p (M - k), p (M -
     k)}) . \]
  Putting all the previous observations together, we obtain
  \begin{eqnarray}
    &  & \mathbb{E} [\| : (\varphi^G_{\varepsilon})^k : h^{M -
    k}_{\varepsilon} - : (\varphi^G)^k : h^{M - k} \|_{B^{- \delta}_{2, 2}}^2]
    \nonumber\\
    & \lesssim & (\mathbb{E} [\| : (\varphi^G_{\varepsilon})^k : - :
    (\varphi^G)^k : \|_{\mathcal{C}^{- \delta}}^{2 q'}])^{\frac{1}{q'}}
    (\mathbb{E} [\| h_{\varepsilon} \|_{B^{2 \delta}_{p (M - k), p (M -
    k)}}^{p'}])^{\frac{p (M - k)}{p'}} \nonumber\\
    &  & + \sum_{\ell = 1}^{2 n - k - 1} \| A_{\varepsilon} \|_{\mathcal{L}
    (B^{3 \delta}_{p, q}, B^{2 \delta}_{p, q})} \mathbb{E} [\| : (\varphi^G)^k
    : \|_{\mathcal{C}^{- \delta}}^{2 q'}]^{\frac{1}{q'}} \mathbb{E} [\|
    h_{\varepsilon} \|_{B^{3 \delta}_{p (M - k), p (M - k)}}^{p'}]^{\frac{p (M
    - k - \ell)}{p'}} \times \nonumber\\
    &  & \times \mathbb{E} [\| h \|^{p'}_{B^{3 \delta}_{p (M - k), p (M -
    k)}}]^{\frac{p (\ell + 1)}{p'}} \nonumber
  \end{eqnarray}
  where, as usual, $q' \geqslant 1$ such that $\frac{1}{q'} + \frac{1}{p'} =
  1$, and $A_{\varepsilon} (f) = \rho_{\varepsilon} \ast f - f$. Since by
  Lemma \ref{lemma:Wickfreefield} $\mathbb{E} [\| :
  (\varphi^G_{\varepsilon})^k : - : (\varphi^G)^k : \|_{\mathcal{C}^{-
  \delta}}^{2 q'}] \rightarrow 0$, as $\varepsilon \rightarrow 0$ as
  $\mathbb{E} [\| h_{\varepsilon} \|_{B^{2 \delta}_{p (M - k), 2 (M -
  k)}}^{p'}]$ is uniformly bounded when $0 \leqslant \varepsilon \leqslant 1$
  and since by Remark \ref{lemma:convolutionoperator} $\lim_{\varepsilon
  \rightarrow 0} \| A_{\varepsilon} \|_{\mathcal{L} (B^{3 \delta}_{p, q}, B^{2
  \delta}_{p, q})} = 0$, we get $\mathbb{E} [\| : (\varphi^G_{\varepsilon})^k
  : h^{2 n - k}_{\varepsilon} - : (\varphi^G)^k : h^{2 n - k} \|_{B^{-
  \delta}_{p, p}}^p] \rightarrow 0$ as $\varepsilon \rightarrow 0$. By
  inequality \eqref{eq:sumWikproduct2} and since the previous proof can be
  repeated (with different $\delta', \delta'' > 0$, $p' \geqslant 1$ and $0
  \leqslant \theta \leqslant 1$) for any $n \geqslant k \in \mathbb{N}$, this
  implies the thesis.
\end{proof*}

\begin{remark}
  \label{remark:pseudoWickpower}A consequence of the proof of Lemma
  \ref{lemma:pseudoWickpower} and of the estimates in Lemma
  \ref{lemma:convolutionoperator} and Lemma \ref{lemma:Wickfreefield} is that
  there exists a $c > 0$ for which
  \[ \mathbb{E} [\| (\varphi^A_{\varepsilon})^{\circ M} - (\varphi^A)^{\circ
     M} \|_{B^{- \delta}_{p, p}}^p] \lesssim \varepsilon^{c \wedge \delta} .
  \]
\end{remark}

Let $P (x) = \sum_{k = 0}^M c_k x^k$ be a polynomial then we write
\[ P^{\circ} (\varphi^A) = \sum_{k = 0}^M c_k (\varphi^A)^{\circ k} . \]

\begin{theorem}
  \label{theorem:Lp}Let $P$ be a polynomial of even degree with positive
  leading coefficient, then for any $p \geqslant 0$ we have
  \[ \mathbb{E} \left[ \exp \left( - p \int_{\mathbb{T}^2} P^{\circ}
     (\varphi^A) (x) \mathd x \right) \right] < + \infty . \]
\end{theorem}

\begin{proof}
  The proof is similar to the proof of exponential integrability of the even
  Wick products of Gaussian free field in 2d (see, e.g., Chapter V of
  {\cite{Simonbookphi}}). For this reason, here we only provide a sketch of
  the proof in the case where $P (x) = x^{2 n}$ for some $n \in \mathbb{N}$.
  
  Fix $n \in \mathbb{N}$, then there is $K \geqslant 0$ such that $H_{2 n}
  (x) \geqslant - K$. This means that
  \[ (\varphi_{\varepsilon}^A)^{\circ 2 n} \geqslant - K c_{\varepsilon}^{2
     n} . \]
  By recalling that $c_{\varepsilon} \gtrsim \sqrt{\log
  \frac{1}{\varepsilon}}$, and thus the previous inequality becomes
  \begin{equation}
    \int_{\mathbb{T}^2} (\varphi_{\varepsilon}^A)^{\circ 2 n} \mathd x
    \geqslant - K' \left( \log \left( \frac{1}{\varepsilon} \right) \right)^n
    \geqslant 1 - 2 K' \left( \log \left( \frac{1}{\varepsilon} \right)
    \right)^n \label{eq:inequalityexponentialbounds1}
  \end{equation}
  for $\varepsilon \ll 1$ and for some $K' \geqslant 0$. By hypercontractivity
  and the proof of Lemma \ref{lemma:pseudoWickpower} (see Remark
  \ref{remark:pseudoWickpower}) there is $c, K'' > 0$ such that, for every $p
  \geqslant 2$,
  \begin{eqnarray}
    \mathbb{E} \left[ \left| \int_{\mathbb{T}^2}
    (\varphi_{\varepsilon}^A)^{\circ 2 n} \mathd x - \int_{\mathbb{T}^2}
    (\varphi^A)^{\circ 2 n} \mathd x \right|^p \right] & \lesssim & (p - 1)^{n
    p} (\mathbb{E} [\| (\varphi_{\varepsilon}^A)^{\circ 2 n} -
    (\varphi^A)^{\circ 2 n} \|^2_{H^{- \delta}}])^{\frac{p}{2}} \nonumber\\
    & \leqslant & (K'')^p (p - 1)^{n p} \varepsilon^{\frac{c p}{2}} . 
    \label{eq:inequalityexponentialbounds2}
  \end{eqnarray}
  Now, if $\int_{\mathbb{T}^2} (\varphi^A)^{\circ 2 n} \mathd x \leqslant - 2
  K' \left( \log \left( \frac{1}{\varepsilon} \right) \right)^n$ then, by
  inequality \eqref{eq:inequalityexponentialbounds1}, $\left|
  \int_{\mathbb{T}^2} (\varphi^A)^{\circ 2 n} \mathd x - \int_{\mathbb{T}^2}
  (\varphi^A_{\varepsilon})^{\circ 2 n} \mathd x \right| \geqslant 1$, and
  thus, by Markov inequality and inequality
  \eqref{eq:inequalityexponentialbounds2} we get
  \begin{eqnarray}
    &  & \mathbb{P} \left( \int_{\mathbb{T}^2} (\varphi^A)^{\circ 2 n} \mathd
    x \leqslant - 2 K' \left( \log \left( \frac{1}{\varepsilon} \right)
    \right)^n \right) \nonumber\\
    & \leqslant & \mathbb{P} \left( \left| \int_{\mathbb{T}^2}
    (\varphi^A)^{\circ 2 n} \mathd x - \int_{\mathbb{T}^2}
    (\varphi^A_{\varepsilon})^{\circ 2 n} \mathd x \right| \geqslant 1 \right)
    \nonumber\\
    & \leqslant & \mathbb{E} \left[ \left| \int_{\mathbb{T}^2}
    (\varphi^A)^{\circ 2 n} \mathd x - \int_{\mathbb{T}^2}
    (\varphi^A_{\varepsilon})^{\circ 2 n} \mathd x \right|^p \right] \lesssim
    (K'')^p (p - 1)^n \varepsilon^{\frac{c p}{2}} . \nonumber
  \end{eqnarray}
  If we choose $p = \left( \frac{1}{\varepsilon} \right)^{\frac{c}{6 n}}$ we
  get, for $\varepsilon \ll 1$, that
  \begin{equation}
    \mathbb{P} \left( \int_{\mathbb{T}^2} (\varphi^A)^{\circ 2 n} \mathd x
    \leqslant - 2 K' \left( \log \left( \frac{1}{\varepsilon} \right)
    \right)^n \right) \leqslant e^{- \left( \frac{1}{\varepsilon}
    \right)^{\frac{c}{6 n}}} \label{eq:inequalityexponentialbounds3}
  \end{equation}
  Since inequality \eqref{eq:inequalityexponentialbounds3} holds for any
  $\varepsilon > 0$ small enough, by standard methods (see Theorem V.7
  {\cite{Simonbookphi}}), inequality \eqref{eq:inequalityexponentialbounds3}
  implies the thesis. \ 
\end{proof}

Thanks to the previous theorem we can define Anderson $\Phi^4_2$ Gibbs measure
(so the $\Phi^4_2$ measure with respect to the AGFF). Indeed we introduce the
following definition.

\begin{definition}[Anderson $\Phi_2^4$]
  Consider $\omega \in \Omega$ and $\mu^{\mathbb{H}^{\omega, K}}$ (defined in
  Section \ref{sec:Wick}), we defined the Anderson $\Phi_2^4$ measure as the
  measure $\nu^{\omega}$ defined as
  \[ \nu^{\omega} = (\mathcal{Z} (\omega))^{- 1} \exp \left( -
     \int_{\mathbb{T}^2} (\varphi^A)^{\circ 4} \mathd x + K (\omega)
     \int_{\mathbb{T}^2} (\varphi^A)^{\circ 2} \right) \mathd
     \mu^{\mathbb{H}^{\omega, K}} \]
  where $\mathcal{Z} (\omega) = \int \exp \left( - \int_{\mathbb{T}^2}
  (\varphi^A)^{\circ 4} \mathd x + K (\omega) \int_{\mathbb{T}^2}
  (\varphi^A)^{\circ 2} \right) \mathd \mu^{\mathbb{H}^{\omega, K}} (\varphi)$
  which is well defined by Theorem \ref{theorem:Lp}.
\end{definition}

\begin{remark}
  \label{rem:abscont}From the definition of the Anderson $\Phi_2^4$ measure
  and Theorem \ref{theorem:Lp} we obtain readily that $\nu^{\omega}$ is
  absolutely continuous with respect to the Gaussian AGFF measure.
\end{remark}

\section{Local-in-time solution}\label{sec:localsol}

In this section we prove local-in-time well-posedness for a renormalised
nonlinear wave equation with white noise potential with initial data supported
in the Gibbs measure
\[ (u_0, u_1) \sim \nu^{\omega} \otimes \mu^{I_{L^2}}, \]
where $\nu^{\omega}$ is the Anderson $\Phi_2^4$ measure
from Section \ref{sec:AndersonWick} and $\mu^{I_{L^2}}$
is the white noise measure, defined in Section \ref{sec:Wick}.

As we have seen in Remark \ref{rem:abscont}, the Anderson $\Phi_2^4$ measure
is mutually absolutely continuous w.r.t. the Anderson Free Field
$\mu^{\mathbb{H}^{\omega, K}}$ and so we can write the initial conditions via
random series as
\begin{eqnarray}
  u_0 (\omega') & = & \underset{n \in \mathbb{N}}{\sum} \frac{\hat{g}_0
  (\omega')}{\lambda_n} f_n \in \mathcal{C}^{- \delta}  \text{a.s.} 
  \label{eqn:randomu0}\\
  u_1 (\omega') & = & \underset{n \in \mathbb{N}}{\sum} \hat{g}_1 (\omega')
  f_n, \in \mathcal{C}^{- 1 - \delta} \text{ a.s.}  \label{eqn:randomu1}
\end{eqnarray}
for any $\delta > 0,$ where $\lambda_n, f_n$ are eigenvalues and -function of
$\mathbb{H}^{\omega}$ as in \eqref{eqn:Heigen} and the $\hat{g}_0, \hat{g}_1$
are i.i.d. standard Gaussians, similar to the definition of $\xi$ in
\eqref{xisum}. Note here that we distinguish the two probability spaces
$\omega' \in \Omega'$ for the random initial data and $\omega \in \Omega$ for
the randomness in the Anderson Hamiltonian. Since the operator
$\mathbb{H}^{\omega, K}$ depends on $\omega \in \Omega$, the eigenvalues and
-functions of course also depend on $\omega$ but we omit that dependence here
for brevity.

We consider the random data SPDE formally given by
\begin{eqnarray}
  (\partial^2_t +\mathbb{H}^{\omega}) u + u^3 & = & 0  \label{eqn:formalNLW}\\
  (u, \partial_t u) |_{t = 0} & = & (u_0 (\omega'), u_1 (\omega')), \nonumber
\end{eqnarray}
where due to the irregularity of the initial data, we actually have to
Wick-ordered the nonlinearity as was done in {\cite{OhLa2020}} in the
classical case without the white noise potential.

In order to illustrate this point, we make the following ansatz for $u$
(named after Da Prato Debussche in the SPDE literature and Bourgain/McKean in
the dispersive PDE literature) and we use the shifted operator
$\mathbb{H}^{\omega, K}$ as in \eqref{def:HK} in order to avoid difficulties
with taking square roots. This means we change the equation
\eqref{eqn:formalNLW} by adding a linear term but we will see that one has to
renormalize the cube by subtracting an infinite linear term in any case; We
set
\[ u \assign \theta + v, \text{ where } \left\{\begin{array}{lll}
     (\partial^2_t +\mathbb{H}^{\omega, K}) \theta & = & 0\\
     (\theta, \partial_t \theta) |_{t = 0} & = & (u_0 (\omega'), u_1
     (\omega'))
   \end{array}\right. \]
and $v$ formally satisfies the equation
\begin{equation}
  \left\{\begin{array}{lll}
    (\partial^2_t +\mathbb{H}^{\omega, K}) v & = & - (\theta + v)^3 = -
    \theta^3 - 3 v \theta^2 - 3 v^2 \theta - v^3\\
    (v, \partial_t v) |_{t = 0} & = & 0
  \end{array}\right. . \label{eqn:vformal}
\end{equation}
Now from the mild formulation (as was derived in {\cite{GUZ}})
\begin{equation}
  \theta (t) = \cos \left( t \sqrt{\mathbb{H}^{\omega, K}} \right) u_0
  (\omega') + \frac{\sin \left( t \sqrt{\mathbb{H}^{\omega, K}}
  \right)}{\sqrt{\mathbb{H}^{\omega, K}}} u_1 (\omega') \label{def:theta}
\end{equation}
and Theorem \ref{thm:AHprop} we see that $\theta$ will have regularity $C
(\mathbb{R}_+ ; H^{- \delta})$, using hypercontractivity
it is even in $C (\mathbb{R}_+ ; \mathcal{C}^{- \delta})$ cf. Section
\ref{sec:Wick} but crucially it has negative regularity. Thus it is not
possible to classically define powers of $\theta$ appearing in
\eqref{eqn:vformal}, however we can replace them by Wick powers as introduced
in Section \ref{sec:AndersonWick} as was done in
{\cite{bourgain1996invariant}}, {\cite{OhLa2020}} etc.

\begin{proposition}[Anderson wave Wick polynomials]
  \label{prop:wavewick}Let $(u_0, u_1) \sim \mu^{\mathbb{H}^{\omega, K}}
  \otimes \mu^{I_{L^2}}$, i.e. as in \eqref{eqn:randomu0} and
  \eqref{eqn:randomu1}, then
  \[ \theta (t) \sim u_0  \sim \mu^{\mathbb{H}^{\omega, K}}  \text{ for all
     times} t \in \mathbb{R}, \]
  where $\theta$ is defined in \eqref{def:theta}. Moreover
  \[ \theta^{\circ k} (t) \sim (u_0)^{\circ k}  \text{ for all } k \in
     \mathbb{N} \text{and } t \in \mathbb{R} \]
  i.e. in law the time dependent Wick ordering is the same as the Wick
  ordering of the initial condition w.r.t. the Anderson GFF from Section
  \ref{sec:AndersonWick}. Also one has the bound
  \[ \| \theta^{\circ k} \|_{L_{\Omega'}^p L_{[0, T]}^p \mathcal{C}^{- k
     \delta}} \leqslant T^{\frac{1}{p}} \| (u_0)^{\circ k} \|_{L_{\Omega'}^p
     \mathcal{C}^{- k \delta}}  \text{for } 2 \leqslant p < \infty, \delta > 0
     \text{and } k \in \mathbb{N} \]
  and the norm on the right hand side is finite by Theorem \ref{theorem:Lp}
  and one also has an exponential tail estimate of the form
  \begin{equation}
    \mathbb{P}^{'} (\| \theta^{\circ k} \|_{L_{[0, T]}^p \mathcal{C}^{- k
    \delta}} > R) \lesssim e^{- C R} \label{eq:exptail}
  \end{equation}
  for some $C > 0$ and all $R > 0.$
  
  \begin{proof}
    This is because the law of $\theta (t)$ is a rotated Gaussian, see
    Proposition 2.3 in {\cite{OhLa2020}}.
    
    In fact for i.i.d. centered standard Gaussian random variables $g^i_0
    (\omega')$ and $g^j_1 (\omega')$ one can write
    \begin{equation}
      u_0 (\omega') = \underset{i \in \mathbb{N}}{\sum} \frac{g_0^i
      (\omega')}{\lambda_i} f_i  \text{ and } u_1 (\omega') = \underset{j \in
      \mathbb{N}}{\sum} \frac{g_1^j (\omega')}{\lambda_j} f_j,
      \label{eqn:gaussinitial}
    \end{equation}
    where $\lambda_n$ and $f_n$ are the eigenvalues and -functions of
    $\mathbb{H}^{\omega, K}$ respectively. Thus one has
    \[ \theta (t) = \cos \left( t \sqrt{\mathbb{H}^{\omega, K}} \right) u_0
       (\omega') + \frac{\sin \left( t \sqrt{\mathbb{H}^{\omega, K}}
       \right)}{\sqrt{\mathbb{H}^{\omega, K}}} u_1 (\omega') =
       \underset{i}{\sum} \frac{1}{\lambda_i} (\cos (t \lambda_i) g_0^i
       (\omega') + \sin (t \lambda_i) g_1^i (\omega')) f_i \]
    and one sees that the term in the brackets is nothing but a rotated
    Gaussian with mean zero and Variance $\cos^2  (t \lambda_i) + \sin^2  (t
    \lambda_i) = 1$. Thus $\theta (t)$ is in law equal to $u_0$ and the $L^p
    $bound follows readily. The exponential tail estimate follows from the
    previous bound and Theorem \ref{theorem:Lp}.
  \end{proof}
\end{proposition}

The corrected equation for $v$ then reads
\begin{equation}
  \left\{\begin{array}{lll}
    (\partial^2_t +\mathbb{H}^{\omega, K}) v & = & - (\theta + v)^{\circ 3} =
    - \theta^{\circ 3} - 3 v \theta^{\circ 2} - 3 v^2 \theta - v^3\\
    (v, \partial_t v) |_{t = 0} & = & 0
  \end{array}\right., \label{eqn:vWick}
\end{equation}
where $\theta^{\circ k} \in L_T^p \mathcal{C}^{- \varepsilon k}$ for any
$\varepsilon > 0$ and $p < \infty$ is the Wick power from Proposition
\ref{prop:wavewick}. This looks quite similar to the renormalized wave
equation from {\cite{OhLa2020}} except that we have replaced the Laplace
operator by the Anderson Hamiltonian and consequently we use a different Wick
ordering.

\

However, due to the norm equivalence $\| u \|_{H^s} \approx \left\|
(\mathbb{H}^{\omega, K})^{\frac{s}{2}} u \right\|_{L^2}$ and the regularizing
property of $\frac{\sin \left( t \sqrt{\mathbb{H}^{\omega, K}}
\right)}{\sqrt{\mathbb{H}^{\omega, K}}}$, see Theorem \ref{thm:AHprop}, we are
able to solve \eqref{eqn:vWick} locally in time without much effort. To do
this we prove a simple lemma which quantifies those two properties of
$\mathbb{H}^{\omega, K}$.

\begin{lemma}[Inhomogeneous estimate]
  \label{lem:inhom}
  
  For $\sigma \in (0, 1)$ and $p \geqslant 1$ we have
  \[ \left\| \int^t_0 \frac{\sin \left( (t - s) \sqrt{\mathbb{H}^{\omega, K}}
     \right)}{\sqrt{\mathbb{H}^{\omega, K}}} f (s) d s \right\|_{{H^{1 -
     \sigma}} } \lesssim t^{\frac{p - 1}{p}} \| f \|_{L_{[0, t]}^p H^{-
     \sigma}} \]
  for all times $t \geqslant 0$, with the obvious modification for $p = \infty
  .$
  
  \begin{proof}
    This is a simple consequence of H{\"o}lder's inequality in time and the
    aforementioned norm equivalence from Theorem \ref{thm:AHprop}. Indeed we
    may bound
    \begin{eqnarray*}
      \left\| \int^t_0 \frac{\sin \left( (t - s) \sqrt{\mathbb{H}^{\omega, K}}
      \right)}{\sqrt{\mathbb{H}^{\omega, K}}} f (s) d s \right\|_{{H^{1 -
      \sigma}} } & \overset{\eqref{eqn:normequ}}{\approx} & \left\|
      \sqrt{\mathbb{H}^{\omega, K}}^{1 - \sigma} \int^t_0 \frac{\sin \left( (t
      - s) \sqrt{\mathbb{H}^{\omega, K}} \right)}{\sqrt{\mathbb{H}^{\omega,
      K}}} f (s) d s \right\|_{{L^2} }\\
      & \lesssim & \int^t_0 \left\| \frac{\sin \left( (t - s)
      \sqrt{\mathbb{H}^{\omega, K}} \right)}{\sqrt{\mathbb{H}^{\omega,
      K}}^{\sigma}} f (s) \right\|_{{L^2} } d s\\
      & \overset{\eqref{eqn:cosL2}}{\lesssim} & \int^t_0 \left\|
      \sqrt{\mathbb{H}^{\omega, K}}^{- \sigma} f (s) \right\|_{{L^2} } d s\\
      & \overset{\eqref{eqn:normequ}}{\approx} & \int^t_0 \| f (s) \|_{{H^{-
      \sigma}} } d s\\
      & \lesssim & t^{\frac{p - 1}{p}} \| f \|_{L_{[0, t]}^p H^{- \sigma}}
    \end{eqnarray*}
    and thus we are done.
  \end{proof}
\end{lemma}

This allows us to prove local well-posedness for the SPDE \eqref{eqn:vWick}
via fixed point.

\begin{theorem}[Local well-posedness]
  \label{thm:localwp}Let $0 < \delta \ll 1$ and $p \gg 1.$With $\theta$
  defined as above, there exists a time
  \[ T \sim \left( \| \theta^{\circ 3} \|^{\frac{1}{3}}_{L_{[0, 1]}^p H^{-
     \delta}} + \| \theta^{\circ 2} \|^{\frac{1}{2}}_{L_{[0, 1] \mathcal{C}^{-
     \delta}}^p} + \| \theta \|_{L_{[0, 1]}^{\infty} \mathcal{C}^{- \delta}} +
     1 \right)^{- 2 \frac{p}{p - 1}} \]
  which is almost surely in $(0, 1) $so that there exists a unique solution
  \[ v \in C ([0, T] ; H^{1 - \delta}) \cap C^1 ([0, T] ; H^{- \delta}) \]
  to the equation
  \begin{equation}
    v (t) = \int^t_0 \frac{\sin \left( (t - s) \sqrt{\mathbb{H}^{\omega, K}}
    \right)}{\sqrt{\mathbb{H}^{\omega, K}}} (\theta^{\circ 3} (s) + 3 v (s)
    \theta^{\circ 2} (s) + 3 v^2 (s) \theta (s) + v^3 (s)) d s,
  \end{equation}
  which is the mild formulation of \eqref{eqn:vWick}.
\end{theorem}

\begin{proof}
  We make a contraction argument for $v$ in a ball in $L_T^{\infty} H^{1 -
  \delta}$ and then show a posteriori that one in fact has continuity in time
  as well.
  
  As usual, we define the map
  \[ \Psi (v) \assign \int^t_0 \frac{\sin \left( (t - s)
     \sqrt{\mathbb{H}^{\omega, K}} \right)}{\sqrt{\mathbb{H}^{\omega, K}}}
     (\theta^{\circ 3} (s) + 3 v (s) \theta^{\circ 2} (s) + 3 v^2 (s) \theta
     (s) + v^3 (s)) d s \]
  and we want to show that it is a contraction on a suitable ball. We make the
  following estimations which are valid for any time $t > 0$ using heavily
  Lemma \ref{lem:inhom}
  \begin{eqnarray*}
    \left\| \int^t_0 \frac{\sin \left( (t - s) \sqrt{\mathbb{H}^{\omega, K}}
    \right)}{\sqrt{\mathbb{H}^{\omega, K}}} \theta^{\circ 3} (s) d s
    \right\|_{H^{1 - \delta}} & \lesssim & t^{\frac{p - 1}{p}} \|
    \theta^{\circ 3} \|_{L_{[0, t]}^p H^{- \delta}}\\
    \left\| \int^t_0 \frac{\sin \left( (t - s) \sqrt{\mathbb{H}^{\omega, K}}
    \right)}{\sqrt{\mathbb{H}^{\omega, K}}} (3 v (s) \theta^{\circ 2} (s)) d s
    \right\|_{H^{1 - \delta}} & \lesssim & t^{\frac{p - 1}{p}} \| v \|_{L_{[0,
    t]}^{\infty} H^{2 \delta}} \| \theta^{\circ 2} (s) \|_{L_{[0, t]}^p
    \mathcal{C}^{- \delta}}\\
    \left\| \int^t_0 \frac{\sin \left( (t - s) \sqrt{\mathbb{H}^{\omega, K}}
    \right)}{\sqrt{\mathbb{H}^{\omega, K}}} (3 v^2 (s) \theta (s)) d s
    \right\|_{H^{1 - \delta}} & \lesssim & t \| \theta \|_{L_{[0, t]}^{\infty}
    \mathcal{C}^{- \delta}} \| v^2 \|_{L^{\infty} H^{2 \delta}}\\
    & \lesssim & t \| \theta \|_{L_{[0, t]}^{\infty} \mathcal{C}^{- \delta}}
    \| v \|^2_{L^{\infty} H^{\frac{1}{2}}}\\
    \left\| \int^t_0 \frac{\sin \left( (t - s) \sqrt{\mathbb{H}^{\omega, K}}
    \right)}{\sqrt{\mathbb{H}^{\omega, K}}} v^3 (s) d s \right\|_{H^{1 -
    \delta}} & \lesssim & t \| v \|^3_{L_{[0, t]}^{\infty} L^6}\\
    & \lesssim & t \| v \|^3_{L_{[0, t]}^{\infty} H^{\frac{2}{3}}} .
  \end{eqnarray*}
  This allows us to bound (taking $t < 1$)
  \begin{eqnarray}
    &  & \| \Psi (v) \|_{L_{[0, t]}^{\infty} H^{1 - \delta}} \nonumber\\
    & \lesssim & t^{\frac{p - 1}{p}} \left( \| \theta^{\circ 3} \|_{L_{[0,
    t]}^p H^{- \delta}} + \| v \|_{L_{[0, t]}^{\infty} H^{2 \delta}} \|
    \theta^{\circ 2} \|_{L_{[0, t]}^p \mathcal{C}^{- \delta}} + \| \theta
    \|_{L_{[0, t]}^{\infty} \mathcal{C}^{- \delta}} \| v \|^2_{L_{[0,
    t]}^{\infty} H^{\frac{1}{2}}} + \| v \|^3_{L_{[0, t]}^{\infty}
    H^{\frac{2}{3}}} \right) \nonumber
  \end{eqnarray}
  and similarly
  \begin{eqnarray}
    &  & \| \Psi (v) - \Psi (w) \|_{L_{[0, t]}^{\infty} H^{1 - \delta}}
    \nonumber\\
    & \lesssim & \left( \| \theta^{\circ 2} \|_{L_{[0, t]}^p \mathcal{C}^{-
    \delta}} + \| \theta \|_{L_{[0, t]}^p \mathcal{C}^{- \delta}} (\| v
    \|_{L_{[0, t]}^{\infty} H^{2 \delta}} + \| w \|_{L_{[0, t]}^{\infty} H^{2
    \delta}}) + \| v \|^2_{L_{[0, t]}^{\infty} H^{\frac{2}{3}}} + \| w
    \|^2_{L_{[0, t]}^{\infty} H^{\frac{2}{3}}} \right) \times \nonumber\\
    &  & t^{\frac{p - 1}{p}} \| v - w \|_{L_{[0, t]}^{\infty} H^{2 \delta}}
    \nonumber
  \end{eqnarray}
  Now we simply take $\| v \|_{L_{[0, T]}^{\infty} H^{1 - \delta}}, \| w
  \|_{L_{[0, T]}^{\infty} H^{1 - \delta}} \leqslant R$ with $R > 0$ and $0 < T
  \leqslant 1$ chosen s.t.
  \begin{eqnarray*}
    T^{\frac{p - 1}{p}} \left( \| \theta^{\circ 3} \|_{L_{[0, T]}^p H^{-
    \delta}} + R \| \theta^{\circ 2} \|_{L_{[0, T] \mathcal{C}^{- \delta}}^p}
    + \| \theta \|_{L_{[0, T]}^{\infty} \mathcal{C}^{- \delta}} R^2 + R^3
    \right) & \leqslant & R\\
    \text{and} &  & \\
    T^{\frac{p - 1}{p}} \left( \| \theta^{\circ 2} \|_{L_{[0, T]
    \mathcal{C}^{- \delta}}^p} + \| \theta \|_{L_{[0, T] \mathcal{C}^{-
    \delta}}^p} 2 R + 2 R^2 \right) & \leqslant & \frac{1}{2}
  \end{eqnarray*}
  this can be achieved e.g. by setting
  \[ R = \| \theta^{\circ 3} \|^{\frac{1}{3}}_{L_{[0, 1]}^p H^{- \delta}} + \|
     \theta^{\circ 2} \|^{\frac{1}{2}}_{L_{[0, 1] \mathcal{C}^{- \delta}}^p} +
     \| \theta \|_{L_{[0, 1]}^{\infty} \mathcal{C}^{- \delta}} + 1 \quad
     \text{and } T = \left( \frac{1}{10 R^2} \right)^{\frac{p}{p - 1}} . \]
  Thus by Banach's fixed point theorem we have that there exists a unique
  solution $v$ to \eqref{eqn:vWick}.
  
  Continuity in time follows as per usual from Stone's theorem, see e.g. the
  proof of Theorem 3.19 in {\cite{GUZ}}.
\end{proof}

We summarize this result as follows: The flow of the equation
\begin{equation}
  \left\{\begin{array}{lll}
    (\partial^2_t +\mathbb{H}^{\omega, K}) u & = & - u^{\circ 3}\\
    (u, \partial_t u) |_{t = 0} & = & (u_0 (\omega'), u_1 (\omega'))
  \end{array}\right. \label{eqn:waveWick}
\end{equation}
which we denote by $\Phi^{\omega} (t) ((u_0 (\omega'), u_1 (\omega'))) \assign
u (t)$ is measurable as a map
\begin{equation}
  \Phi^{\omega} : B (R) \rightarrow \theta + C ([0, T (R)] ; H^{1 - \delta})
  \cap C^1 ([0, T (R)] ; H^{- \delta}),
\end{equation}
where as before $\theta$ is as in \eqref{def:theta}, the linear propagator
applied to the initial data which satisfy
\begin{align*}
  &(u_0 (\omega'), u_1 (\omega')) \in B (R), \text{ where}\\
  &B (R) \assign  \left\{ (u_0, u_1) \in \tmop{supp} \left(
  \mu^{\mathbb{H}^{w, K}} \otimes \mu^{\mathbb{I}_{L^2}} \right. : \|
  \theta^{\circ 3} \|^{\frac{1}{3}}_{L_{[0, 1]}^p H^{- \delta}} + \|
  \theta^{\circ 2} \|^{\frac{1}{2}}_{L_{[0, 1] \mathcal{C}^{- \delta}}^p} + \|
  \theta \|_{L_{[0, 1]}^{\infty} \mathcal{C}^{- \delta}} \leqslant R - 1
  \right\}
\end{align*}
for $R \gg 1$ and $0 < T < 1$ satisfying $T = \left( \frac{1}{10 R^2}
\right)^{\frac{p}{p - 1}}$ where $p \gg 2$ is taken to be large.

Proposition \ref{prop:wavewick} implies that the measure of $B (R)$ is large,
i.e.
\begin{equation}
  \mathbb{P}' (B (R)^c) \lesssim e^{- C R} .
\end{equation}
This means one has
\[ \| \Phi^{\omega} (u_0 (\omega'), u_1 (\omega')) - \theta \|_{L_{[0, T
   (R)]}^{\infty} H^{1 - \delta}} \leqslant R \text{ and }  \| \Phi^{\omega}
   (u_0 (\omega'), u_1 (\omega')) \|_{L_{[0, T (R)]}^{\infty} H^{- \delta}}
   \leqslant 2 R \]
for such initial data.

\begin{remark}
  \label{rem:focussing}As we have a Wick cube in the equation which is a cube
  renormalized by a logarithmically diverging constant times the function, see
  Section \ref{sec:AndersonWick}, one could consider also the
  \tmtextit{focusing} version of the equation (meaning to change the sign of
  the nonlinearity) with Gaussian initial data and tune the parameters in such
  a way that the divergence cancels with the logarithmically diverging
  constant, see \eqref{def:ceps}, in the definition of the renormalized
  product of the Anderson Hamiltonian.
  
  As the short-time well-posedness does not depend on the sign and is
  continuous in the noise and the initial data, one would obtain local-in-time
  dynamics for this focusing equation where the infinities cancel. Clearly
  the globalization and invariance arguments will fail, however. 
\end{remark}

\subsection{Local-in-time convergence of approximations }

Furthermore we need an analogous short-time well-posedness result for
different approximations to the SPDE \eqref{eqn:vWick}. We consider three
different approximations and one approximation that combines two of those.

One natural approximation is to smoothen the noise $\xi$ i.e. replacing the
operator $\mathbb{H}^{\omega, K}$ by the operator
$\mathbb{H}_{\delta}^{\omega, K}$from \eqref{def:HKeps} and to replace the
Wick powers by regularized Wick powers as in Section \ref{sec:AndersonWick}.
This would amount to solving the SPDE
\begin{eqnarray}
  (\partial^2_t +\mathbb{H}_{\delta}^{\omega, K}) u_{(\delta)} +
  u_{(\delta)}^{\circ_{\delta} 3} & = & 0  \label{eqn:deltaNLW}\\
  (u_{(\delta)}, \partial_t u_{(\delta)}) |_{t = 0} & = & (u^{(\delta)}_0
  (\omega'), u_1 (\omega')), \nonumber
\end{eqnarray}
where one would define the Wick ordering $(\cdummy)^{\circ_{\delta}}$
analogously to Section \ref{sec:AndersonWick} by replacing
$\mathbb{H}^{\omega, K}$ by $\mathbb{H}_{\delta}^{\omega, K}$ and
$u^{(\delta)}_0$is in the support of the Gaussian measure with covariance
$(\mathbb{H}_{\delta}^{\omega, K})^{- 1}$. While this seems like a very
natural approximation, it is actually not very useful if we want to
approximate the dynamics \eqref{eqn:waveWick} since one has not only a
different operator in the equation but also, crucially, a different reference
Gaussian measure $\mu^{\mathbb{H}_{\delta}^{\omega, K}}$ for the support of
the initial condition which is mutually singular w.r.t. the reference Gaussian
for \eqref{eqn:waveWick} $\mu^{\mathbb{H}^{\omega, K}},$ namely the Anderson
Free Field as was remarked in Lemma \ref{lem:abscont}. This then of course
also leads to a different Wick ordered nonlinearity for which one would have
to prove strong enough convergence.

\

Another downside of the approximation \eqref{eqn:deltaNLW} is that it does
not behave well under finite dimensional projection which is crucial if one
wants to prove invariance. This leads us to the next natural choice, namely is
the finite dimensional Galerkin approximation, where we project the equation
onto the span of eigenfunctions $f_n$ of $\mathbb{H}^{\omega, K}$, see
\eqref{eqn:Heigen}. To this end we define the projection
\begin{eqnarray*}
  P_{\leqslant N} & :& L^2 (\mathbb{T}^2) \rightarrow L^2 (\mathbb{T}^2)\\
  P_{\leqslant N} g & \assign &\underset{n \leqslant N}{\sum} (g, f_n) f_n
\end{eqnarray*}
for $N \in \mathbb{N}$ and $u_N$ as the solution to
\begin{eqnarray}
  (\partial^2_t +\mathbb{H}^{\omega, K}) u_N + P_{\leqslant N} (u_N^{\circ 3})
  & = & 0  \label{eqn:finiteNLW}\\
  (u_N, \partial_t u_N) |_{t = 0} & = & (P_{\leqslant N} u_0 (\omega'),
  P_{\leqslant N} u_1 (\omega')) \nonumber
\end{eqnarray}
and we denote by $\Phi^{\omega, N}$ its flow. This will be useful in Section
\ref{sec:global} since this is just a finite dimensional Hamiltonian system
whose Gibbs measure is invariant and approximates $\nu^{\omega}$. The finite
dimensional projections are also compatible with the choice of initial
conditions (see \eqref{eqn:gaussinitial}) and the Wick product, see Section
\ref{sec:AndersonWick}.

\

The third approximation we consider is a regularization of the nonlinearity
in which we replace the cubic nonlinearity by a regularized cube in a way that
we still get an invariant dynamics. The modified equation then reads \
\begin{eqnarray}
  (\partial^2_t +\mathbb{H}^{\omega, K}) u_{\varepsilon} + \rho_{\varepsilon}
  \ast ((u_{\varepsilon} \ast \rho_{\varepsilon})^{\circ 3}) & = & 0 
  \label{eqn:epsNLW}\\
  (u_{\varepsilon}, \partial_t u_{\varepsilon}) |_{t = 0} & = & (u_0
  (\omega'), u_1 (\omega')), \nonumber
\end{eqnarray}
where $\rho_{\varepsilon}$ is the convolution with a standard symmetric
mollifier. Again we define $\Phi_{\varepsilon}^{\omega}$ as its flow.

This approximation has the upside that it has the (for now formally) invariant
measure $\nu^{\varepsilon}$ which is mutually absolutely continuous w.r.t. the
Anderson Free Field $\mu^{\mathbb{H}^{\omega, K}}$ with smooth density
\[ \frac{d \nu^{\omega}_{\varepsilon}}{d \mu^{\mathbb{H}^{\omega, K}}} (v) =
   \int_{\mathbb{T}^2} (v \ast \rho_{\varepsilon})^{\circ 4} d x. \]
By construction of the Wick ordering in Section \ref{sec:AndersonWick}, we
have that this density converges strongly in $L^p $ to
\[ \frac{d \nu^{\omega}}{d \mu^{\mathbb{H}^{\omega, K}}} (v) =
   \int_{\mathbb{T}^2} v^{\circ 4} d x \]
implying convergence in total variation of the measures.

Due to this property, one may use the the same initial data when approximating
the equation \eqref{eqn:waveWick} by \eqref{eqn:epsNLW}.

Finally we define an approximation that combines the last two, i.e. we define
the Galerkin approximation of \eqref{eqn:epsNLW} so
\begin{eqnarray}
  (\partial^2_t +\mathbb{H}^{\omega, K}) u_{N, \varepsilon} + P_{\leqslant N}
  (\rho_{\varepsilon} \ast ((u_{N, \varepsilon} \ast
  \rho_{\varepsilon})^{\circ 3})) & = & 0  \label{eqn:NepsNLW}\\
  (u_{N, \varepsilon}, \partial_t u_{N, \varepsilon}) |_{t = 0} & = &
  (P_{\leqslant N} u_0 (\omega'), P_{\leqslant N} u_1 (\omega')), \nonumber
\end{eqnarray}
writing $\Phi_{\varepsilon}^{\omega, N}$ for its flow.

By the Hamiltonian structure, as in {\cite{OhLa2020}}, the equations
\eqref{eqn:finiteNLW} and \eqref{eqn:NepsNLW} are globally well-posed and
their flows leave the following finite dimensional Gibbs measures
\begin{eqnarray*}
  d \nu_N^{\omega} \assign & \exp \left( - \frac{1}{4} \int_{\mathbb{T}^2}
  (P_{\leqslant N} \phi)^{\circ 4} \mathd x \right) \mathd
  \mu^{\mathbb{H}^{\omega, K}} (\phi) \otimes \mathd \mu^{\mathbb{I}_{L^2}}
  (\partial_t \phi)\\
  \text{and} & \\
  d \nu_{N, \varepsilon}^{\omega} \assign & \exp \left( - \frac{1}{4}
  \int_{\mathbb{T}^2} (\rho_{\varepsilon} \ast P_{\leqslant N} \phi)^{\circ 4}
  \mathd x \right) \mathd \mu^{\mathbb{H}^{\omega, K}} \otimes \mathd
  \mu^{\mathbb{I}_{L^2}} (\partial_t \phi)
\end{eqnarray*}
invariant respectively.

For \eqref{eqn:epsNLW} one has the same local well-posedness as for the
limiting equation \eqref{eqn:waveWick}. Later we will globalize the solutions
to these equations and prove that their flows are invariant.

We summarize these results in the following result.

\begin{proposition}[Well-posedness of approximate equations]
  \label{prop:lwpapprox}
  
  Let $(u_0, u_1) \in \tmop{supp} \left( \mu^{\mathbb{H}^{\omega, K}} \otimes
  \mu^{\mathbb{I}_{L^2}} \right)$, then the flows of the equations
  \eqref{eqn:finiteNLW} and \eqref{eqn:NepsNLW} called $\Phi^{\omega, N}$ and
  $\Phi_{\varepsilon}^{\omega, N}$respectively, exist for all times for
  initial data $(P_{\leqslant N} u_0, P_{\leqslant N} u_1)$ and we have the
  following convergence
  \[ \| \Phi_{\varepsilon}^{\omega, N} (P_{\leqslant N} u_0, P_{\leqslant N}
     u_1) - \Phi^{\omega, N} (P_{\leqslant N} u_0, P_{\leqslant N} u_1)
     \|_{L_{[0, T]}^{\infty} L^2 (\mathbb{T}^2)} \rightarrow 0 \text{ as }
     \varepsilon \rightarrow 0 \text{ for all } T > 0. \]
  Now let $R \gg 1,$ $p \gg 2$ and $T = \left( \frac{1}{10 R^2}
  \right)^{\frac{p}{p - 1}}$ as well as $\delta > 0$ small. If the initial
  data satisfy additionally
  \[ \| u^{\circ 3}_0 \|^{\frac{1}{3}}_{\mathcal{C}^{- \delta}} + \| u^{\circ
     2}_0 \|^{\frac{1}{2}}_{\mathcal{C}^{- \delta}} + \| u_0
     \|_{\mathcal{C}^{- \delta}} + \| u_1 \|_{H^{^{- 1 - \delta}}} + 1
     \leqslant R \]
  then the flow $\Phi^{\omega}_{\varepsilon}$ of the equation
  \eqref{eqn:epsNLW} exists up to the time $T$ and we have
  \begin{eqnarray}
    \| \Phi^{\omega, \varepsilon} (u_0, u_1) \|_{L_{[0, T]}^{\infty} H^{-
    \delta}} & \leqslant & 2 R \\
    \underset{\varepsilon \rightarrow 0}{\lim} \| \Phi^{\omega}_{\varepsilon}
    (u_0, u_1) - \Phi^{\omega} (u_0, u_1) \|_{L_{[0, T]}^{\infty} H^{-
    \delta}} & = & 0 \\
    \underset{N \rightarrow \infty}{\lim} \| \Phi^{\omega}_{\varepsilon} (u_0,
    u_1) - \Phi^{\omega, N}_{\varepsilon} (P_{\leqslant N} u_0, P_{\leqslant
    N} u_1) \|_{L_{[0, T]}^{\infty} H^{- \delta}} & = & 0, 
  \end{eqnarray}
  where we recall that $\Phi^{\omega} $is the flow of the full equation
  \eqref{eqn:waveWick} which was shown to exist up to time $T$ in Theorem
  \ref{thm:localwp} for such initial data.
\end{proposition}

\section{Globalization and Invariance}\label{sec:global}

By the computations on the covariance, we have that the pair $\left(
\mu^{\mathbb{H}^{\omega, K}}, \mu^{\mathbb{I}_{L^2}} \right)$ is invariant
under the flow of the linear equation
\begin{equation}
  (\partial^2_t +\mathbb{H}^{\omega, K}) u = 0 \qquad \begin{array}{rl}
    (u, \partial_t u) |_{t = 0} & = (u_0, u_1) .
  \end{array}
\end{equation}
Now, let $\rho_{\varepsilon}$ be a standard mollifier. Let $P_{\leqslant N}$
be the projection on to the first $N$ eigenfunctions of $\mathbb{H}^{\omega,
K}$. We will denote by $u^{N, \varepsilon}$ the solution to the equation
\begin{equation}
  \begin{array}{ll}
    (\partial^2_t +\mathbb{H}^{\omega, K}) u^{N, \varepsilon}  &= -
    P_{\leqslant N} (\rho_{\varepsilon} \ast (\rho_{\varepsilon} \ast
    P_{\leqslant N} u^{N, \varepsilon})^{\circ 3})\\
    (u, \partial_t u) |_{t = 0}  &= \varphi = (\varphi_0, \varphi_1) .
  \end{array} \qquad \begin{array}{l}
    
  \end{array} \label{eq:approximate-wave-globa}
\end{equation}
Note that here the Wick ordering is taken with respect to the covariance of
$\rho_{\varepsilon} \ast W$ where $W$ is an AGFF. We will consider $u^{N,
\varepsilon}$ as a function of $\varphi$ and when we want to stress this
dependence, we will write $\Phi_t^{N, \varepsilon} \varphi = u^{N,
\varepsilon} (t)$, where $\Phi^{N, \varepsilon}$ denotes the flow (we again
drop the dependence on $\omega$ of this and related objects for brevity). We
will also consider the solution $u^{\varepsilon}$ equation
\begin{equation}
  (\partial^2_t +\mathbb{H}^{\omega, K}) u^{\varepsilon} = -
  (\rho_{\varepsilon} \ast (\rho_{\varepsilon} \ast u^{\varepsilon})^{\circ
  3}) \quad \begin{array}{rl}
    (u, \partial_t u) |_{t = 0} & = \varphi = (\varphi_0, \varphi_1),
  \end{array} \label{eq:approximate-wave-globa-2}
\end{equation}
and the Wick ordering is the same as above. We write $\Phi_t^{\varepsilon}
\varphi = u^{\varepsilon} (t)$. Finally we also consider
\begin{equation}
  (\partial^2_t +\mathbb{H}^{\omega, K}) u = - u^{\circ 3} \quad
  \begin{array}{rl}
    (u, \partial_t u) |_{t = 0} & = \varphi = (\varphi_0, \varphi_1) .
  \end{array} \label{eq:wave-global}
\end{equation}
where the Wick ordering is now taken with respect to the Anderson free field,
see Section \ref{sec:AndersonWick}. We will denote $\Phi_t \varphi = u (t)$.
Our goal is to show that \eqref{eq:wave-global} has $\nu$-almost surely global
solutions and $\nu$ is invariant under the flow $\Phi_t$. To do this we
implement the well-known Bourgain argument. \

\subsection{Approximation by finite dimensional system}

In this section we fix an $\varepsilon > 0$ and will prove that a smoothened
version of our system can be approximated by a finite-dimensional one. Note
that \eqref{eq:approximate-wave-globa} splits into a finite dimensional
Hamiltonian dynamical system and a linear equation. By Liouville's theorem in
finite dimensions we know that
\[ \nu^{N, \varepsilon} = \left( \exp \left( - \frac{1}{4}
   \int_{\mathbb{T}^2} (\rho_{\varepsilon} \ast P_{\leqslant N} \phi)^{\circ
   4} \mathd x \right) \mathd \mu^{\mathbb{H}^{\omega, K}}, \mathd
   \mu^{\mathbb{I}_{L^2}} \right) \]
is invariant under the flow of \eqref{eq:approximate-wave-globa}. Observe that
if we keep $\varepsilon$ fixed $(\rho_{\varepsilon} \ast P_{\leqslant N}
\phi)^{\circ 4}$ is bounded below uniformly in $N$ since the renormalization
function $\mathbb{E} (\rho_{\varepsilon} \ast \phi) (x)^2$ is bounded
uniformly in $N$ and $x$ (but not $\varepsilon$). Furthermore,
$(\rho_{\varepsilon} \ast P_{\leqslant N} \phi)^{\circ 4} \rightarrow
(\rho_{\varepsilon} \ast \phi)^{\circ 4}$ in $L^1 (\mathbb{T}^2)$ $\mu$-almost
surely. Combined with the lower bound, this gives us by dominated convergence
\begin{eqnarray*}
  \nu^{N, \varepsilon} & \rightarrow & \nu^{\varepsilon} : = \left( \exp
  \left( - \frac{1}{4} \int_{\mathbb{T}^2} (\rho_{\varepsilon} \ast
  \phi)^{\circ 4} \mathd x \right) \mathd \mu^{\mathbb{H}^{\omega, K}}, \mathd
  \mu^{\mathbb{I}_{L^2}} \right)
\end{eqnarray*}
in total variation. Our first goal is to show that $\nu^{\varepsilon}$ is
invariant under $\Phi^{\varepsilon}$. To this end we first have to establish
that $u^{\varepsilon}$ exists $\nu^{\varepsilon}$-almost surely for all times.
We begin with the following local well-posedness statement:

\begin{lemma}
  \label{lemma-2approx-local}For every $\delta > 0$ there exists an $L \in
  \mathbb{N}$ and $c > 0$ such that for $t \leqslant c (1 + \| \varphi
  \|_{H^{- \delta}})^{- L}$ such that
  \[ \| \Phi^{N, \varepsilon}_t \varphi \|_{H^{- \delta}} \leqslant 2 \|
     \varphi \|_{H^{- \delta}} . \]
  Note that $c, L$ can depend on $\varepsilon$ and $\delta$ but not on $N$.
\end{lemma}

\begin{proof}
  This is similar to Proposition \ref{prop:lwpapprox}. 
\end{proof}

\begin{lemma}
  \label{approx-global-well}With the notation of Lemma
  \ref{lemma-2approx-local} we have
  \[ \sup_{N \in \mathbb{N}} \nu^{N, \varepsilon} (\sup_{0 \leqslant t
     \leqslant T} \| \Phi^{N, \varepsilon}_t \varphi \|_{H^{- \delta}} > 2 R)
     \rightarrow 0 \quad \text{as } R \rightarrow \infty . \]
\end{lemma}

\begin{proof}
  By Lemma \ref{lemma-2approx-local} we have for $\tau = c (1 + R)^{- L}$,
  using invariance in the 4th line of the following computation
  \begin{eqnarray*}
    &  & \nu^{N, \varepsilon} (\sup_{0 \leqslant t \leqslant T} \| \Phi^{N,
    \varepsilon}_t \varphi \|_{H^{- \delta}} > 2 R)\\
    & \leqslant & \nu^{N, \varepsilon} (\sup_{n \leqslant T / \tau} \|
    \Phi^{N, \varepsilon}_{n \tau} \varphi \|_{H^{- \delta}} > R)\\
    & \leqslant & \sum_{n \leqslant T / \tau} \nu^{N, \varepsilon} (\|
    \Phi^{N, \varepsilon}_{n \tau} \varphi \nobracket \| \nobracket_{H^{-
    \delta}} > R)\\
    & = & (T / \tau) \nu^{N, \varepsilon} (\| \varphi \|_{H^{- \delta}} >
    R)\\
    & \leqslant & C (T) (1 + R)^L \nu^{N, \varepsilon} (\| \varphi \|_{H^{-
    \delta}} > R)\\
    & \leqslant & C (T) (1 + R)^L R^{- p}
  \end{eqnarray*}
  for any $p < \infty$. Choosing $p > L$ and taking the supremum over $N$ we
  can conclude. 
\end{proof}

\begin{lemma}
  \label{eq:convegence-1}Let $u^{\varepsilon}, u_N^{\varepsilon}$ be solutions
  to equations \eqref{eq:approximate-wave-globa-2} and
  \eqref{eq:approximate-wave-globa} respectively and suppose that \
  \[ \sup_{0 \leqslant t \leqslant T} (\| u^{\varepsilon} (t) \|_{H^{-
     \delta}} + \| u^{N, \varepsilon} (t) \|_{H^{- \delta}}) \leqslant R. \]
  Then, for $0 < \delta' < \delta$
  \[ \sup_{0 \leqslant t \leqslant T} \| (u^{N, \varepsilon} -
     u^{\varepsilon}) (t) \|_{H^{- \delta}} \lesssim_{\varepsilon, \delta,
     \delta'} R^2 N^{- \delta'} . \]
\end{lemma}

\begin{proof}
  Denote
  \[ f^{\varepsilon, N} (\varphi) = P_{\leqslant N} (\rho_{\varepsilon} \ast
     (\rho_{\varepsilon} \ast P_{\leqslant N} \varphi)^{\circ 3}) =
     P_{\leqslant N} (\rho_{\varepsilon} \ast ((\rho_{\varepsilon} \ast
     P_{\leqslant N} \varphi)^{\circ 3} - a_N^{\varepsilon}
     (\rho_{\varepsilon} \ast P_{\leqslant N} \varphi))) \]
  \[ f^{\varepsilon} (\varphi) = (\rho_{\varepsilon} \ast (\rho_{\varepsilon}
     \ast \varphi)^{\circ 3}) = (\rho_{\varepsilon} \ast ((\rho_{\varepsilon}
     \ast \varphi)^3 - a_N^{\varepsilon} (\rho_{\varepsilon} \ast \varphi)))
  \]
  where
  \[ a_N^{\varepsilon} (x) = 3\mathbb{E} [(\rho_{\varepsilon} \ast
     P_{\leqslant N} W)^2 (x)] \qquad a^{\varepsilon} (x) = 3\mathbb{E}
     [(\rho_{\varepsilon} \ast W)^2 (x)] = a^{\varepsilon} (0) \]
  Note that $a^{\varepsilon}$ is translation invariant (hence a constant
  function) while $a_N^{\varepsilon}$ is not. Furthermore $a_N^{\varepsilon}
  \rightarrow a^{\varepsilon}$ in $C_b^2 (\mathbb{T}^2)$ as $N \rightarrow
  \infty$. By Remark \ref{remark:N-delta} and the smoothing properties of
  $\rho_{\varepsilon} \ast$, we can estimate
  \begin{eqnarray*}
    &  & \| f^{\varepsilon, N} (\varphi) - f^{\varepsilon} (\varphi) \|_{H^{-
    \delta}}\\
    & \leqslant & \| P_{> N} (\rho_{\varepsilon} \ast (\rho_{\varepsilon}
    \ast P_{\leqslant N} \varphi)^{\circ 3}) \|_{H^{- \delta}} + \|
    (\rho_{\varepsilon} \ast (\rho_{\varepsilon} \ast P_{\leqslant N}
    \varphi)^{\circ 3}) - (\rho_{\varepsilon} \ast (\rho_{\varepsilon} \ast
    \varphi)^{\circ 3}) \|_{H^{- \delta}}\\
    & \leqslant & N^{- 1} \| (\rho_{\varepsilon} \ast (\rho_{\varepsilon}
    \ast P_{\leqslant N} \varphi)^{\circ 3}) \|_{H^{1 - \delta}} + \|
    (\rho_{\varepsilon} \ast (\rho_{\varepsilon} \ast P_{\leqslant N}
    \varphi)^3) - (\rho_{\varepsilon} \ast (\rho_{\varepsilon} \ast
    \varphi)^3) \|_{H^{- \delta}}\\
    &  & + a^{\varepsilon} \| (\rho_{\varepsilon} \ast (\rho_{\varepsilon}
    \ast P_{> N} \varphi)) \|_{H^{- \delta}} + \| (a^{\varepsilon}_N -
    a^{\varepsilon}) (\rho_{\varepsilon} \ast (\rho_{\varepsilon} \ast
    P_{\leqslant N} \varphi)) \|_{H^{- \delta}}\\
    & \lesssim_{\varepsilon} & N^{- 1} \| : (\rho_{\varepsilon} \ast
    (\rho_{\varepsilon} \ast P_{\leqslant N} \varphi)^3) : \|_{H^{1 - \delta}}
    + \| \rho_{\varepsilon} \ast P_{> N} \varphi \|^3_{L^{\infty}} + 
    a^{\varepsilon} \| P_{> N} \varphi \|_{H^{- 2 \delta}} \\
    & &+ \| (a^{\varepsilon}_N - a^{\varepsilon}) \|_{L^{\infty}} \|
    \rho_{\varepsilon} \ast P_{\leqslant N} \varphi \|_{L^{\infty}}\\
    & \lesssim_{\varepsilon} & N^{- \delta'} (1 + \| \varphi \|^3_{H^{-
    \delta}}) .
  \end{eqnarray*}
  It is not hard to see that $f^{\varepsilon}$ is locally Lipschitz continuous
  on $H^{- \delta}$.
  
  From the mild formulation, we get
  
  \[ \begin{array}{ll}
       & \| (u^{\varepsilon}_N - u^{\varepsilon}) (t) \|_{H^{- \delta}}\\
       = & \left\| \int^t_0 \frac{\sin \left( (t - s) 
       \sqrt{\mathbb{H}^{\omega, K}} \right)}{\sqrt{\mathbb{H}^{\omega, K}}}
       f^{\varepsilon, N} (u^{N, \varepsilon} (s)) - f^{\varepsilon}
       (u^{\varepsilon} (s)) \mathd s \right\|_{H^{- \delta}}\\
       = & \left\| \int^t_0 \frac{\sin \left( (t - s) 
       \sqrt{\mathbb{H}^{\omega, K}} \right)}{\sqrt{\mathbb{H}^{\omega, K}}}
       \left( \left( f^{\varepsilon, N} (u^{\varepsilon}_N (s)) -
       f^{\varepsilon} \left( u_N^{\varepsilon} (s) \right) \right) +
       (f^{\varepsilon} (u^{N, \varepsilon} (s)) - f^{\varepsilon}
       (u^{\varepsilon} (s))) \right. \mathd s \right\|_{H^{- \delta}}\\
       \lesssim_{\varepsilon, T} & N^{- \delta'} \sup_{s \leqslant T} (1 + \|
       u^{N, \varepsilon} (s) \|^3_{H^{- \delta}}) + \int^T_0 \|
       f^{\varepsilon} (u^{N, \varepsilon} (s)) - f^{\varepsilon}
       (u^{\varepsilon} (s)) \|_{H^{- \delta}} \mathd s\\
       \lesssim_{\varepsilon, T, R} & N^{- \delta'} + \int^T_0 \| (\nobracket
       u^{N, \varepsilon} (s)) - \nobracket u^{\varepsilon} (s)) \nobracket
       \|_{H^{- \delta}} \mathd s.
     \end{array} \]
  
  So from Gronwall's Lemma we can conclude. 
\end{proof}

Note that if $\sup_{0 \leqslant t \leqslant T} (\| u^{N, \varepsilon} (t)
\|_{H^{- \delta}}) \leqslant R$ then from \eqref{eq:approximate-wave-globa} we
have
\[ \| u^{N, \varepsilon} \|_{W_t^{\delta, \infty} H_x^{- 2 \delta}} \leqslant
   \| f^{\varepsilon, N} (u^{N, \varepsilon}) \|_{L_t^{\infty} H_x^{- \delta}}
   + \| S (t) \varphi \|_{W_t^{\delta, \infty} H_x^{- 2 \delta}}
   \lesssim_{\varepsilon} 1 + \| u^{N, \varepsilon} \|^3_{L_t^{\infty} H_x^{-
   \delta}} \leqslant 1 + R^3, \]
where $S (t) \varphi$ is a short-hand notation for the linear part of the
equation. In particular $u^{N, \varepsilon}$ possesses a convergent
subsequence in $L_t^{\infty} H_x^{- \delta}$ by the compact embedding
$W_t^{\delta, \infty} H_x^{- 2 \delta} \hookrightarrow L_t^{\infty} H_x^{- 4
\delta}$. Any subsequential limit can be seen to satisfy
\eqref{eq:approximate-wave-globa-2} in the same way as in the proof of Lemma
\ref{eq:convegence-1}. In particular if we denote by
\[ \Sigma_R^{\varepsilon} = \{ \varphi : \sup_{0 \leqslant t \leqslant T} \|
   \Phi^{\varepsilon}_t \varphi \|_{H^{- \delta}} \leqslant R \} \]
and
\[ \Sigma_R^{\varepsilon, N} = \{ \varphi : \sup_{0 \leqslant t \leqslant T}
   \| \Phi^{\varepsilon, N}_t \varphi \|_{H^{- \delta}} \leqslant R \}, \]
we have
\[ \Sigma_R^{\varepsilon} \supset \limsup_{N \rightarrow \infty}
   \Sigma_R^{\varepsilon, N} = \bigcap_{N = 1}^{\infty} \bigcup^{\infty}_{N_1
   = N} \Sigma_R^{\varepsilon, N} . \]
This implies by Fatou's lemma that
\begin{equation}
  \label{eq:global-wellposed} {\nu^{\varepsilon}}  (\Sigma_R^{\varepsilon})
  \geqslant \limsup_{N \rightarrow \infty} \nu^{\varepsilon}
  (\Sigma_R^{\varepsilon, N}) = \limsup_{N \rightarrow \infty}
  \nu^{\varepsilon, N} (\Sigma_R^{\varepsilon, N}) \geqslant 1 - C (T) (1 +
  R)^L R^{- p}
\end{equation}
where the equality holds since $\nu^{\varepsilon, N} \rightarrow
\nu^{\varepsilon}$ in total variation. In particular, $\Phi^{\varepsilon}_t$
is well defined for all $t$, $\nu^{\varepsilon}$-almost surely. Note that
since $\Phi^{\varepsilon}_t \varphi$ is well defined
$\nu^{\varepsilon}$-almost surely we can define the pushforward measure
$(\Phi^{\varepsilon}_t)^{\ast} \nu^{\varepsilon}$ by
\[ \int f (\varphi) \mathd (\Phi^{\varepsilon}_t)^{\ast} \nu^{\varepsilon} =
   \int f (\Phi^{\varepsilon}_t \varphi) \mathd \nu^{\varepsilon} \]
for any continuous and bounded function $f : H^{- \delta} \rightarrow
\mathbb{R}$.

We now establish invariance.

\begin{proposition}
  \label{prop:invariance-approx}For any $0 < t < \infty$ and $\varepsilon >
  0$one has
  \[ (\Phi^{\varepsilon}_t)^{\ast} \nu^{\varepsilon} = \nu^{\varepsilon} . \]
\end{proposition}

\begin{proof}
  We will show that for any Lipschitz-continuous and bounded function $f :
  H^{- \delta} \rightarrow \mathbb{R}$ with Lipschitz constant $L$, we have
  \[ \int f (\Phi^{\varepsilon}_t (\varphi)) \mathd \nu^{\varepsilon} = \int
     f (\varphi) \mathd \nu^{\varepsilon} . \]
  We already know that
  \[ \int f (\Phi^{\varepsilon, N}_t (\varphi)) \mathd \nu^{\varepsilon, N} =
     \int f (\varphi) \mathd \nu^{\varepsilon, N}, \]
  and since $\nu^{\varepsilon, N} \rightarrow \nu^{\varepsilon}$ in total
  variation
  \[ \lim_{N \rightarrow \infty} \int f (\varphi) \mathd \nu^{\varepsilon, N}
     = \int f (\varphi) \mathd \nu^{\varepsilon} . \]
  So it remains to show that
  \[ \lim_{N \rightarrow \infty} \int f (\Phi^{\varepsilon, N}_t (\varphi))
     \mathd \nu^{\varepsilon, N} = \int f (\Phi^{\varepsilon}_t (\varphi))
     \mathd \nu^{\varepsilon} \]
  Now we bound
  \begin{eqnarray*}
    &  & \left| \int f (\Phi^{\varepsilon, N}_t (\varphi)) \mathd
    \nu^{\varepsilon, N} - \int f (\Phi^{\varepsilon}_t (\varphi)) \mathd
    \nu^{\varepsilon} \right|\\
    & \leqslant & \left| \int f (\Phi^{\varepsilon, N}_t (\varphi)) \mathd
    \nu^{\varepsilon, N} - \int f (\Phi^{\varepsilon, N}_t (\varphi)) \mathd
    \nu^{\varepsilon} \right|\\
      & &+ \left| \int f (\Phi^{\varepsilon, N}_t (\varphi)) \mathd
    \nu^{\varepsilon} - \int f (\Phi^{\varepsilon}_t (\varphi)) \mathd
    \nu^{\varepsilon} \right|
  \end{eqnarray*}
  The first term goes to $0$ since $\nu^{\varepsilon, N}$ converges in total
  variation. For the second term we have
  \begin{eqnarray*}
    &  & \left| \int f (\Phi^{\varepsilon, N}_t (\varphi)) - f
    (\Phi^{\varepsilon}_t (\varphi)) \mathd \nu^{\varepsilon} \right|\\
    & \leqslant & \| f \|_{L^{\infty}} (\nu^{\varepsilon} (\sup_{t \leqslant
    T} \| \Phi^{\varepsilon, N}_t (\varphi) \|_{H^{- \delta}} \geqslant R) +
    \nu^{\varepsilon} (\sup_{t \leqslant T} \| \Phi^{\varepsilon}_t (\varphi)
    \|_{H^{- \delta}} \geqslant R))\\
    &  & + \left| \int \mathbbm{1}_{\{ \sup_t (\| \Phi^{\varepsilon, N}_t
    (\varphi) \|_{H^{- \delta}} + \| \Phi^{\varepsilon}_t (\varphi) \|_{H^{-
    \delta}}) \leqslant 2 R \}} (f (\Phi^{\varepsilon, N}_t (\varphi)) - f
    (\Phi^{\varepsilon}_t (\varphi))) \mathd \nu^{\varepsilon} \right|
  \end{eqnarray*}
  Now applying Lemma \ref{eq:convegence-1} we get by Lipschitz continuity of
  $f$
  \begin{eqnarray*}
    &  & \left| \int \mathbbm{1}_{\{ \sup_t (\| \Phi^{\varepsilon, N}_t
    (\varphi) \|_{H^{- \delta}} + \| \Phi^{\varepsilon}_t (\varphi) \|_{H^{-
    \delta}}) \leqslant 2 R \}} (f (\Phi^{\varepsilon, N}_t (\varphi)) - f
    (\Phi^{\varepsilon}_t (\varphi))) \mathd \nu^{\varepsilon} \right|\\
    & \leqslant & L R^2 N^{- \delta'}
  \end{eqnarray*}
  Finally, by Lemma \ref{approx-global-well} and \eqref{eq:global-wellposed}
  we get
  \begin{eqnarray*}
    \sup_{N \in \mathbb{N}} \nu^{\varepsilon} (\sup_t \| \Phi^{\varepsilon,
    N}_t (\varphi) \|_{H_x^{- \delta}} \geqslant R) + \nu^{\varepsilon}
    (\sup_t \| \Phi^{\varepsilon}_t (\varphi) \|_{H_x^{- \delta}} \geqslant R)
    & \rightarrow & 0
  \end{eqnarray*}
  as $R \rightarrow \infty$. From this we can conclude by taking $N
  \rightarrow \infty$ and then $R \rightarrow \infty$.
\end{proof}

\subsection{Removal of the smoothing}

We now want to show that \eqref{eq:wave-global} has global solutions
$\nu$-almost surely and that $\nu$ is invariant under $\Phi_t$. Firstly we
define $\theta (t) = U (t) \varphi_1 + S (t) \varphi_2$ to be the ``linear
part'' of the solution, we also set $\theta^{\varepsilon} = \rho_{\varepsilon}
\ast \theta .$ We then have
\begin{eqnarray*}
  \mathbb{E}_{\nu^{\varepsilon}} \left[ \int^T_0 \| \theta^{\varepsilon}
  (t)^{\circ i} \|^p_{B_{p, p}^{- \delta}} \mathd t \right] & = & \int_0^T
  \mathbb{E}_{\nu^{\varepsilon}} [\| \theta^{\varepsilon} (0)^{\circ i}
  \|^p_{B_{p, p}^{- \delta}}] \mathd t\\
  & \lesssim & T\mathbb{E}_{\mu^{\mathbb{H}^{\omega, K}}} [\|
  \theta^{\varepsilon} (0)^{\circ i} \|^{2 p}_{B_{p, p}^{- \delta}}]^{1 / 2}\\
  & \leqslant & C T
\end{eqnarray*}
where, in the second line we used invariance of the AGFF with respect to the
free field and that $\nu^{\varepsilon}$ is absolutely continuous with respect
to $\mu^{\mathbb{H}^{\omega, K}}$ and the density is in $L^2 (\mu)$ uniformly
in $\varepsilon$. So in particular $(\theta^{\varepsilon})^{\circ i}$ are in
$L^p \mathcal{C}^{- \delta}$ almost surely with respect to both
$\nu^{\varepsilon}, \nu$ uniformly in $\varepsilon$ and the same holds for
$\theta^{\circ i}$. From now on we write $v^{\varepsilon} = u^{\varepsilon} -
\theta^{\varepsilon} .$

\begin{lemma}
  \label{bound-convergence}There exists $\delta > 0$ such that, assuming
  \[ \sup_{t \leqslant T} (\| v^{\varepsilon} (t) \|_{H^{1 - \delta}})
     \leqslant R \]
  and
  \[ \sup_{i \leqslant 3} \sup_{t \leqslant T} (\| \theta^{\circ i} \|_{H^{-
     \delta}} + \| (\theta^{\varepsilon})^{\circ i} \|_{H^{- \delta}})
     \leqslant R, \]
  we have
  \[ \left| v^{\varepsilon} (t) - \int^t_0 S (t - s) (u^{\varepsilon}
     (s))^{\circ 3} \mathd s \right| \lesssim R^3 \mathcal{R}^{\varepsilon},
  \]
  where $\mathcal{R}^{\varepsilon}$ is a random variable such that $\|
  \mathcal{R}^{\varepsilon} \|_{L^p (\mu)} \rightarrow 0$ as $\varepsilon
  \rightarrow 0$.
\end{lemma}

\begin{proof}
  Recall that
  \begin{eqnarray*}
    0 & = & v^{\varepsilon} (t) - \int^t_0 S (t - s) \rho_{\varepsilon} \ast
    (\rho_{\varepsilon} \ast u^{\varepsilon} (s))^3 \mathd s
  \end{eqnarray*}
  and
  \begin{eqnarray*}
    &  & v^{\varepsilon} (t) - \int^t_0 S (t - s) (u^{\varepsilon}
    (s))^{\circ 3} \mathd s\\
    & = & v^{\varepsilon} (t) - \int^t_0 S (t - s) \rho_{\varepsilon} \ast
    (\rho_{\varepsilon} \ast u^{\varepsilon} (s))^{\circ 3} \mathd s\\
    &  & + \int^t_0 S (t - s) \rho_{\varepsilon} \ast (\rho_{\varepsilon}
    \ast u^{\varepsilon} (s))^{\circ 3} \mathd s - \int^t_0 S (t - s)
    (u^{\varepsilon} (s))^{\circ 3} \mathd s
  \end{eqnarray*}
  so we have to estimate the last line. Recalling that $u^{\varepsilon} =
  \theta + v^{\varepsilon}$ that is equal to
  \begin{eqnarray*}
    &  & \int^t_0 S (t - s) \rho_{\varepsilon} \ast (\rho_{\varepsilon} \ast
    u^{\varepsilon} (s))^{\circ 3} \mathd s - \int^t_0 S (t - s)
    (u^{\varepsilon} (s))^{\circ 3} \mathd s\\
    & = & \sum_{i = 0}^3 \int^t_0 S (t - s) (((\rho_{\varepsilon} \ast \theta
    (s))^{\circ i} (\rho_{\varepsilon} \ast u^{\varepsilon} (s))^{i - j}) -
    (\theta (s))^{\circ i} (u^{\varepsilon})^{i - j}) \mathd s\\
    &  & + \int^t_0 S (t - s) ((\rho_{\varepsilon} \ast u^{\varepsilon}
    (s))^{\circ 3} - \rho_{\varepsilon} \ast (\rho_{\varepsilon} \ast
    u^{\varepsilon} (s))^3) \mathd s\\
    & \leqslant & \sum_{i = 0}^3 \int^t_0 \| \nobracket (((\rho_{\varepsilon}
    \ast \theta (s))^{\circ i} (\rho_{\varepsilon} \ast u^{\varepsilon}
    (s))^{i - j}) - (\theta (s))^{\circ i} (u^{\varepsilon})^{i - j}) \|_{H^{-
    \delta}} \nobracket \mathd s\\
    &  & + \int^t_0 \| (\rho_{\varepsilon} \ast u^{\varepsilon} (s))^3 -
    \rho_{\varepsilon} \ast : (\rho_{\varepsilon} \ast u^{\varepsilon} (s))^3
    : \|_{H^{- \delta}} \mathd s\\
    & \backassign & \Iota + \Iota \Iota
  \end{eqnarray*}
  Now to estimate $\Iota$ we have for $p$ sufficiently large and $1 / p + 1 /
  p' = 1$:
  \begin{align*}
    &   | ((\theta^{\varepsilon} (s))^{\circ i} (\rho_{\varepsilon} \ast
    u^{\varepsilon} (s))^{i - j}) - (\theta (s))^{\circ i}
    (u^{\varepsilon})^{i - j} |\\
     \leqslant & \| (\theta^{\varepsilon})^{\circ i} - (\theta)^{\circ i}
    \|_{B_{p, p}^{- 2 \delta}} (\| u^{\varepsilon} (s)^{i - j} \|_{B^{2
    \delta}_{p', p'}} + \| (\rho_{\varepsilon} \ast u^{\varepsilon} (s))^{i -
    j} \|_{B^{2 \delta}_{p', p'}})\\
      & + (\| (\theta^{\varepsilon})^{\circ i} \|_{B_{p, p}^{- 2 \delta}} +
    \| (\theta)^{\circ i} \|_{B_{p, p}^{- 2 \delta}}) \| u^{\varepsilon}
    (s)^{i - j} - (\rho_{\varepsilon} \ast u^{\varepsilon} (s))^{i - j}
    \|_{B^{2 \delta}_{p', p'}}\\
     \leqslant & \| (\theta^{\varepsilon})^{\circ i} - (\theta)^{\circ i}
    \|_{B_{p, p}^{- 2 \delta}} (\| u^{\varepsilon} (s) \|^{i - j}_{H^{1 -
    \delta}} + \| (\rho_{\varepsilon} \ast u^{\varepsilon} (s)) \|^{i -
    j}_{H^{1 - \delta}})\\
      & + (\| (\theta^{\varepsilon})^{\circ i} \|_{B_{p, p}^{- 2 \delta}} +
    \| (\theta)^{\circ i} \|_{B_{p, p}^{- 2 \delta}}) \\
    &\times \| u^{\varepsilon} (s) - (\rho_{\varepsilon} \ast u^{\varepsilon}
    (s)) \|_{H^{1 - \delta}} (\| u^{\varepsilon} (s) \|_{H^{1 - \delta}} + \|
    (\rho_{\varepsilon} \ast u^{\varepsilon} (s)) \|_{H^{1 - \delta}})^{i - j
    - 1}\\
     = & \| (\theta^{\varepsilon})^{\circ i} - (\theta)^{\circ i} \|_{B_{p,
    p}^{- 2 \delta}} R^3 + R^3 (\| (\theta^{\varepsilon})^{\circ i} \|_{B_{p,
    p}^{- 2 \delta}} + \| (\theta)^{\circ i} \|_{B_{p, p}^{- 2 \delta}}) \|
    u^{\varepsilon} (s) - (\rho_{\varepsilon} \ast u^{\varepsilon} (s))
    \|_{H^{1 - \delta}}
  \end{align*}
  Now
  \[ \| u^{\varepsilon} (s) - (\rho_{\varepsilon} \ast u^{\varepsilon} (s))
     \|_{H^{1 - \delta}} \lesssim \varepsilon^{\delta / 2} \| u^{\varepsilon}
     (s) \|_{H^{1 - \delta / 2}} \lesssim \varepsilon^{\delta / 2} R \]
  so
  \begin{eqnarray*}
    &  & \| (\theta^{\varepsilon})^{\circ i} (s) - (\theta)^{\circ i} (s)
    \|_{B_{p, p}^{- 2 \delta}} R^3\\
    &  & + R^3 (\| (\theta^{\varepsilon})^{\circ i} (s) \|_{B_{p, p}^{- 2
    \delta}} + \| \theta^{\circ i} (s) \|_{B_{p, p}^{- 2 \delta}}) \|
    u^{\varepsilon} (s) - (\rho_{\varepsilon} \ast u^{\varepsilon} (s))
    \|_{H^{1 - \delta}}\\
    & \leqslant & \| (\theta^{\varepsilon})^{\circ i} (s) - \theta^{\circ i}
    (s) \|_{B_{p, p}^{- 2 \delta}} R^3 + \varepsilon^{\delta / 2} R^4 (\|
    (\theta^{\varepsilon})^{\circ i} (s) \|_{B_{p, p}^{- 2 \delta}} + \|
    (\theta)^{\circ i} (s) \|_{B_{p, p}^{- 2 \delta}})
  \end{eqnarray*}
  so integrating in time, we have that
  \[ \Iota \lesssim R^3 \int^t_0 \| (\theta^{\varepsilon})^{\circ i} (s) -
     (\theta)^{\circ i} (s) \|_{B_{p, p}^{- 2 \delta}} \mathd s + R^4
     \varepsilon^{\delta / 2} \int^t_0 (\| (\theta^{\varepsilon})^{\circ i}
     (s) \|_{B_{p, p}^{- 2 \delta}} + \| \theta^{\circ i} (s) \|_{B_{p, p}^{-
     2 \delta}}) \mathd s \]
  and we recall that
  \[ \mathbb{E}_{\mu} \int^T_0 \| (\theta^{\varepsilon})^{\circ i} (s) -
     (\theta)^{\circ i} (s) \|^p_{B_{p, p}^{- 2 \delta}} \mathd s \rightarrow
     0. \]
  Now finally to estimate $\Iota \Iota$ we have
  \[ \int^t_0 \| (\rho_{\varepsilon} \ast u^{\varepsilon} (s))^{\circ 3} -
     \rho_{\varepsilon} \ast (\rho_{\varepsilon} \ast u^{\varepsilon}
     (s))^{\circ 3} \|_{H^{- \delta}} \mathd s \lesssim \varepsilon^{\delta /
     2} \int^t_0 \| (\rho_{\varepsilon} \ast u^{\varepsilon} (s))^{\circ 3}
     \|_{H^{- \delta / 2}} \mathd s \]
  and
  \begin{eqnarray*}
    \| (\rho_{\varepsilon} \ast u^{\varepsilon} (s))^{\circ 3} \|_{H^{- \delta
    / 2}} & \leqslant & \sum_{i = 0}^3 \| (\theta^{\varepsilon})^{\circ i}
    (u^{\varepsilon})^{3 - i} \|_{H^{- \delta / 2}}\\
    & \leqslant & \sum_{i = 0}^3 \| (\theta^{\varepsilon})^{\circ i}
    \|_{\mathcal{C}^{- \delta / 2}} \| (u^{\varepsilon})^{3 - i}
    \|_{H^{\delta}}\\
    & \leqslant & \sum_{i = 0}^3 \| (\theta^{\varepsilon})^{\circ i}
    \|_{\mathcal{C}^{- \delta / 2}} \| u^{\varepsilon} \|_{W^{\delta, 6}}^{3 -
    i}\\
    & \leqslant & \sum_{i = 0}^3 \| (\theta^{\varepsilon})^{\circ i}
    \|_{\mathcal{C}^{- \delta / 2}} \| u^{\varepsilon} \|_{H^{1 - \delta}}^{3
    - i}\\
    & \leqslant & (1 + R)^3 \sum_{i = 0}^3 \| (\theta^{\varepsilon})^{\circ
    i} \|_{\mathcal{C}^{- \delta / 2}}
  \end{eqnarray*}
  and again integrating in time we can conclude. 
\end{proof}

Let $v$ be the solution \eqref{eqn:vWick}. Similarly to the above we have the
following statement:

\begin{lemma}
  There exists a $\delta > 0$ such that, assuming that
  \[ \sup_{t \leqslant T} (\| v^{\varepsilon} (t) \|_{H^{1 - \delta}})
     \leqslant R, \]
  and
  \[ \sup_{i \leqslant 3}  \int^T_0 (\| \theta^{\circ i} \|^p_{H^{- \delta}}
     + \| (\theta^{\varepsilon})^{\circ i} \|^p_{H^{- \delta}}) \mathd t
     \leqslant R, \]
  we get \
  \[ \sup_{t \leqslant T} (\| v^{\varepsilon} (t) - v (t) \|_{H^{1 -
     \delta}}) \lesssim R \bar{\mathcal{R}}^{\varepsilon}, \]
  where $\bar{\mathcal{R}}^{\varepsilon}$ is given by \
  \[ \bar{\mathcal{R}}^{\varepsilon} \assign \sum_{0 \leqslant i \leqslant 3}
     \int^T_0 \| (\theta^{\varepsilon})^{\circ i} (s) - \theta^{\circ i} (s)
     \|^p_{B_{p, p}^{- 2 \delta}} \mathd s. \]
  In particular$\| \bar{\mathcal{R}}^{\varepsilon} \|_{L^p (\mu)} \rightarrow
  0$ as $\varepsilon \rightarrow 0$.
\end{lemma}

\begin{proof}
  Using the definition of $v^{\varepsilon}, v$ we have
  \begin{eqnarray*}
    &  & v (t) - v^{\varepsilon} (t)\\
    & = & \int^t_0 S (t - s) \rho_{\varepsilon} \ast \left( \theta^{\circ 3}
    (s) - (\theta^{\varepsilon})^{\circ 3} (s) + 3 \left( \theta^{\circ 2} (s)
    - (\theta^{\varepsilon})^{\circ 2} (s) \right) v (s) + 3
    (\theta^{\varepsilon})^{\circ 2} (v (s) - v^{\varepsilon} (s)) \right)
    \mathd s\\
    &  & + 3 \int^t_0 S (t - s) \rho_{\varepsilon} \ast ((\theta (s) -
    \theta^{\varepsilon} (s)) v^2 (s) + (\theta^{\varepsilon} (s)) (v^2 (s) -
    (v^{\varepsilon})^2 (s))) \mathd s\\
    &  & + \int^t_0 S (t - s) \rho_{\varepsilon} \ast (v^3 (s) -
    (v^{\varepsilon})^3 (s)) \mathd s
  \end{eqnarray*}
  Now with the same estimates as in the proof of Lemma
  \ref{bound-convergence} we obtain
  \begin{eqnarray*}
    \| v (t) - v^{\varepsilon} (t) \|_{H^{1 - 2 \delta}} & \lesssim & R^3
    \int^t_0 \| (\theta^{\varepsilon})^{\circ 2} (s) \|_{\mathcal{C}^{-
    \delta}} \| v (s) - v^{\varepsilon} (s) \|_{H^{1 - 2 \delta}} \mathd s + C
    R \varepsilon^{\delta}\\
    &  & + \sum_{0 \leqslant i \leqslant 3} \int^T_0 \|
    (\theta^{\varepsilon})^{\circ 2} (s) - \theta^{\circ 2} (s) \|^p_{B_{p,
    p}^{- 2 \delta}} \mathd s
  \end{eqnarray*}
  \[ \  \]
  and Gronwall's lemma gives
  \begin{eqnarray*}
    &  & \| v (t) - v^{\varepsilon} (t) \|_{H^{1 - 2 \delta}}\\
    & \lesssim & C R (\varepsilon^{\delta} + \bar{\mathcal{R}}^{\varepsilon})
    \exp \left( R^3 \int^T_0 \| \theta^{\circ 2}_{\varepsilon} (s)
    \|_{\mathcal{C}^{- \delta}} \mathd s \right)\\
    & \lesssim & \varepsilon^{\delta} (\varepsilon^{\delta} +
    \bar{\mathcal{R}}^{\varepsilon}) \exp (C R^4)
  \end{eqnarray*}
  from which we can conclude. 
\end{proof}

\begin{lemma}
  \label{bound-remainder}Assume that
  \[ \sup_{t \leqslant T} \| (\rho_{\varepsilon} \ast u^{\varepsilon})^{\circ
     3} \|_{H^{- \delta}} \leqslant R \]
  Then for $\alpha \in (0, 1)$ and $\delta > 0$
  \[ \| v^{\varepsilon} \|_{\mathcal{C}_t^{\alpha} H^{1 - \delta - \alpha}}
     \lesssim R. \]
  The analogous statement also holds for $u, v$. 
\end{lemma}

\begin{proof}
  By definition of $v^{\varepsilon}$ we have
  \[ v^{\varepsilon} (t) = \int^t_0 S (t - s) \rho_{\varepsilon} \ast
     ((\rho_{\varepsilon} \ast u^{\varepsilon})^{\circ 3}) \mathd s \]
  so by the properties of the Wave propagator
  \[ \| v^{\varepsilon} \|_{\mathcal{C}^{\alpha}_t ([0, T], H^{1 - \alpha -
     \delta})} \lesssim_T \| (u^{\varepsilon})^{\circ 3} \|_{L^2 ([0, T], H^{-
     \delta})} . \]
  which gives the statement. The proof for $v$ is analogous. 
\end{proof}

\begin{proposition}
  \label{prop:bound-wick-u}For $K > 0$ we have the bound
  \[ \nu^{\varepsilon} (\| (\rho_{\varepsilon} \ast u^{\varepsilon})^{\circ 3}
     \|_{L^p ([0, T], H^{- \delta})} \leqslant R) \geqslant 1 - \frac{C
     T}{R^p} . \]
\end{proposition}

\begin{proof}
  By Markov's inequality
  \begin{eqnarray*}
    &  & \nu^{\varepsilon} (\| (\rho_{\varepsilon} \ast
    u^{\varepsilon})^{\circ 3} \|_{L^p ([0, T], H^{- \delta})} \geqslant R)\\
    & \leqslant & \frac{1}{R^p} \mathbb{E}_{\nu^{\varepsilon}} [\|
    (\rho_{\varepsilon} \ast u^{\varepsilon})^{\circ 3} \|^p_{L^p ([0, T],
    H^{- \delta})}]\\
    & \leqslant & \frac{1}{R^p} \int^T_0 \mathbb{E}_{\nu^{\varepsilon}} [\|
    (\rho_{\varepsilon} \ast u^{\varepsilon} (t))^{\circ 3} \|^p_{H^{-
    \delta}}]\\
    & = & \frac{1}{R^p} \int^T_0 \mathbb{E}_{\nu^{\varepsilon}} [\|
    (\rho_{\varepsilon} \ast u^{\varepsilon} (0))^{\circ 3} \|^p_{H^{-
    \delta}}]\\
    & = & \frac{1}{R^p} C T
  \end{eqnarray*}
  where in the last line we used invariance of $\nu^{\varepsilon}$ under the
  flow of $u^{\varepsilon}$ .
\end{proof}

\begin{corollary}
  Recall that \ $v = u - \theta$. If \ $\| v^{\varepsilon} (t) \|_{H^{1 -
  \delta}} \leqslant R$, and for $i \leqslant 3$ $\|
  (\theta^{\varepsilon})^{\circ i} \|_{H^{- \delta}} \leqslant R$, then
  $u^{\varepsilon} \rightarrow u$ in $\mathcal{C}^{\alpha}_t ([0, T], H^{1 -
  \delta - \alpha})$.
\end{corollary}

\begin{corollary}
  With the same notation as above
  \[ \nu^{\varepsilon} (\| v^{\varepsilon} \|_{\mathcal{C}_t^{\alpha} H^{1 -
     \delta - \alpha}} \leqslant R) \geqslant 1 - \frac{C T}{R^p} . \]
\end{corollary}

\begin{proof}
  This follows immediately from Lemmas \ref{bound-remainder} and
  \ref{prop:bound-wick-u}.
\end{proof}

Next we show that having a sequence of uniformly bounded approximate solutions
for the cut-off flow is sufficient to construct the solution for the limiting
equation.

\begin{lemma}
  \label{lemma:subsequence}Assume that
  \[ \sup_{\varepsilon > 0} \sup_{0 \leqslant i \leqslant 3}  \|
     (\theta^{\varepsilon})^{\circ i} \|_{L^p ([0, T], H^{- \delta})}
     \leqslant R \]
  and that $v^{\varepsilon}$ solves the equation
  \[ (\partial^2_t +\mathbb{H}^{\omega, K}) v^{\varepsilon} - \sum_{i = 0}^3
     (\theta^{\varepsilon})^{\circ i} (v^{\varepsilon})^{3 - i} = 0. \qquad
     v^{\varepsilon} (0) = 0 \]
  and satisfies
  \[ \| v^{\varepsilon} \|_{L^{\infty} H^{1 - \delta}} \leqslant L \]
  Then $v^{\varepsilon}$ has a subsequence converging in $L^{\infty} H^{1 - 2
  \delta}$ and the limit $v$ solves
  \[ (\partial^2_t +\mathbb{H}^{\omega, K}) v - \sum_{i = 0}^3 \theta^{\circ
     i} v^{3 - i} = 0 \quad (v, \partial_t v) |_{t = 0} = 0 \]
  and satisfies $\| v \|_{L^{\infty} H^{1 - \delta}} \lesssim R L^3$.
\end{lemma}

\begin{proof}
  Note that to obtain a converging subsequence we only need to bound
  $v^{\varepsilon}$ in $\mathcal{C}^{\alpha}_t H^{1 - \alpha - \delta'}$ for
  small $\alpha, \delta' > 0.$ To this end we will apply Lemma
  \ref{bound-remainder}. Recall that $u (s) = \theta (s) + v (s)$ so
  \begin{eqnarray*}
    \| v^{\varepsilon} \|_{H^{1 - \delta}} & \leqslant & \sum_{i = 0}^3 \|
    (\theta^{\varepsilon})^{\circ i} (\rho_{\varepsilon} \ast
    v^{\varepsilon})^{3 - i} \|_{H^{- \delta}}\\
    & \leqslant & \sum_{i = 0}^3 \| \theta^{\circ i} \|_{H^{- \delta}} \| v
    \|^{3 - i}_{H^{1 - \delta}}\\
    & \leqslant & 3 R (1 + L)^3
  \end{eqnarray*}
  so applying Lemma \ref{bound-remainder} we get that
  \[ \| v^{\varepsilon} \|_{\mathcal{C}_t^{\alpha} H^{1 - \alpha - \delta'}}
     \lesssim 3 R (1 + L)^3 \]
  and the compactness claim follows. Lastly, that $v$ solves the limiting
  equation follows from Lemma \ref{bound-convergence}. 
\end{proof}

\begin{proposition}
  We have that
  \[ \lim_{R \rightarrow \infty} \nu (\| v \|_{L^{\infty} H^{1 - \delta}}
     \leqslant R) = 1. \]
\end{proposition}

\begin{proof}
  Denote by $\Sigma_R^{\varepsilon} = \{ \varphi : \| v^{\varepsilon}
  \|_{L^{\infty} H^{1 - \delta}} \leqslant R \}$ and $\Sigma_R = \{ \varphi :
  \| v \|_{L^{\infty} H^{1 - \delta}} \leqslant R \}$. From Lemma
  \ref{lemma:subsequence} we have that
  \[ \Sigma_R \supset \limsup_{\varepsilon \rightarrow 0}
     \Sigma_R^{\varepsilon} = \bigcap_{\varepsilon > 0} \bigcup_{\varepsilon'
     < \varepsilon} \Sigma_R^{\varepsilon} . \]
  Note that $\sup_{\varepsilon > 0} \sup_{0 \leqslant i \leqslant 3}  \| :
  (\rho_{\varepsilon} \ast w)^i : \|_{L^p ([0, T], H^{- \delta})} < \infty$
  $\nu, \nu^{\varepsilon}$-almost surely. Then by Fatou's Lemma
  \[ \nu (\Sigma_R) \geqslant \nu (\limsup_{\varepsilon \rightarrow 0}
     \Sigma_R^{\varepsilon}) \geqslant \limsup_{\varepsilon \rightarrow 0} \nu
     (\Sigma_R^{\varepsilon}) = \limsup_{\varepsilon \rightarrow 0}
     \nu^{\varepsilon} (\Sigma_R^{\varepsilon}) \]
  where the last equality is true since $\nu^{\varepsilon} \rightarrow \nu$ in
  total variation. Now an application of Proposition \ref{prop:bound-wick-u}
  yields the claim. 
\end{proof}

\begin{proposition}
  The measure $\nu$ is invariant under the flow $\Phi_t$, that is
  \[ \Phi_t^{\ast} \nu = \nu . \]
\end{proposition}

\begin{proof}
  This follows in exactly the same way as Proposition
  \ref{prop:invariance-approx}. 
\end{proof}

\appendix\section{Besov spaces and related concepts}\label{app:Besov}

We collect some basic definitions and elementary results about Besov spaces,
paraproducts etc., see, e.g., {\cite{allez_continuous_2015,BCD,GIP2015}} for
more details. We work here in the case of Besov spaces defined on the
$d$-dimensional torus \
\[ \mathbb{T}^d = (\mathbb{R}/\mathbb{Z})^d . \]
First we define the Sobolev space $H^{\alpha} (\mathbb{T}^d)$ with index
$\alpha \in \mathbb{R}$ which is the Banach space of distribution $u$ such
that $(1 - \Delta)^{- \frac{\alpha}{2}} (u)$ is a function and
\begin{equation}
  H^{\alpha} (\mathbb{T}^d) \assign \left\{ u \in \mathcal{S}' (\mathbb{T}^d)
  : \left\| (1 - \Delta)^{\frac{\alpha}{2}} u \right\|_{L^2} < \infty \right\}
  .
\end{equation}
Next, we recall the definition of Littlewood-Paley blocks. We denote by $\chi$
and $\rho$ two non-negative smooth and compactly supported radial functions
$\mathbb{R}^d \rightarrow \mathbb{C}$ such that
\begin{enumerateroman}
  \item The support of $\chi$ is contained in a ball and the support of $\rho$
  is contained in an annulus $\{ x \in \mathbb{R}^d : a \leqslant | x |
  \leqslant b \}$
  
  \item For all $\xi \in \mathbb{R}^d$, $\chi (\xi) + \underset{j \geqslant
  0}{\sum} \rho (2^{- j} \xi) = 1$;
  
  \item For $j \geqslant 1$, $\chi (\cdummy) \rho (2^{- j} \cdummy) = 0$ and
  $\rho (2^{- j} \cdummy) \rho (2^{- i} \cdummy) = 0$ for $| i - j | > 1$.
\end{enumerateroman}
The Littlewood-Paley blocks $(\Delta_j)_{j \geqslant - 1}$ associated to $f
\in \mathcal{S}' (\mathbb{T}^d)$ are defined by
\[ \Delta_{- 1} f \assign \mathcal{F}^{- 1} \chi \mathcal{F}f \text{and }
   \Delta_j f \assign \mathcal{F}^{- 1} \rho (2^{- j} \cdummy) \mathcal{F}f
   \text{for} j \geqslant 0, \]
and we define the Littlewood-Paley function $K_j = \mathcal{F}^{- 1}
(\Delta_j)$, i.e. the function for which $K_j \ast f = \Delta_j f$. We also
set, for $f \in \mathcal{S}' (\mathbb{T}^d)$ and $j \geqslant - 1$
\[ S_j f \assign \underset{i = - 1}{\overset{j - 1}{\sum}} \Delta_i f. \]
Then the Besov space with parameters $p \in [1, \infty], q \in [1, \infty),
\alpha \in \mathbb{R}$ can now be defined as \
\[ B_{p, q}^{\alpha} (\mathbb{T}^d) \assign \{ u \in \mathcal{S}'
   (\mathbb{T}^d) : \| u \|_{B_{p, q}^{\alpha}} < \infty \}, \]
where the norm is defined as
\begin{equation}
  \| u \|_{B_{p, q}^{\alpha}} \assign \left( \underset{k \geqslant - 1}{\sum}
  ((2^{\alpha k} \| \Delta_j u \|_{L^p})^q) \right)^{\frac{1}{q}},
  \label{def:Besov}
\end{equation}
with the obvious modification for $q = \infty .$ In the paper we often omit
the dependence of $B^{\alpha}_{p, q}$ from the torus $\mathbb{T}^d$. There are
two special cases of Besov spaces: the \tmtextit{Besov-H{\"o}lder} spaces for
$p = q = \infty$, i.e.
\begin{equation}
  \mathcal{C}^{\alpha} \assign B_{\infty \infty}^{\alpha} (\mathbb{T}^d)
\end{equation}
and the Sobolev spaces $H^{\alpha} = B_{2, 2}^{\alpha} (\mathbb{T}^d)$
(defined above) for $p = q = 2$.

\

Using this notation, we can formally decompose the product $f \cdummy g$ of
two distributions $f$ and $g$ as \
\[ f \cdummy g = f \prec g + f \circ g + f \succ g, \]
where
\[ f \prec g \assign \underset{j \geqslant - 1}{\sum} S_{j - 1} f \Delta_j g
   \quad \text{and\quad} f \succ g \assign \underset{j \geqslant - 1}{\sum}
   \Delta_j f S_{j - 1} g \]
are referred to as the \tmtextit{paraproducts}, whereas
\begin{equation}
  f \circ g \assign \underset{j \geqslant - 1}{\sum} \underset{| i - j |
  \leqslant 1}{\sum} \Delta_i f \Delta_j g \label{eq:resonantproduct}
\end{equation}
is called the \tmtextit{resonant product}. An important point is that the
paraproduct terms are always well defined whatever the regularity of $f$ and
$g$. The resonant product, on the other hand, is a priori only well defined if
the sum of their regularities is positive. We collect some results.

\begin{lemma}
  \label{lem:para}Let $\alpha, \alpha_1, \alpha_2 \in \mathbb{R}$ and $p, p_1,
  p_2, q \in \{ 2, \infty \}$ be such that
  \[ \alpha_1 \neq 0 \quad \alpha = (\alpha_1 \wedge 0) + \alpha_2, \quad
     \frac{1}{p} = \frac{1}{p_1} + \frac{1}{p_2}  \text{and}  \frac{1}{q} =
     \frac{1}{q_1} + \frac{1}{q_2} . \]
  Then we have the bound
  \[ \| f \prec g \|_{B_{p, q}^{\alpha}} \lesssim \| f \|_{B_{p_1,
     q_1}^{\alpha_1}} \| g \|_{{B^{\alpha_2}_{p_2, q_2}} } \]
  and in the case where $\alpha_1 + \alpha_2 > 0$ we have the bound
  \[ \| f \circ g \|_{B_{p, q}^{\alpha_1 + \alpha_2}} \lesssim \| f
     \|_{B_{p_1, q_1}^{\alpha_1}} \| g \|_{{B^{\alpha_2}_{p_2, q_2}} } . \]
\end{lemma}

\begin{proof}
  The proof can be found in {\cite{BCD}} Theorem 2.47 and Theorem 2.52 for
  Besov spaces on $\mathbb{R}^d$. The proof for Besov spaces on $\mathbb{T}^d$
  is similar.
\end{proof}

\begin{lemma}[Bernstein's inequalities]
  \label{lem:bernstein}Let $\mathcal{A}$ be an annulus and $\mathcal{B}$ be a
  ball. For any $k \in \mathbb{N}, \lambda > 0,$and $1 \leqslant p \leqslant q
  \leqslant \infty$ we have
  \begin{enumerate}
    \item if $u \in L^p (\mathbb{R}^d) $is such that $\tmop{supp}
    (\mathcal{F}u) \subset \lambda \mathcal{B}$ then
    \[ \underset{\mu \in \mathbb{N}^d : | \mu | = k}{\max} \| \partial^{\mu} u
       \|_{L^q} \lesssim_k \lambda^{k + d \left( \frac{1}{p} - \frac{1}{q}
       \right)} \| u \|_{L^p} \]
    \item if $u \in L^p (\mathbb{R}^d) $is such that $\tmop{supp}
    (\mathcal{F}u) \subset \lambda \mathcal{A}$ then
    \[ \lambda^k \| u \|_{L^p} \lesssim_k \underset{\mu \in \mathbb{N}^d : |
       \mu | = k}{\max} \| \partial^{\mu} u \|_{L^p} . \]
  \end{enumerate}
\end{lemma}

\begin{proof}
  The proof can be found in {\cite{BCD}} Lemma 2.1
\end{proof}

\begin{lemma}
  \label{lemma:symbols}Let $\sigma : \mathbb{Z}^d \rightarrow \mathbb{R}_+$
  such that
  \[ | \sigma (k) |^{\pm 1} \lesssim (| k | + 1)^{\pm \alpha} \]
  for some $\alpha \in \mathbb{R}$ then, for every $p, q \in [1, \infty]$ and
  $s \in \mathbb{R}$, operator $\sigma (\nabla)$ with symbol $\sigma$ is a
  linear homeomorphism from $B^s_{p, q}$ into $B^{s - \alpha}_{p, q}$. 
\end{lemma}

\begin{proof}
  The proof is an easy application of Lemma \ref{lem:bernstein} (see, e.g.
  {\cite{BCD}} Chapter 2 ). 
\end{proof}

\begin{remark}
  \label{remark:symbols}The hypotheses of Lemma \ref{lemma:symbols} are
  satisfied when $\sigma (k) = (| k |^2 + m^2)^{\alpha}$, for any $\alpha \in
  \mathbb{R}$ and $m > 0$, and thus $\sigma (\nabla) = (- \Delta +
  m^2)^{\alpha}$.
\end{remark}

\begin{lemma}[Besov embedding]
  \label{lem:besovem} Let $\alpha < \beta \in \mathbb{R}$ and $p > r \in [1,
  \infty]$ be such that
  \begin{equation}
    \beta \leqslant \alpha + d \left( \frac{1}{r} - \frac{1}{p} \right),
    \label{eq:Besovembedding}
  \end{equation}
  then we have the following bound
  \[ \| f \|_{B_{p, q}^{\alpha} (\mathbb{T}^d)} \lesssim \| f \|_{B_{r,
     q}^{\beta} (\mathbb{T}^d)} . \]
  If the inequality \eqref{eq:Besovembedding} is strict, the embedding is
  compact.
\end{lemma}

\begin{proof}
  The proof can be found in Proposition 2.71 of {\cite{BCD}}.
\end{proof}

\begin{proposition}[Commutator Lemma]
  \label{prop:commu}Given $\alpha \in (0, 1)$, $\beta, \gamma \in \mathbb{R}$
  such that $\beta + \gamma < 0$ and $\alpha + \beta + \gamma > 0$, the
  following trilinear operator $C$ defined for any smooth functions $f, g, h$
  by
  \[ C (f, g, h) \assign (f \prec g) \circ h - f (g \circ h) \]
  can be extended continuously to the product space $H^{\alpha} \times
  \mathcal{C}^{\beta} \times \mathcal{C}^{\gamma}$. Moreover, we have the
  following bound
  \[ || C (f, g, h) ||_{H^{\alpha + \beta + \gamma - \delta}} \lesssim || f
     ||_{H^{\alpha}} || g ||_{\mathcal{C}^{\beta}} || h
     ||_{\mathcal{C}^{\gamma}} \]
  for all $f \in H^{\alpha}$, $g \in \mathcal{C}^{\beta}$ and $h \in
  \mathcal{C}^{\gamma}$, and every $\delta > 0$.
  
  The analogous bound is true if we replace the Sobolev space by a
  H{\"o}lder-Besov space, i.e.
  \[ || C (f, g, h) ||_{\mathcal{C}^{\alpha + \beta + \gamma}} \lesssim || f
     ||_{\mathcal{C}^{\alpha}} || g ||_{\mathcal{C}^{\beta}} || h
     ||_{\mathcal{C}^{\gamma}}, \]
  as was shown in {\cite{GIP2015}}.
\end{proposition}

\begin{proof}
  The proof can be found in Proposition 4.3 of {\cite{allez_continuous_2015}}.
\end{proof}

\begin{lemma}[Fractional Leibniz rule]
  \label{lem:fracleib} Let $1 < p < \infty$ and $p_1, p_2, p'_1, p'_2$ such
  that
  \[ \frac{1}{p_1} + \frac{1}{p_2} = \frac{1}{p'_1} + \frac{1}{p'_2} =
     \frac{1}{p} . \]
  Then for any $s, \alpha \geqslant 0$ there exists a constant s.t.
  \[ \| \langle \nabla \rangle^s (f g) \|_{L^p} \leqslant C \| \langle \nabla
     \rangle^{s + \alpha} f \|_{L^{p_2}} \| \nabla^{- \alpha} g \|_{L^{p_1}} +
     C \| \langle \nabla \rangle^{- \alpha} f \|_{L^{p'_2}} \| \nabla^{s +
     \alpha} g \|_{L^{p'_1}} . \]
\end{lemma}

\begin{proof}
  The proof can be found in Theorem 1.4 of {\cite{gulisashvili1996exact}}.
\end{proof}

In the rest of the appendix, we discuss some properties of the operator $J_s$
introduced in Section \ref{section:coupling}, see Remark
\ref{notation:omegaprime}. Hereafter, we require a more explicit form of the
operator $J_s$. Let $\rho : \mathbb{R}^2 \rightarrow \mathbb{R}_+$ be a bump
function which is identically $1$ in the ball of center $0$ and radius $1 / 2$
and it has compact support in the ball of center $0$ and radius $1$. We
suppose also that $\rho$ is radially symmetric of the form $\rho (| x |^2)$.
We write, for $t \geqslant 0$, $\rho_t (x) = \rho \left( \frac{| x |^2}{t^2}
\right)$. It is clear that the operator
\begin{equation}
  J_t \assign \sigma_t (- \Delta) (- \Delta + m^2)^{- 1 / 2} \assign \left(
  \frac{\Delta}{t^3} \rho' \left( \frac{- \Delta}{t^2} \right) \right)^{1 / 2}
  (- \Delta + m^2)^{- 1 / 2} \label{eq:Js}
\end{equation}
satisfies the condition of Remark \ref{notation:omegaprime}.

\begin{proposition}
  \label{proposition:Js}Let $J_s$ be defined by equation \eqref{eq:Js} then
  for every $s \in \mathbb{R}_+$, $p, q \in [1, + \infty]$ and $0 < m
  \leqslant 1$
  \[ \| J_s \|_{\mathcal{L} (B^s_{p, q}, B^{s + m}_{p, q})} \lesssim (1 +
     s)^{- \frac{1}{2} - 1 + m} . \]
  where the constant in the symbol $\lesssim$ do not depend on $s \in
  \mathbb{R}_+$.
\end{proposition}

\begin{proof}
  This is a standard application of the regularization properties of the
  Laplacian (see Lemma \ref{lem:bernstein}, Lemma \ref{lemma:symbols} and
  Remark \ref{remark:symbols}).
\end{proof}

A consequence of Proposition \ref{proposition:Js} is the following:

\begin{proposition}
  \label{proposition:commutatorJs1}Suppose that $f_1 \in H^1$, $f_2, f_3 \in
  \mathcal{C}^{- 1 - \delta}$ (where $\delta < \frac{1}{6}$) then we have
  there is $0 < \eta (\delta) < 1$ such that
  \begin{eqnarray}
    &  & \left| \int_{\mathbb{T}^2} (J_s (f_1 \prec f_2) J_s (f_3 \prec f_4)
    - (J_s (f_2) \circ J_s (f_4)) f_1 f_3) \mathd x \right| \nonumber\\
    & \lesssim & (1 + s)^{- 1 - \eta (\delta)} \| f_1 \|_{H^{1 / 2 - \delta}}
    \| f_3 \|_{H^{1 / 2 - \delta}} \| f_2 \|_{\mathcal{C}^{- 1 - \delta}} \|
    f_3 \|_{\mathcal{C}^{- 1 - \delta}} ;  \label{eq:Nikolay1}
  \end{eqnarray}
  \begin{eqnarray}
    &  & \left| \int_{\mathbb{T}^2} (J_s (f_1 \prec f_2) J_s (f_3) - (J_s
    (f_2) \circ J_s (f_3)) f_1) \mathd x \right| \nonumber\\
    & \lesssim & (1 + s)^{- 2 + \eta (\delta)} \| f_1 \|_{H^1} \| f_2
    \|_{\mathcal{C}^{- 1 - \delta}} \| f_3 \|_{\mathcal{C}^{- 1 - \delta}} 
    \label{eq:Nikolay2}
  \end{eqnarray}
\end{proposition}

\begin{proof}
  Inequality \eqref{eq:Nikolay1} is proved in Proposition 11 of
  {\cite{BG2022}}, meanwhile \eqref{eq:Nikolay2} can be proved in a similar
  way with slight modifications.
\end{proof}

\section{An alternative proof of the coupling existence}\label{app:AGFFdiff}

We give a shorter, alternative (but less general) proof of the existence of a coupling between
the Gaussian measures $\mu^{\mathbb{H}^{\omega, K}}$ (i.e. the
\tmtextit{Anderson free field}) and the Gaussian free field $\mu^{- \Delta +
K}$ from Theorem \ref{theorem:coupling}, see Section \ref{sec:prel} for the
relevant definitions.

\begin{proposition}
  \label{prop:AGFFdiff}Consider the Gaussian measures
  $\mu^{\mathbb{H}^{\omega, K}}$and $\mu^{- \Delta + K}$ from Section
  \ref{sec:prel}. Then we may write elements in the support of these measures
  as
  \begin{eqnarray*}
    (\mathbb{H}^{\omega, K})^{- \frac{1}{2}} \psi & \in & \tmop{supp}
    (\mu^{\mathbb{H}^{\omega, K}})\\
    (- \Delta + K)^{- \frac{1}{2}} \psi & \in & \tmop{supp} (\mu^{- \Delta +
    K})\\
    \text{where} &  & \\
    \psi & \in & \tmop{supp} \left( \mu^{\mathbb{I}_{L^2}} \right),
  \end{eqnarray*}
  where $\mu^{\mathbb{I}_{L^2}}$ is the white noise measure defined in Section
  \ref{sec:Wick}.
  
  Then we have that for any $\delta > 0$
  \[ \int \left\| \left( (\mathbb{H}^{\omega, K})^{- \frac{1}{2}} - (- \Delta
     + K)^{- \frac{1}{2}} \right) \psi \right\|^2_{H^{1 - \delta}} d
     \mu^{\mathbb{I}_{L^2}} (\psi) < \infty . \]
  In other words, we can see the Anderson free field $\mu^{\mathbb{H}^{\omega,
  K}}$as a random shift (of regularity $H^{1 - \delta}$) of the GFF.
\end{proposition}

\begin{proof}
  This is somewhat similar to what was done in Sections 4 and 5 of
  {\cite{bailleul2022analysis}} but we adapt it to our notation and setting.
  
  We collect a few properties from previous sections. Firstly, from
  \eqref{eqn:Heigen}, we have an orthonormal eigenbasis of
  $\mathbb{H}^{\omega, K}$ (suppressing the $\omega$ for readability)
  \[ \mathbb{H}^{\omega, K} f_n = \lambda_n f_n, \ f_n \in \mathcal{D}
     (\mathbb{H}^{\omega .K}), \lambda_n \sim n \text{ as } n \rightarrow
     \infty . \]
  Thus we may write $\psi \in \tmop{supp} \left( \mu^{\mathbb{I}_{L^2}}
  \right) \subset \mathcal{C}^{- 1 - \varepsilon} \ \forall \varepsilon > 0$ as
  a random series
  \begin{equation}
    \psi = \underset{n \in \mathbb{N}}{\sum} g_n \underset{}{} f_n 
    \text{where the } g_n \sim \mathcal{N} (0, 1)  \text{ are i.i.d Gaussians}
    \label{eqn:wnsupp}
  \end{equation}
  then $(\mathbb{H}^{\omega, K})^{- \frac{1}{2}} \psi \in \tmop{supp}
  (\mu^{\mathbb{H}^{\omega, K}})$ and $(- \Delta + K)^{- \frac{1}{2}} \psi \in
  \tmop{supp} (\mu^{- \Delta + K})$. We further recall the following formula
  from functional calculus
  \begin{equation}
    (\mathbb{H}^{\omega, K})^{- \frac{1}{2}} = c \int^{\infty}_0 t^{-
    \frac{1}{2}} e^{- t\mathbb{H}^{\omega, K}} d t, \label{eqn:funccal}
  \end{equation}
  see Theorem 35 in {\cite{bailleul2022analysis}} and the discussion
  thereafter. Moreover, from the definition of $\Gamma$ in \eqref{def:Gamma}
  and its inverse in \eqref{def:Phi} together with the paraproduct bounds from
  Lemma \ref{lem:para}, one can readily check that the following regularizing
  property holds
  \begin{equation}
    (\Gamma^{\pm 1} - 1) : H^{- \sigma} \rightarrow H^{1 - \sigma -
    \varepsilon} \tmop{bounded } \forall \sigma, \varepsilon > 0.
    \label{Gamma-id}
  \end{equation}
  Moreover, we have using \eqref{def:Homega} and the paraproduct estimates
  \begin{equation}
    \| (\mathbb{H}^{\omega, K} \Gamma - (K - \Delta)) v \|_{H^{\kappa}}
    \lesssim \| v \|_{H^{1 + \kappa + \varepsilon}}\ \forall \kappa,
    \varepsilon > 0. \label{eqn:diffH}
  \end{equation}
  Lastly, we use that Theorem \ref{thm:AHprop} implies that
  \begin{equation}
    \text{$(\mathbb{H}^{\omega, K})^{- \frac{1}{2}} : H^{- \sigma} \rightarrow
    H^{1 - \sigma}$ for $\sigma \in (0, 1)$} . \label{H-1/2bound}
  \end{equation}
  We make a computation where we start with \eqref{eqn:funccal} and
  artificially insert the $\Gamma$ in order to use the fact that
  $\mathbb{H}^{\omega, K} \Gamma$ is a lower order perturbation of the
  Laplacian. In order to isolate the problematic term, we adopt the brief
  notation $\mathcal{O} (1 -)$ to mean a term which is bounded from
  $\mathcal{C}^{- 1 - \varepsilon} \rightarrow H^{1 - \varepsilon - \kappa}$
  for $\varepsilon, \kappa > 0$ i.e. for which one does not use the Gaussian
  property.
  
  We compute
  \begin{eqnarray*}
    (\mathbb{H}^{\omega, K})^{- \frac{1}{2}} & \overset{\eqref{Gamma-id},
    \eqref{H-1/2bound}}{=} & (\mathbb{H}^{\omega, K})^{- \frac{1}{2}} \Gamma +
    \underset{\mathcal{O} (1 -)}{\underbrace{(\mathbb{H}^{\omega, K})^{-
    \frac{1}{2}} (1 - \Gamma)}}\\
    & \overset{\text{\eqref{eqn:funccal}}}{=} & c \int^{\infty}_0 t^{-
    \frac{1}{2}} e^{- t\mathbb{H}^{\omega, K}} \Gamma  d t +\mathcal{O} (1
    -)\\
    & = & c \int^{\infty}_0 t^{- \frac{1}{2}} e^{- t (K - \Delta)} d t \Gamma
    + c \int^{\infty}_0 t^{- \frac{1}{2}} (e^{- t\mathbb{H}^{\omega, K}}
    \Gamma - e^{- t (K - \Delta)}) d t +\mathcal{O} (1 -)\\
    & = & c \int^{\infty}_0 t^{- \frac{1}{2}} e^{- t (K - \Delta)} d t +
    \underset{= (K - \Delta)^{- \frac{1}{2}} (\Gamma  - 1) =\mathcal{O} (1
    -)}{\underbrace{c \int^{\infty}_0 t^{- \frac{1}{2}} e^{- t (K - \Delta)} d
    t (\Gamma - 1)}} +\\
    &  & + c \int^{\infty}_0 t^{- \frac{1}{2}} (e^{- t\mathbb{H}^{\omega, K}}
    \Gamma - e^{- t (K - \Delta)}) d t +\mathcal{O} (1 -)\\
    & = & (K - \Delta)^{- \frac{1}{2}} + c \int^{\infty}_0 t^{- \frac{1}{2}}
    (e^{- t\mathbb{H}^{\omega, K}} \Gamma - e^{- t (K - \Delta)}) \Gamma^{- 1}
    d t +\mathcal{O} (1 -) .
  \end{eqnarray*}
  Thus it remains to control the integral term for elements in the support of
  $\mu^{\mathbb{I}_{L^2}}$ of the form \eqref{eqn:wnsupp}.
  
  We begin by rewriting it as follows
  \begin{eqnarray*}
    (e^{- t\mathbb{H}^{\omega, K}} \Gamma - e^{- t (K - \Delta)}) \Gamma^{- 1}
    \underset{n}{\sum} g_n f_n & = & \underset{n}{\sum} g_n \int^t_0 e^{- (t -
    s) (K - \Delta)} (\mathbb{H}^{\omega, K} \Gamma - (K - \Delta)) e^{-
    s\mathbb{H}^{\omega, K}} \Gamma \Gamma^{- 1} f_n\\
    & = & \underset{n}{\sum} g_n \int^t_0 e^{- (t - s) (K - \Delta)}
    (\mathbb{H}^{\omega, K} \Gamma - (K - \Delta)) e^{- s\mathbb{H}^{\omega,
    K}} f_n\\
    & = & \underset{n}{\sum} g_n \int^t_0 e^{- s \lambda_n} e^{- (t - s) (K -
    \Delta)} (\mathbb{H}^{\omega, K} \Gamma - (K - \Delta)) f_n
  \end{eqnarray*}
  now we may compute the $L_{\mathbb{P}}^2 H^{\alpha}$ $\alpha < 1$ norm.
  Recall that $H^{\sigma} =\mathcal{D} \left( (\mathbb{H}^{\omega,
  K})^{\frac{\sigma}{2}} \right)$ for $\sigma \in (- 1, 1)$) with equivalent
  norms, see Theorem \ref{thm:AHprop}. For a sequence of functions $h_n$, we
  can thus bound, using the independence of the Gaussians and the
  orthogonality of the eigenfunctions,
  \begin{eqnarray*}
    \left\| \underset{n}{\sum} g_n h_n \right\|^2_{L_{\mathbb{P}}^2
    H^{\alpha}} & \approx & \mathbb{E} \left( \underset{j}{\sum} \left(
    \underset{n}{\sum} g_n h_n, \sqrt{H}^{\alpha} f_j \right)^2 \right)\\
    & = & \mathbb{E} \left( \underset{j}{\sum} \left( \underset{n}{\sum} g_n
    h_n, \sqrt{H}^{\alpha} f_j \right) \left( \underset{n}{\sum} g_m h_m,
    \sqrt{H}^{\alpha} f_j \right) \right)\\
    & = & \underset{j}{\sum} \underset{n, m}{\sum} \underset{= \delta_{n,
    m}}{\underbrace{\mathbb{E} (g_n, g_m)}} \left( h_n, \sqrt{H}^{\alpha} e_j
    \right) \left( h_m, \sqrt{H}^{\alpha} e_j \right)\\
    & = & \underset{j}{\sum} \underset{n}{\sum} \left( h_n, \sqrt{H}^{\alpha}
    e_j \right)^2\\
    & = & \underset{n}{\sum} \left( \sqrt{H}^{\alpha} h_n, \sqrt{H}^{\alpha}
    h_n \right)\\
    & = & \underset{n}{\sum} \| h_n \|^2_{H^{\alpha}} .
  \end{eqnarray*}
  Now if we apply this to $h_n = \int^{\infty}_0 t^{- \frac{1}{2}} \int^t_0
  e^{- s \lambda_n} e^{- (t - s) (K - \Delta)} (\mathbb{H}^{\omega, K} \Gamma
  - (K - \Delta)) f_n d s d t$ and we bound its $H^{\alpha}$ norm for $\alpha
  < 1$ norm and show its square summability this implies the desired result.
  
  We compute for some parameters $0 < \delta, \kappa \ll 1$ to be fixed later and $\bar{\lambda}_n\assign\lambda_n-K$
  \begin{eqnarray*}
    \| h_n \|_{H^{\alpha}} & \lesssim & \int^{\infty}_0 t^{- \frac{1}{2}}
    \int^t_0 e^{- s \lambda_n} \| e^{- (t - s) (K - \Delta)}
    (\mathbb{H}^{\omega, K} \Gamma - (K - \Delta)) f_n \|_{H^{\alpha}} d s d
    t\\
    & \overset{\eqref{heatkernel}}{\lesssim} & \int^{\infty}_0 t^{-
    \frac{1}{2}} \int^t_0 e^{- s \lambda_n} | t - s |^{- \frac{1}{2} (\alpha -
    \delta)} e^{- (t - s) K} \| (\mathbb{H}^{\omega, K} \Gamma - (K - \Delta))
    f_n \|_{H^{\delta}} d s d t\\
    & \overset{\eqref{Gamma-id}, \eqref{eqn:diffH}}{\lesssim} &
    \int^{\infty}_0 t^{- \frac{1}{2}} e^{- K t} \int^t_0 e^{- s
    \bar{\lambda}_n| t - s |^{- \frac{1}{2} (\alpha - \delta)}} (\| (\mathbb{H}^{\omega, K}
    \Gamma - (K - \Delta)) \Gamma^{- 1} f_n \|_{H^{\delta}} + \| f_n
    \|_{H^{\delta + \kappa}}) d s d t\\
    & \overset{\eqref{bound:eigenfct}}{\lesssim} & \int^{\infty}_0 t^{-
    \frac{1}{2}} e^{- K t} \int^t_0 e^{- s \overline{\lambda}_n} | t - s |^{-
    \frac{1}{2} (\alpha - \delta)} \underset{\lesssim \lambda^{\frac{1}{2} +
    \frac{\delta}{2} + \frac{\kappa}{2}}_n}{\underbrace{\| \Gamma^{- 1} f_n
    \|_{H^{1 + \delta + \kappa}}}} d s d t\\
    & \underset{s \overline{\lambda}_n \backassign \sigma}{\overset{t
    \overline{\lambda}_n \backassign \tau}{\lesssim}} & \overline{\lambda}^{-
    2}_n \lambda^{\frac{1}{2} + \frac{\delta}{2} + \frac{\kappa}{2}}_n
    \int^{\infty}_0 \overline{\lambda}^{\frac{1}{2}}_n \tau^{- \frac{1}{2}}
    e^{- \frac{\tau}{\overline{\lambda}_n}} \int^{\tau}_0 e^{- \sigma} | \tau
    - \sigma |^{- \frac{1}{2} (\alpha - \delta)} \lambda^{\frac{1}{2} (\alpha
    - \delta)}_n d \sigma d \tau\\
    & \sim & \overline{\lambda}^{- 1 + \frac{\alpha}{2} + \frac{\kappa}{2}}_n
    \int^{\infty}_0 \tau^{- \frac{1}{2}} e^{-
    \frac{\tau}{\overline{\lambda}_n}} \int^{\tau}_0 e^{- \sigma} | \tau -
    \sigma |^{- \frac{1}{2} (\alpha - \delta)}\\
    & \overset{\eqref{3}}{\lesssim} & \lambda^{- \frac{1}{2} (1 + \kappa)}_n 
    \text{choosing} \alpha = 1 - 2 \kappa \text{ hence}\\
    \underset{n}{\sum} \| h_n \|^2_{H^{\alpha}} & \lesssim &
    \underset{n}{\sum} \lambda^{- 1 - \kappa}_n \lesssim \infty
  \end{eqnarray*}
  having used the Weyl asymptotic $\lambda_n \sim n$ for large $n,$ the bound
  \begin{equation}
    \| \Gamma^{- 1} f_n \|_{H^{1 + \sigma}} \sim \left\| (\mathbb{H}^{\omega,
    K})^{\frac{1 + \sigma}{2}} f_n \right\|_{L^2} \sim \lambda^{\frac{1 +
    \sigma}{2}}_n \label{bound:eigenfct}
  \end{equation}
  which is true for $\sigma \in \{ 0, 1 \}$ by Theorem \ref{thm:AHprop} and
  for $\sigma \in (0, 1)$ by interpolation, and the standard heat kernel bound
  \begin{equation}
    \| e^{t \Delta} g \|_{H^{\beta}} \lesssim t^{- \frac{\beta - \gamma}{2}}
    \| g \|_{H^{\gamma}}  \text{for all } t > 0 \text{ and } \beta > \gamma .
    \label{heatkernel}
  \end{equation}
  Lastly we prove the finiteness of the integral that we used to conclude. We
  compute
  \begin{eqnarray}
    \int^{\infty}_0 \tau^{- \frac{1}{2}} e^{-
    \frac{\tau}{\overline{\lambda}_n}} \int^{\tau}_0 e^{- \sigma} | \tau -
    \sigma |^{- \frac{1}{2} (\alpha - \delta)} & \overset{\rho =
    \frac{\sigma}{\tau}}{=} & \int^{\infty}_0 \tau^{- \frac{1}{2}} e^{-
    \frac{\tau}{\overline{\lambda}_n}} \int^1_0 \tau e^{- \rho \tau} \tau^{-
    \frac{1}{2} (\alpha - \delta)} (1 - \rho)^{- \frac{1}{2} (\alpha -
    \delta)}  \label{3}\\
    & = & \int^{\infty}_0 \tau^{\frac{1}{2} (1 - \alpha + \delta)} e^{-
    \frac{\tau}{\overline{\lambda}_n}} \int^1_0 e^{- \rho \tau} (1 - \rho)^{-
    \frac{1}{2} (\alpha - \delta)} \nonumber\\
    & \overset{\nu = \rho \tau}{=} & \int^{\infty}_0 {\nu^{\frac{1}{2} (1 -
    \alpha + \delta)}}  e^{- \nu} \int^1_0 \underset{\leqslant
    1}{\underbrace{e^{- \frac{\nu}{\rho \overline{\lambda}_n}}}}
    \frac{1}{\rho^{\frac{1}{2} (1 - \alpha + \delta)}} (1 - \rho)^{-
    \frac{1}{2} (\alpha - \delta)} \nonumber\\
    & \lesssim & 1 \nonumber
  \end{eqnarray}
  having used that
  \[ \int^1_0 \frac{1}{x^a} \frac{1}{(1 - x)^b} d x < \infty \text{ if } a, b
     < 1. \]
  This finishes the proof.
\end{proof}

\bibliographystyle{plain}
\bibliography{paracontrolled-wave3}

\end{document}